\documentclass[a4paper,10pt]{article}

\usepackage{amssymb,amsmath,amscd,amsthm,amssymb,amsxtra,latexsym,epsf,mathrsfs,latexsym}
\usepackage[all]{xy}
\usepackage{hyperref}
\newfont{\Bb}{msbm10}
\newfont{\Bc}{msbm8}

\DeclareMathOperator{\CCC}{\mbox{\Bc C}}
\DeclareMathOperator{\OO}{\cal O}

\DeclareMathOperator{\Der}{Der}
\newcommand{\QQ}{\mathbb{Q}}
\newcommand{\KK}{\mathbb{K}}
\newcommand{\RR}{\mathbb{R}}
\newcommand{\CC}{\mathbb{C}}
\newcommand{\GG}{\mathbb{G}}
\newcommand{\NN}{\mathbb{N}}
\newcommand{\PP}{\mathbb{P}}
\newcommand{\ZZ}{\mathbb{Z}}

\newcommand\s{\mathscr}

\def\Tk{T^{(k)}}
\def\xk{x^{(k)}}
\def\xo{x^{(0)}}
\def\11{\mathbf{1}}
\def\cc{\circ}

\def\ssm{\smallsetminus}

\def\ben{\begin{enumerate}}
\def\bin{\begin{itemize}}
\def\ein{\end{itemize}}
\def\een{\end{enumerate}}
\def\beq{\begin{equation}}
\def\eeq{\end{equation}}
\def\bit{\begin{itemize}}
\def\eit{\end{itemize}}
\def\ec{\end{center}}
\def\bc{\begin{center}}
\def\ld{{\ldots}}
\def\cd{{\cdots}}

\def\d{\mathbf{d}}

\def\Gl{\text{Gl}}

\def\gl{\mathfrak{gl}}
\def\gg{\mathfrak{g}}
\def\mm{\mathfrak{m}}

\def\Sl{\text{Sl}}

\def\GD{G^0_D}
\def\ve{{\varepsilon}}
\def\pp{{\varphi}}
\def\a{{\alpha}}

\def\G{{\Gamma}}
\def\o{{\omega}}
\def\O{{\Omega}}

\def\eop{\hspace*{\fill}$\Box$ \vskip \baselineskip} %This puts and box (i.e. QED) at the
\def\bu{\bullet}
\def\tr{{\pitchfork}}

\def\w{{\wedge}}
\def\p{{\partial}}

\DeclareMathOperator{\supp}{supp}

\DeclareMathOperator{\Rep}{Rep}
\def\vol{\text{vol}}

\DeclareMathOperator{\Sym}{Sym}

\DeclareMathOperator{\Hom}{Hom}
\DeclareMathOperator{\End}{End}

\def\Tr{T^1_{\s R_h/\CC}}

\newcommand{\diag}{\textup{diag}}

\def\iff{{\Leftrightarrow}}

\def\to{{\rightarrow\,}}

\newcommand{\Id}{\text{Id}}

\newtheorem{theorem}{Theorem}[section]
\newtheorem{lemma}[theorem]{Lemma}

\newtheorem{lemmadefinition}[theorem]{Lemma and Definition}
\newtheorem{proposition}[theorem]{Proposition}
\numberwithin{equation}{section}
\newtheorem{corollary}[theorem]{Corollary}

\theoremstyle{definition}
\newtheorem{definition}[theorem]{Definition}
\newtheorem{conjecture}[theorem]{Conjecture}

\theoremstyle{remark}
\newtheorem{algorithm}{Algorithm}
\newtheorem{example}[theorem]{Example}
\newtheorem{question}[theorem]{Question}
\newtheorem{remark}[theorem]{Remark}

\begin{document}

\title{Linear free divisors and Frobenius manifolds}
\author{Ignacio de Gregorio, David Mond and Christian Sevenheck}
\date{December 17, 2008}

\maketitle

\begin{abstract}
We study linear functions on fibrations whose central fibre is a linear free divisor.
We analyse the Gau\ss-Manin system associated to these functions,
and prove the existence of a primitive and homogenous form. As a consequence, we show
that the base space of the semi-universal unfolding of such a function
carries a Frobenius manifold structure.
\end{abstract}

\tableofcontents

\renewcommand{\thefootnote}{}
\footnote{\noindent 2000 \emph{Mathematics Subject Classification.}
Primary: 16G20, 32S40. Secondary: 14B07, 58K60\\
Keywords: Linear free divisors, prehomogenous vector spaces, quiver
representations, Gau\ss-Manin-system, Brieskorn lattice, Birkhoff
problem, spectral numbers, Frobenius manifolds.}

\section{Introduction}

In this paper we study Frobenius manifolds arising as deformation
spaces of linear
functions on certain non-isolated singularities, the so-called
linear free divisors. It is
a nowadays classical result that the semi-universal unfolding space of an isolated
hypersurface singularity can be equipped with a Frobenius structure. One of the main
motivations to study Frobenius manifolds comes from the fact that they also arise in a
very different area: the total cohomology space of a projective manifold carries such a
structure, defined by the quantum multiplication. Mirror symmetry postulates an
equivalence between these two types of Frobenius structures. In order to carry this
program out, one is forced to study not only local singularities (which are in fact never
the mirror of a quantum cohomology ring) but polynomial functions on affine manifolds.  It
has been shown in \cite{DS} (and later, with a somewhat different strategy in \cite{Dou})
that given a convenient and non-degenerate Laurent polynomial
$\widetilde{f}:(\CC^*)^n\rightarrow \CC$, the base space $M$ of a semi-universal unfolding
$\widetilde{F}:(\CC^*)^n \times M \rightarrow \CC$ can be equipped with a (canonical)
Frobenius structure.  An  important example is the function
$\widetilde{f}=x_1+\ldots+x_{n-1}+\frac{t}{x_1\cdot\ldots\cdot x_{n-1}}$ for some fixed $t\in \CC^*$: the Frobenius
structure obtained on its unfolding space is known (see \cite{Giv6}, \cite{Giv7} and \cite{Ba}) to be
isomorphic to the full quantum cohomology of the projective space $\PP^{n-1}$.
More generally, one can consider the Laurent polynomial
$\widetilde{f}=x_1+\ldots+x_{n-1}+\frac{t}{x_1^{w_1}\cdot\ldots\cdot x_{n-1}^{w_{n-1}}}$ for some
weights $(w_1,\ldots,w_{n-1})\in\NN^{n-1}$, here the Frobenius structure corresponds to
the (orbifold) quantum cohomology of the weighted projective space $\PP(1,w_1,\ldots,w_{n-1})$
(see \cite{Mann} and \cite{CCLH}). A detailed
analysis on how to construct the Frobenius structure for these functions is given
in \cite{DS2}; some of the techniques in this paper are similar to those used here.
Notice that the mirror of the ordinary projective space can be interpreted in a slightly different way, namely, as the restriction of
the linear polynomial $f=x_1+\ldots+x_n:\CC^n\rightarrow \CC$ to the non-singular
fibre $h(x_1,\ldots, x_n)-t=0$ of the torus fibration defined by the homogeneous
polynomial $h=x_1\cdot\ldots\cdot x_n$.

In the present paper, we construct Frobenius structures on the unfolding spaces of a class
of functions generalizing this basic example, namely, we consider homogenous functions $h$
whose zero fibre is a \emph{linear free divisor}.  Linear free divisors were
recently introduced by R.-O.~Buchweitz and Mond in \cite{bm} (see also
\cite{gmns}), but are closely related to the more classical {\em prehomogeneous spaces} of
T.~Kimura and M.~Sato (\cite{sk}). They are defined as free divisors $D=h^{-1}(0)$ in some
vector space $V$ whose sheaf of derivations can be generated by vector fields having only
linear coefficients. The classical example is of course the normal crossing divisor.
Following the analogy with the mirror of $\PP^{n-1}$, we are interested in characterising when
there exist linear functions $f$ having only
isolated singularities on the Milnor fibre $D_t=h^{-1}(t)$, $t\neq 0$.  As
it turns out, not all linear free divisors support such functions, but the large class of
{\em reductive} ones do, and for these the set of linear functions having only isolated
singularities can be characterised as the complement of the {\em dual divisor}.

Let us give a short overview on the paper.
In section~\ref{sec:ReductiveLFD} we state and prove some general results on linear free divisors.
In particular, we introduce the notion of \emph{special linear free divisors}, and show that
reductive ones are always special. This is proved by studying the relative logarithmic
de Rham complex (subsection~\ref{subsec:RelLogdeRham}) which is also important in the later discussion
of the Gau\ss-Manin-system. The cohomology of this complex is computed in
the reductive case, thanks to a classical theorem of Hochschild and Serre.

Section~\ref{sec:funcOnLFD} discusses linear functions $f$ on linear
free divisors $D$, as well
as on their Milnor fibres $D_t$. We show  (in an even more general situation
where $D$ is not a linear free divisor) that $f_{|D_t}$ is a Morse function
if the restriction $f_{|D}$ is right-left
stable. This implies in particular that the Frobenius structures associated to the
functions $f_{|D_t}$ are all semi-simple. Subsection \ref{subsec:REquivalences} discusses
deformation problems associated to the two functions $(f,h)$. In particular, we show that
linear forms in the complement of the dual divisor have the necessary finiteness properties.
In order that we can construct Frobenius structures, the
fibration defined by $f_{|D_t}$ is required
to have good behaviour at infinity,
comprised in the notion of {\em tameness}. In
subsection \ref{subsec:tame} it is shown that this property indeed holds
for these functions.

In section \ref{sec:GMBrieskorn} we study the (algebraic)
Gau\ss-Manin system and the (algebraic) Brieskorn lattice of
$f_{|D_t}$. We actually define both as families over the parameter
space of $h$, and using logarithmic forms along $D$ (more precisely,
the relative logarithmic de Rham complex mentioned above) we get
very specific extensions of these families over $D$.
The fact that $D$ is a linear free divisor allows us to
construct explicitly a basis of this family of Brieskorn lattice, hence showing
its freeness. Next we give a solution to the so-called Birkhoff problem.
Although this solution is not a good basis in the sense of M.~Saito \cite{SM}, that is, it might not
compute the spectrum at infinity of $f_{|D_t}$, we give an algorithmic procedure to turn
it into one. This allows us in particular to compute the monodromy of $f_{|D_t}$.
We finish this section by showing that this solution to the Birkhoff
problem is also compatible with a natural pairing defined on the
Brieskorn lattice, at least under an additional hypothesis (which is
satisfied in many examples) on the spectral numbers.

In section \ref{sec:Frobenius} we finally apply all these results to
construct Frobenius structures on the unfolding spaces of the
functions $f_{|D_t}$ (subsection \ref{subsec:Frobenius-tneq0}) and on
$f_{|D}$ (subsection \ref{subsec:Frobenius-teq0}). Whereas the former
exists in all cases, the latter depends on a conjecture concerning
a natural pairing on the Gauss-Manin-system. Similarly, assuming this
conjecture, we give some partial results concerning \emph{logarithmic
Frobenius structures} as defined in \cite{Reich1} in subsection \ref{subsec:LogFrobenius}.

We end the paper with some examples (section \ref{sec:examples}). On the one
hand, they illustrate the different phenomena that can occur, as for instance, the fact that there might not be a
\emph{canonical} choice (as in \cite{DS}) of a primitive form. On the
other hand, they support the conjecture concerning the pairing used in
the discussion of the Frobenius structure associated to $f_{|D}$.

\bigskip
\noindent{\bf Acknowledgements.}  We are grateful to Antoine Douai, Michel Granger,
Claus Hertling, Dmitry Rumynin,
Claude Sabbah and Mathias Schulze for helpful discussions on the subject of this article.
We thank the anonymous referee for a careful reading of a first version and a number of very helpful
remarks.

\section{Reductive and special linear free divisors}
\label{sec:ReductiveLFD}

\subsection{Definition and examples}
\label{subsec:RedLFDDefEx}

A hypersurface $D$ in a complex manifold $X$ is a {\it free divisor} if the $\OO_X$-module
$\Der(-\log D)$ is locally free. If $X=\CC^n$ then $D$ is furthermore a {\it linear free
  divisor} if $\Der(-\log D)$ has an $\OO_{\CC^n}$-basis consisting of weight-zero vector
fields -- vector fields whose coefficients, with respect to a standard linear coordinate
system, are linear functions (see \cite[Section 1]{gmns}). By Serre's conjecture, if
$D\subset\CC^n$ is a free divisor then $\Der(-\log D)$ is globally free.  If
$D\subset\CC^n$ is a linear free divisor then the group $G_D:=\{A\in\Gl_n(\CC):AD=D\}$ of
its linear automorphisms is algebraic of dimension $n$.  We denote by $G_D^0$ the
connected component of $G_D$ containing the identity, and by $\Sl_D$ the intersection of
$\GD$ with $\Sl_n(\CC)$. The infinitesimal action of the Lie algebra $\gg_D$ of $G_D^0$
generates $\Der(-\log D)$ over $\OO_{\CC^n}$, and it follows that the complement of $D$ is
a single $G_D^0$-orbit (\cite[Section 2]{gmns}).
Thus, $\CC^n$, with this action of $G^0_D$, is  a
\emph{prehomogeneous vector space} (\cite{sk})  i.e., a representation
$\rho$ of a group $G$ on a vector space $V$ in which the group has an
open orbit. The complement of the open orbit in a prehomogeneous vector space
is known as the {\it discriminant}. The (reduced)
discriminant in a prehomogeneous vector space is a linear free divisor if and
only if the dimensions
of $G$ and $V$ and the degree of the discriminant are all equal.

By Saito's criterion
(\cite{saito}), the determinant of the matrix of coefficients of a set of generators of
$\Der(-\log D)$ is a reduced equation for $D$, which is therefore homogeneous of degree
$n$. Throughout the paper we will denote the reduced homogeneous equation of the linear
free divisor $D$ by $h$.

If the group $G$ acts on the vector space $V$, then a rational function $f\in\CC(V)$ is a
{\it semi-invariant} (or {\it relative invariant}) if there is a character
$\chi_f:G\to\CC^*$ such that for all $g\in G$, $f\cc g=\chi_f(g)f$. In this case $\chi_f$
is the {\it character associated to $f$}. Sato and Kimura prove (\cite[\S 4 Lemma 4]{sk})
that semi-invariants with multiplicatively independent associated characters are
algebraically independent.  If $D$ is a linear free divisor with equation $h$, then $h$ is
a semi-invariant (\cite[\S 4]{sk}) (for the action of $\GD$).  For it is clear that $g$
must leave $D$ invariant, and thus $h\cc g$ is some complex multiple of $h$. This multiple
is easily seen to define a character, which we call $\chi_h$.
\begin{definition}\label{mond:specdef}
  We call the linear free divisor $D$ {\em special} if $\chi_h$ is equal to the
determinant of the representation, and {\it reductive} if the group $\GD$ is reductive.
\end{definition}
We show in \ref{cor2} below that every reductive linear free divisor is special. We do
not know if the converse holds.
The term ``special'' is used here because the condition means
that the elements of $G_D$ which fix $h$ lie in $\Sl_n(\CC)$.

Not all linear free divisors are special.  Consider the example of
the group $B_k$ of upper triangular complex matrices acting on the space $V=\Sym_k(\CC)$
of symmetric $k\times k$ matrices by transpose conjugation,
\begin{equation}
  \label{symmatrices}
  B\cdot S={^tB}SB
\end{equation}
The discriminant here is a linear free divisor (\cite[Example 5.1]{gmns}).  Its equation
is the product of the determinants of the top left-hand $l\times l$ submatrices of the
generic $k\times k$ symmetric matrix, for $l=1,\ld, k$. It follows that if
$B=\mbox{diag}(\lambda_1,\ld,\lambda_k)\in B_k$ then
\begin{displaymath}
  h\cc\rho(B)=\lambda_1^{2k}\lambda_2^{2k-2}\cd \lambda_k^2\,h,
\end{displaymath}
and $D$ is not special. The simplest example is the case $k=2$, here the divisor
has the equation
\begin{equation}\label{eqNonReductiveDim3}
h=x(xz-y^2)
\end{equation}

Irreducible prehomogeneous vector spaces are classified in \cite{sk}. However, irreducible
representations account for very few of the linear free divisors known.
For more examples we turn to the representation spaces of quivers:
\begin{proposition}[\cite{bm}]
  \label{dy}
  \begin{enumerate}
  \item Let $Q$ be a quiver without oriented loops and let $\d$ be (a dimension vector
    which is) a real Schur root of $Q$. Then the triple $(\Gl_{Q,\d},\rho, \Rep(Q,\d))$ is
    a prehomogeneous vector space and the complement of the open orbit is a divisor $D$
    (the ``discriminant'' of the representation $\rho$ of the quiver group $Gl_{Q,\d}$ on
    the representation space $\Rep(Q,\d)$)).
  \item If in each irreducible component of $D$ there is an open orbit, then $D$ is a
    linear free divisor.
  \item If $Q$ is a Dynkin quiver then the condition of (ii) holds for all real Schur
    roots $\d$.
  \end{enumerate}
\end{proposition}
We note that the normal crossing divisor appears as the discriminant in the representation
space $\Rep(Q,\11)$ for every quiver $Q$ whose underlying graph is a tree. Here {\bf 1} is
the dimension vector which takes the value $1$ at every node.

All of the linear free divisors constructed in Proposition 2.2 are reductive.
For if $D$ is the discriminant in $\Rep(Q,\d)$ then
$\GD$ is the quotient of $\Gl_{Q,\d}=\prod_i\Gl_{d_i}(\CC)$
by a 1-dimensional central subgroup.
\begin{example}\label{D_4/E_6}
\begin{enumerate}
\item
Consider
the quiver of type $D_4$ with real Schur root
$$
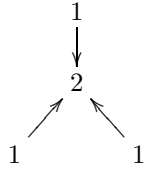
\xymatrix@R=0.2in@C=0.17in{&1\ar[d]\\&2\\
1\ar[ur]&&1\ar[ul]}$$
By choosing a basis for each vector space
we can identify the representation space $\Rep(Q,\d)$ with the space of $2\times 3$ matrices,
with each of the three morphisms corresponding to a column. The open orbit
in $\Rep(Q,\d)$
consists of matrices whose columns are pairwise linearly independent. The discriminant
thus
has equation
\begin{equation}\label{eqStar3}
h=(a_{11}a_{22}-a_{12}a_{21})
(a_{11}a_{23}-a_{13}a_{21})
(a_{12}a_{23}-a_{22}a_{13}).
\end{equation}
This example generalises: instead of three arrows converging to the central node,
we take $m$, and
set the dimension of the space at the central node to $m-1$. The representation
space can now be identified with the space of $(m-1)\times m$ matrices,
and the discriminant
is once again
defined by the vanishing of the product of maximal minors. Again it is a linear
free divisor (\cite[Example 5.3]{gmns}), even though for $m>3$ the quiver is no longer
a Dynkin quiver. We refer to it as the {\it star quiver}, and denote it by
$\star_m$.

\item
The linear free divisor arising by the construction
of Proposition \ref{dy} from the quiver of type $E_6$ with real Schur root
$$
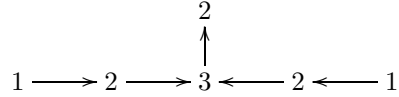
\xymatrix@R=0.2in{&&2\\1\ar[r]&2\ar[r]&3\ar[u]&2\ar[l]&1\ar[l]}
$$
has five irreducible components. In
the $22$-dimensional representation space $\Rep(Q,\d)$,  we take coordinates $a,b,\ld, v$.
Then
\begin{equation}\label{eqE6}
h=F_1\cdot F_2\cdot F_3\cdot F_4\cdot F_5
\end{equation}
where four of the components have the equations
$$
\begin{array}{|ccl|}
\hline
F_1&=&dfpq-cgpq-dfor+cgor+efps-chps+egrs-dhrs-efot+chot-egqt+dhqt\\
F_2&=&jlpq-impq-jlor+imor+klps-inps+kmrs-jnrs-klot+inot-kmqt+jnqt\\
F_3&=&-aejl-bhjl+adkl+bgkl+aeim+bhim-ackm-bfkm-adin-bgin+acjn+bfjn\\
F_4&=&egiu-dhiu-efju+chju+dfku-cgku+eglv-dhlv-efmv+chmv+dfnv-cgnv
\\
\hline
\end{array}
$$
and the fifth has the equation $F_5=0$, which is of degree 6, with 48 monomials. This example is
discussed in detail in \cite[Example 7.3]{bm}
\end{enumerate}
\end{example}

\subsection{The relative logarithmic de Rham complex}
\label{subsec:RelLogdeRham}

Let $D$ be a linear free divisor with equation $h$.  We set $\Der(-\log
h)=\{\chi\in\Der(-\log D):\chi\cdot h=0\}$. Under the infinitesimal action
of $\GD$, the Lie algebra of $\ker(\chi_h)$, which we denote by $\gg_h$, is identified with the
weight zero part of $\Der(-\log\,h)$, which we denote by  $\Der(-\log\,h)_0$.
$\Der(-\log h)$ is a summand of $\Der(-\log D)$, as is shown by the equality
\begin{displaymath}\xi=\frac{\xi\cdot h}{E\cdot h}
  E +\left(\xi-\frac{\xi\cdot h}{E\cdot h}E\right)\end{displaymath}
in which $E$ is the Euler vector field and the second summand on the right
is easily seen to annihilate $h$.

The quotient complex
\begin{displaymath}
  \Omega^\bu(\log h):=\frac{\O^\bu(\log D)}{dh/h\wedge\O^{\bu-1}(\log
    D)} =
    \frac{\O^\bu(\log D)}{h^*\left(\Omega^1_\CC(\log\{0\})\right)\wedge\O^{\bu-1}(\log
    D)}
\end{displaymath}
is the \emph{relative logarithmic de Rham complex} associated with the function $h:\CC^n\to
\CC$. Each module $\O^k(\log h)$ is isomorphic to the submodule
\begin{displaymath}
  \O^k(\log h)':=\{\o\in \O^k(\log D):\iota_{E}\o=0\} \subset \O^k(\log D).
\end{displaymath}
This is because the natural map $i:\o\mapsto \o+(dh/h)\w\O^{\bu-1}(\log D)$ gives
an injection $\O^k(\log h)'\to \O^k(\log h)$, since for $\o\in\O^k(\log h)'$, if
$\o=(dh/h)\w \o_1$ for some $\o_1$ then
\begin{displaymath}0=\iota_{E}(\o)=\iota_{E}\left(\frac{dh}{h}\w\o_1\right)=n\o_1-
  \frac{dh}{h}\w\iota_{E}(\o_1)\end{displaymath} and thus $\o_1=(dh/nh)\w\iota_{E}(\o_1)$
and $\o=(dh/h)\w(dh/nh)\w\iota_{E}(\o_1)=0.$ Because
\begin{displaymath}\frac{1}{n}\iota_{E}\left(\frac{dh}{h}\w \o\right)\in \O^k(\log
  h)'\end{displaymath}
and
\begin{equation}
  \label{fi}
  \o-\frac{1}{n}\iota_{E}\left(\frac{dh}{h}\w\o\right)
  \ \in\ \frac{dh}{h}\w\O^{k-1}(\log D)
\end{equation}
$i$ is surjective.
However, the collection of $\O^k(\log h)'$ is not a subcomplex of
$\O^\bu(\log D)$:  The form $\iota_{E}(d\o)$ may not be zero even when
$\iota_{E}(\o)=0$. We define $d':\O^k(\log h)'\to
\O^{k+1}(\log h)'$ by composing the usual exterior
derivative $\O^k(\log h)'\to\O^{k+1}(\log D)$ with with the
projection operator $P:\O^\bu(\log D)\to \O^\bu(\log h)'$
defined by
\begin{equation}\label{defd'}
  P(\o)=\frac{1}{n}\iota_{E}\left(\frac{dh}{h}\w\o\right)
  \ =\ \o-\frac{1}{n}\frac{dh}{h}\w\iota_{E}(\o).\end{equation}
\begin{lemma}
  \begin{enumerate}
  \item $d\cc i=i\cc d'$.
  \item The differential satisfies $(d')^2=0$.
  \item The mapping $i:\bigl(\O^\bu(\log h)',d'\bigr) \to \bigl(\O^\bu(\log h),d\bigr)$ is an
    isomorphism of complexes.
  \end{enumerate}
\end{lemma}
\begin{proof} The first statement is an obvious consequence of the second equality in
  (\ref{defd'}). The second follows because $d^2=0$ and $i$ is an injection. The third is
  a consequence of (i) and (ii).
\end{proof}
\begin{lemma}\label{wt0}
  The weight zero part of $\bigl(\O^\bu(\log h)',d'\bigr)$ is a subcomplex of
  $\bigl(\O^\bu(\log D),d\bigr)$.
\end{lemma}
\begin{proof}
  Let $\o\in\O^k(\log h)'$. We have
  \begin{displaymath}P_{k+1}(d\o)=\frac{1}{n}\iota_{E}\left(\frac{dh}{h}\w d\o\right)=
    d\o-\frac{1}{n}\frac{dh}{h}\w \iota_{E}(d\o) =d\o-\frac{1}{n}\frac{dh}{h}\w
    \bigl(L_{E}(\o)-d\iota_{E}(\o)\bigr)\end{displaymath} where $L_{E}$ is the Lie
  derivative with respect to $E$. By assumption, $\iota_{E}(\o)=0$, and since
  $L_{E}(\sigma)=\mbox{weight}(\sigma)\sigma$ for any homogeneous form, it follows that if
  $\mbox{weight}(\o)=0$ then $d'\o=d\o$.\end{proof} Let
\begin{equation}\label{defal}\a=\iota_{E}\left(\frac{dx_1\w\cd\w
      dx_n}{h}\right)\end{equation}
Evidently $\a\in\O^{n-1}(\log h)'$, and moreover
\begin{displaymath}\alpha= n\frac{dx_1\w\cd\w dx_n}{dh}.\end{displaymath}
For $\xi\in\Der(-\log h)$, we define the form $\lambda_\xi=\iota_{\xi}\a$.
Notice that $\alpha$ generates the rank one $\CC[V]$-module
$\Omega^{n-1}(\log\, h)$: We have $\alpha\wedge dh/nh=dx_1\wedge\ldots dx_n/h$, which is a generator of $\Omega^n(\log\,D)$ (remember
that $dh/nh$ is the element of $\Omega^1(\log \,D)$ dual to $E\in\Der(-\log\,D)$).

\begin{lemma}\label{dl0}
  The linear free divisor $D\subset\CC^n$ is special if and only if $d\lambda_\xi=0$ for
  all $\xi\in\Der(-\log h)_0$.
\end{lemma}
\begin{proof} Let $\xi\in \Der(-\log h)_0$ and let $\lambda_\xi=\iota_{\xi}(\a)=
  \iota_{\xi}\iota_E\vol.$ Since $\alpha$ generates $\Omega^{n-1}(\log\, h)$ and $\lambda_\xi$ has weight zero,
  $d'\lambda_\xi=c\a$ for some scalar $c$.  By the previous lemma, the same is true for $d\lambda_\xi$. Since
  $dh\w\a=\vol$, it follows that $dh\w d\lambda_\xi=c\vol$.  Now $dh\w
  d\lambda_\xi=-d(dh\w \lambda_\xi)=d\iota_{\xi}(\vol)=L_{\xi}(\vol).$ An easy calculation
  shows that $L_{\xi}(\vol)=\operatorname{trace}(A)\vol$, where $A$ is the $n\times n$
  matrix such that $A\cdot x=\xi(x)$. Hence
  \begin{displaymath}d\lambda_\xi=0\quad\iff\quad \operatorname{trace}(A)=0.\end{displaymath} Thus
  $d\lambda_\xi=0$ for all $\xi\in\Der(-\log D)$ if and only if
  $\operatorname{trace}(A)=0$ for all matrices $A\in \ker d\chi_h$, i.e. if and only if
  $\ker d\chi_h  \subseteq \ker d\det$. Since both kernels have codimension $1$, the inclusion holds if
  and only if equality holds, and this is equivalent to $\chi_h$ being a power of
  $\det$. On the other hand, regarding $G^0_D$ as a subgroup of $\Gl_n(\CC)$, both $\det$ and $\chi_h$ are polynomials
of degree $n$, so they must be equal.
  \end{proof}
%%%%%%%%%%%%%%%%%%%%%%%%%%%%%%%%%%%%%%%%%%%%%%%%%%%%%%%%%%%%%%%%%%%%%%%%%%
%\subsection{Relative logarithmic de Rham cohomology}
%\label{sec:Cohomology}

If $D$ is a linear free divisor with reductive group $\GD$ and reduced homogeneous equation
$h$ then
by Mather's lemma (\cite[lemma 3.1]{Math4})
the fibre $D_t:=h^{-1}(t)$, $t\neq 0$,
is a single orbit of the group $\ker(\chi_h)$.
It follows that $D_t$ is a finite quotient of $\ker(\chi_h)$
since $\dim(D_t)=\dim(\ker(\chi_h))$ and the action is algebraic.
Hence $D_t$ has cohomology isomorphic to $H^*(\ker(\chi_h),\CC)$. Now $\ker(\chi_h)$ is reductive -
its Lie algebra $\gg_h$ has the same semi-simple part as $\gg_D$, and a centre
one dimension smaller than that of $\gg_D$. Thus, $\ker(\chi_h)$ has a
compact $n-1$-dimensional Lie group $K_h$ as deformation retract.
Poincar\'e duality for $K_h$ implies a duality on the cohomology of
$\ker(\chi_h)$, and this duality carries over to
$H^*(D_t;\CC)$.  How is this reflected in the cohomology of the complex
$\O^\bu(\log h)$ of relative logarithmic forms (in order to simplify notations,
we write $\O^\bu$ for the spaces of global sections of algebraic differential forms)? Notice
that evidently $H^0(\O^\bu(\log h))=\CC[h]$, since the kernel of
$d_h$ consists precisely of functions constant along the fibres
of $h$. It is considerably less obvious that $H^{n-1}(\O^\bu(\log h))$
should be isomorphic to $\CC[h]$, for this cohomology group is naturally
a quotient, rather than a subspace, of $\CC[V]$.  We prove it (in Theorem \ref{nbi}
below) by showing that
thanks to the reductiveness of $\GD$, it follows from a classical theorem of
Hochschild and Serre (\cite[Theorem 10]{hs}) on the cohomology of Lie algebras.
From Theorem \ref{nbi} we then deduce that every reductive linear free
divisor is special.

We write $\O^\bu(\log h)_m$ for the graded part of $\O^\bu(\log h)$ of weight $m$.
\begin{theorem}\label{nbi}
Let $D\subset\CC^n$ be a reductive linear free divisor with homogeneous equation $h$.
There is a natural graded isomorphism
  \begin{displaymath}
    H^*(\O^\bu(\log h)_0)\otimes_{\CCC}\CC[h]\to H^*(\O^\bu(\log
    h)).
  \end{displaymath}
In particular, $H^*(\O^\bu(\log h))$ is a free $\CC[h]$-module.
\end{theorem}
\begin{proof} The complex $\O^\bu(\log h)_m$ is naturally identified with the complex
  $\bigwedge^\bu(\gg_h;\Sym^m(V^\vee))$ whose cohomology is the Lie algebra cohomology of
  $\gg_h$ with coefficients in the representation $\Sym^m(V^\vee)$, denoted by
  $H^*(\gg_h;\Sym^m(V^\vee))$. This is because we have the following equality  of vector spaces,
  \begin{displaymath}
    \O^k(\log h)_m=\O^k(\log h)_0\otimes_\CC \Sym^m(V^\vee)
    =\left(\bigwedge^k\gg_h^\vee\right)\otimes_{\CCC} \Sym^m\left(V^\vee\right)
    =\bigwedge^k\left(\gg_h^\vee \otimes_{\CCC}\Sym^m\left(V^\vee\right)\right),
  \end{displaymath}
  and inspection of the formulae for the differentials in the two complexes shows that
  they are the same under this identification. Notice that this identification for the case
    $m=0$ was already made in \cite{gmns}, where it gave a proof of the global logarithmic
    comparison theorem for reductive linear free divisors.  The representation of
  $\gg_h$ in $\Sym^k(V^\vee)$ is semi-simple (completely reducible), since $\gg_h$ is a
  reductive Lie algebra and every finite dimensional complex representation of a reductive
  Lie algebra is semisimple.  By a classical theorem of Hochschild and Serre
  (\cite[Theorem 10]{hs}), if $M$ is a semi-simple representation of a finite-dimensional
  complex reductive Lie algebra $\gg$, then
  \begin{displaymath}
    H^*(\gg;M)=H^*(\gg;M^0),
  \end{displaymath}
  where $M^0$ is the
  submodule of $M$ on which $\gg$ acts trivially.  Evidently we have
  $H^*(\gg;M^0)=H^*(\gg;\CC)\otimes_{\CCC} M^0$.  Now
  \begin{displaymath}
    \Sym^m(V^\vee)^0=\left\{
      \begin{array}{ll}\CC\cdot h^\ell&\mbox{if}\ m=\ell n\\
        0&\mbox{otherwise}
      \end{array}
    \right.
  \end{displaymath}
  by the uniqueness, up to scalar multiple, of the semi-invariant with a given character
  on a prehomogeneous vector space (see the proof of lemma \ref{f*h} below for a more detailed
  explanation). It follows
  that
  \begin{align*}
    H^k(\O^\bu(\log h))&=\bigoplus_mH^k(\O^\bu(\log h)_m)
    =\bigoplus_mH^k(\gg_h;\Sym^m(V^\vee))\\
    &=\bigoplus_\ell H^*(\gg_h;\CC)\otimes_{\CCC}\CC
    \cdot h^\ell=H^*(\O^\bu(\log h)_0)\otimes_{\CCC}\CC[h].
  \end{align*}
\end{proof}
\begin{corollary} There is a $\CC[h]$-perfect pairing
  \begin{align*}
    H^k(\O^\bu(\log h))\times H^{n-k-1}(\O^\bu(\log h))&\longrightarrow
    H^{n-1}(\O^\bu(\log h)) \simeq\CC[h]\\
    ([\o_1],[\o_2])&\longmapsto[\o_1\wedge\o_2].
  \end{align*}
\end{corollary}
\begin{proof}
  The pairing is evidently well defined.  Poincar\'e Duality on the compact
  deformation retract $K_h$ of $\ker(\chi_h)$ gives rise to a perfect pairing
  \begin{displaymath}
    H^k(D_t)\times H^{n-k-1}(D_t)\to H^{n-1}(D_t),
  \end{displaymath}
  Now
  \begin{displaymath}
    H^k(D_t)=H^k\left(\O^\bu(\log h) \otimes_{\CC[h]}\CC[h]/(h-t)\right)
  \end{displaymath}
  by the affine de Rham theorem,
  since $\O^k(\log h)/(h-t)=\O^k_{D_t}$.  In view of theorem \ref{nbi}, the perfect pairing on
  $H^*(D_t)$ lifts to a $\CC[h]$-perfect pairing on $H^*\bigl(\G(V,\O^\bu(\log
  h))\bigr)$.
\end{proof}
\begin{corollary}\label{cor2}
  A linear free divisor with reductive group is special.
\end{corollary}
\begin{proof}
  %Reductivity of $\GD$ implies reductivity of $\ker\chi_h$, since their two Lie algebras
  %differ only in the dimension of the centre.
  By what was said before, $H^{n-1}(\ker(\chi_h),\CC)$ is isomorphic to $H^{n-1}(\Omega^\bu(\log\,h)_0)$, so
  Poincar\'e duality for $\ker\chi_h$ implies
  that the class of $\a$ in $H^{n-1}(\O^\bu(\log h)_0)$ is non-zero. Recall from the proof of lemma
  \ref{dl0} that if
  $\lambda=\iota_\xi\alpha=\iota_\xi\iota_E(\vol/h)$ with $\xi\in\Der(-\log h)_0$, then
  $d\lambda=c\alpha$ in $\O^\bu(\log h)_0$ for some $c\in\CC$. As the class of $\alpha$ is non-zero,
  this forces $d\lambda$ to be zero. The conclusion follows from \ref{dl0}.
\end{proof}
\section{Functions on Linear Free Divisors and their Milnor Fibrations}
\label{sec:funcOnLFD}
\subsection{Right-left stable functions on divisors}
\label{subsec:LeftRight}

Let $h$ and $f$ be homogenous polynomials in $n$ variables, where
the degree of $h$ is $n$.
As before, we write $D=h^{-1}(0)$ and $D_t=h^{-1}(t)$ for $t\neq 0$.
However, we do not assume in this subsection that $D$ is a free divisor.
We call $f_{|D_t}$ a {\it Morse function}
if all its critical points are isolated and non-degenerate and all its critical values are
distinct.
\begin{lemma}\label{mond:lst1}
$f_{|D_t}$ is a Morse function
if and only if
$\CC[D_t]/J_{f}$ is generated over $\CC$ by the powers of $f$.
\end{lemma}
\begin{proof} Suppose $f_{|D_t}$ is a Morse function, with critical points
$p_1,\ld,\,p_N$.
Since any quotient of $\CC[D_t]$ with finite support is a product of its
localisations, we have
\begin{displaymath}
  \CC[D_t]/J_f\simeq \oplus_{j=1}^N\OO_{D_t,p_j}/J_f=\oplus_{j=1}^N\CC_{p_j}.
\end{displaymath}
The image in $\oplus_{k=j}^N\CC_{p_j}$ of $f^k$ is the vector
$\bigl(f(p_1)^k,\ld\,, f(p_N)^k\bigr)$. These vectors, for $0\leq k\leq N-1$, make up
the Vandermonde determinant, which is non-zero because the $f(p_j)$ are pairwise
distinct. Hence they span $\oplus_{j=1}^N\CC_{p_j}.$

Conversely, if $1,f,\ld,f^N$ span $\CC[D_t]/J_f$ then the powers of
$f$ span each local ring $\OO_{D_t,p_j}/J_f$. This implies that there is
an $\s R_e$-versal deformation of the singularity of $f_{|D_t}$ at $p_j$
of the form $F(x,u)=g_u\cc f(x)$. In particular, the critical point of $f_{|D_t}$
at $p_j$ does not split, and so must be non-degenerate. Now choose a minimal $R$
such that $1,f,\ld,f^{R-1}$ span $\CC[D_t]/J_f$. Since all the critical points
are non-degenerate, projection of $\CC[D_t]/J_f$ to the product of its
local rings shows that the matrix $M:=[f^{k-1}(p_j)]_{1\leq k\leq R,1\leq j\leq N}$
has rank $N$. But if
$f(p_i)=f(p_j)$ for some $i\neq j$ then $M$ has
two equal columns. So the critical values of $f$ must be pairwise distinct.
\end{proof}
If $(X,x)$ is a germ of complex variety, an analytic map-germ
$f:(X,x)\to (\CC^p,0)$ is {\it right-left stable} if every germ of deformation
$F:(X\times\CC,(x,0))\to (\CC^p\times\CC,(0,0))$ can be trivialised by suitable
parametrised families of bi-analytic diffeomorphisms of source and target.
A necessary and sufficient condition for right-left stability is infinitesimal
right-left stability: $df(\theta_{X,0})+f^{-1}(\theta_{\CC^p,0})=\theta(f)$,
where $\theta_{X,0}$ is the space of germs of vector fields on $X$ and
$\theta(f)=f^*\theta_{\CC^p,0}$ is the space of infinitesimal deformations of $f$ (freely
generated over $\OO_{X,0}$ by $\p/\p y_1,\ld\,,\p/\p y_p$,
where $y_1,\ld\,, y_p$ are
coordinates on $\CC^p$). When $p=1$, $\theta(f)\simeq \OO_{X,0}$ and
$f^{-1}(\theta_{\CC,0})\simeq \CC\{f\}.$ Note also that if $X\subset\CC^n$ then
$\theta_{X,0}$ is the image of $\Der(-\log X)_0$ under the restriction of $\theta_{\CC^n,0}$
to $(X,0)$.
\begin{proposition}\label{mond:lst2}
If $f_{|D}:D\to\CC$ has a right-left stable singularity at $0$
then $f_{|D_t}$ is a Morse function, or non-singular.
\end{proposition}
\begin{proof} $f_{|D}$ has a stable singularity at $0$ if
and only if the image in $\OO_{D,0}$ of
$df(\Der(-\log D))+\CC\{f\}$ is all of $\OO_{D,0}$.
Write $\mathfrak{m}:=\mathfrak{m}_{\CC^n,0}$.
Since $df(\chi_E)=f$, stability implies
\begin{equation}\label{seq}
df(\Der(-\log h))+(f)+(h)\supseteq \mathfrak{m}.
\end{equation}
This is an equality unless $D\cong D'\times\CC$ and $\partial_{t_0}
f \neq 0$, where $t_0$ is a coordinate on the factor $\CC$. In
this case $f$ is non-singular on all the fibres $h^{-1}(t)$ for $t\neq
0$. So we may assume that \eqref{seq} is an equality.

If $\deg(h)=1$, then $D$ is non-singular and the result follows
immediately from Mather's theorem that infinitesimal stability implies
stability. Hence we may also assume that $\mbox{deg}(h) > 1$.
It follows that
$$\mm/\bigl(df(\Der(-\log h))+\mm^2\bigr)=\langle f\rangle_{\CC}.$$
It follows that for all $k\in \NN$,
$$\bigl(\mm^k+df(\Der(-\log h))\bigr)/\bigl(df(\Der(-\log h))+\mm^{k+1}\bigr)=
\langle f^k\rangle_{\CC}$$
and thus that
\beq\label{teq}df(\Der(-\log h))+\CC\{f\}=\OO_{\CC^n,0}.\eeq
Now \eqref{seq} implies that $V(df(\Der(-\log h)))$ is either a line or
a point. Call it $L_f$.
If $L_f\not\subseteq D$, then
the sheaf $h_*\bigl(\OO_{\CC^n}/df(\Der(-\log h))\bigr)$ is
finite over $\OO_{\CC}$,
and (\ref{teq}) shows that its stalk at $0$ is generated by $1,f,\ld,f^R$ for
some finite $R$. Hence these same sections generate
$h_*\bigl(\OO_{\CC^n}/df(\Der(-\log h))\bigr)_t$ for
$t$ near $0$, and therefore for all $t$, by homogeneity. As
$h_*\bigl(\OO_{\CC^n}/df(\Der(-\log h))\bigr)_t=
\CC[D_t]/J_f$, by
\ref{mond:lst1} $f_{|D_t}$ is a Morse function.

On the other hand, if $L_f\subset D$, then $f:D_t\to\CC$ is non-singular.
\end{proof}
We do not know of any example where the latter alternative holds.
%Conclusion: stable functions on divisors give rise to semi-simple $F$-manifolds.
\begin{proposition} If $f:(D,0)\to(\CC,0)$ is right-left stable then $f$ is linear and
$\Der(-\log D)_0$ must contain at least $n$ linearly independent weight zero vector fields.
In particular, the only free divisors supporting right-left stable functions are linear
free divisors.
\end{proposition}
\begin{proof}
From equation \eqref{seq} it is obvious that $f$ must be linear, and that
$Der(-\log h)$ must contain
at least $n-1$ independent weight zero vector fields;
these, together
with the Euler field, make $n$ in $\Der(-\log D)$.
\end{proof}
We note that the hypothesis of the proposition is fulfilled by a generic linear
function on the hypersurface defined by $\sum_jx_j^2=0$, which is not a free divisor if $n\geq 3$.

%%%%%%%%%%%%%%%%%%%%%%%%%%%%%%%%%%%%%%%%%%%%%%%%%%%%%%%%%%%%%%%%%%

\subsection{$\s R_D$- and $\s R_h$-equivalence of functions on divisors}
\label{subsec:REquivalences}

Let $D\subset\CC^n$ be a weighted homogeneous free divisor and let $h$ be its weighted
homogeneous equation. We consider functions $f:\CC^n\to\CC$ and their restrictions to the
fibres of $h$. The natural equivalence relation to impose on functions on $D$ is $\s
R_D$-equivalence: right-equivalence with respect to the group of bianalytic
diffeomorphisms of $\CC^n$ which preserve $D$. However, as we are interested also in the
behaviour of $f$ on the fibres of $h$ over $t\neq 0$, we consider also {\it fibred
right-equivalence} with respect to the function $h:(\CC^n,0)\to\CC$. That is,
right-equivalence under the action of the group $\s R_h$ consisting of germs of bianalytic
diffeomorphisms $\pp:(\CC^n,0)\to(\CC^n,0)$ such that $h\cc\pp=h$. A standard calculation
shows that the tangent spaces to the $\s R_D$ and $\s R_h$-orbits of $f$ are equal to
$df(\Der(-\log D))$ and $df(\Der(-\log h))$ respectively. We define
\begin{align*}
  T^1_{\s R_D}f&:=\frac{\OO_{\CC^n,0}}{df(\Der(-\log D))}\\
  T^1_{\s R_h}f&:=\frac{\OO_{\CC^n,0}}{df(\Der(-\log  h))+(h)}\\
  T^1_{\s R_h/\CC}f&:=\frac{\OO_{\CC^n,0}}{df(\Der(-\log h))}
\end{align*}
and say that $f$ is $\s R_D$-finite or $\s R_h$-finite if
$\dim_{\CCC}T^1_{\s R_D}f<\infty$ or $\dim_{\CCC}T^1_{\s R_h}f<\infty$ respectively.  Note
that it is only in the definition of $T^1_{\s R_h}f$ that we explicitly restrict to the
hypersurface $D$.

We remark that a closely related notion called $_D\s K$-equivalence is studied by Damon in
\cite{legIII}.
% @@@@@@@@@@@@@@@@@@@@@@@@@@@@@@
\begin{proposition}\label{numcons} If the germ $f\in\OO_{\CC^n,0}$ is
  $\s R_h$-finite then there exist $\ve>0$ and $\eta>0$ such that for $t\in\CC$ with
  $|t|<\eta$,
  \begin{displaymath}
    \sum_{x\in D_t\cap B_{\ve}}\mu(f_{|D_t};x)=\dim_{\CC}T^1_{\s
      R_h}f.
  \end{displaymath}
  If $f$ is weighted homogeneous (with respect to the same weights as $h$)
  then $\ve$ and $\eta$ may be taken to be
  infinite.
\end{proposition}
\begin{proof}
  Let $\xi_1,\ld\xi_{n-1}$ be an $\OO_{\CC^n,0}$-basis for $\Der(-\log h)$. The $\s
  R_h$-finiteness of $f$ implies that the functions $df(\xi_1),\ldots, df(\xi_{n-1})$ form a regular
  sequence in $\OO_{\CC^n,0}$, so that $\Tr f$ is a complete intersection ring, and in
  particular Cohen-Macaulay, of dimension 1.  The condition of $\s R_h$-finiteness is
  equivalent to $T^1_{\s R_h/\CC}f$ being finite over $\OO_{\CC,0}$.  It follows that it
  is locally free over $\OO_{\CC,0}$.
\end{proof}

Now suppose that $D\subset \CC^n=V$ is a linear free divisor.
We denote the dual space $\Hom_{\CC}(V,\CC)$
by $V^\vee$. The group $\GD$ acts on $V^\vee$ by the contragredient action $\rho^\vee$ in
which
\begin{displaymath}
  g\cdot f=f\cc\rho(g)^{-1}.
\end{displaymath}
If we write the elements of
$V^\vee\simeq\CC^n$ as column vectors, then the representation $\rho^\vee$ takes the form
$\rho^\vee(g)={^t}\rho(g)^{-1},$ and the infinitesimal action takes the form
$d\rho^\vee(A)=-{^t}A.$ Let $A_1,\ld,A_n$ be a basis for $\gg_D$. Then the vector fields
\begin{equation}
  \label{vf}
\xi_i(x)=(\p/\p x_1, \ld, \p/\p x_n)A_i x,\quad \mbox{for}\ i=1,\ld,n
\end{equation}
form an $\OO_{\CC^n}$ basis for $\Der(-\log D)$, and the determinant of the $n\times n$
matrix of their coefficients is a non-zero scalar multiple of $h$, by Saito's
criterion. The vector fields
\begin{equation}
  \label{v*f}
\xi_i(y)=(\p/\p y_1, \ld, \p/\p y_n)({^tA}_i) y,\quad \mbox{for}\ i=1,\ld,n
\end{equation}
generate the infinitesimal action of $\gg_D$ on $V^\vee$. We denote by $h^\vee$ the
determinant of the $n\times n$ matrix of their coefficients. Its zero-locus is the
complement of the open orbit of $\GD$ on $V^\vee$ (including when the open orbit is empty).  In general $\rho^\vee$ and $\rho$ are
not equivalent representations. Indeed, it is not always the case that
$(\GD,\rho^\vee,V^\vee)$
is a prehomogeneous vector space. We describe an example where this occurs in \ref{path1}
below.

Suppose $f\in V^\vee$. Let $L_f=\supp T^1_{\s R_h/\CC}f$.  Since $\Der(-\log h)$ is
generated by weight zero vector fields, $L_f$ is a linear subspace of $V$.
\begin{proposition}
  \label{fincon}
  Let $f\in V^\vee$. Then
  \begin{enumerate}
  \item The space $L_f$ is a line transverse to $f^{-1}(0) \ \iff \ f$ is $\s R_D$-finite $\iff$ the
    $\GD$-orbit of $f$ in the representation $\rho^\vee$ is open.
  \item Suppose that $f=0$ us an equation for the tangent plane $T_pD_t$, then
    \begin{equation}
      \label{mu1}
      H(p)\neq 0\quad\implies\quad\mu(f_{|D_t};p)=1
    \end{equation}
    where $H$ is the Hessian determinant of $h$.
  \item If $f$ is $\s R_h$-finite then
    \begin{enumerate}
    \item $f$ is $\s R_D$- finite;
    \item the classes of $1, f,\ld, f^{n-1}$ form a $\CC$-basis for $T^1_{\s R_h}f$;
    \item on each Milnor fibre $D_t:=h^{-1}(t), t\neq 0$, $f$ has $n$ non-degenerate
      critical points, which form an orbit under the diagonal action of the group of
      $n$-th roots of unity on $\CC^n$.
    \end{enumerate}
  \end{enumerate}
\end{proposition}
\begin{proof}
  \noindent{(i)} The first equivalence holds simply because
  \begin{displaymath}
    df(\Der(-\log D))=df(\Der(-\log h))+(df(E))=df(\Der(-\log
    h))+(f).
  \end{displaymath} For the second equivalence, observe that the tangent space
  to the $\GD$-orbit of $f$ is naturally identified with $df(\Der(-\log D))\subset
  \mathfrak{m}_{V,0}/\mathfrak{m}^2_{V,0}=V^\vee$. For given $A\in \gg_D$, we have
  \begin{equation}
    \label{icga}
    \left(\frac{d}{dt}\exp(tA)\cdot
      f\right)_{|t=0}(x) =df\left(\frac{d}{dt}\exp(-tA)\cdot x\right)_{|t=0}\ =\
    -df(\xi_A)
  \end{equation}
  where $\xi_A$ is the vector field on $V$ arising from $A$ under the the infinitesimal
  action of $\rho$. Because $\Der(-\log D)$ is generated by vector fields of weight zero,
  $df(\Der(-\log D))$ is generated by linear forms, and so $f$ is $\s R_D$-finite if and
  only if $df(\Der(-\log D))\supset \mathfrak{m}_{V,0}$.

  \smallskip
  \noindent{(ii)} is well-known. To prove it, parametrise $D_t$ around $p$ by
  $\pp:(\CC^{n-1},0)\to (D_t,p)$. Then because $f$ is linear we have
  \begin{equation}
    \label{eq1}
    \frac{\p^2(f\cc\pp)}{\p u_i\p u_j}=\sum_s\frac{\p f}{\p x_s} \frac{\p^2\pp_s}{\p u_i\p
      u_j}.
  \end{equation}
  Because $h\cc\pp$ is constant, we find that
  \begin{equation}
    \label{eq2}
    0=\sum_{s,t}\frac{\p^2h}{\p x_s\p x_t}\frac{\p \pp_s}{\p
      u_i} \frac{\p \pp_t}{\p u_j}+\sum_s\frac{\p h}{\p x_s} \frac{\p^2\pp_s}{\p u_i\p
      u_j}.
  \end{equation}
  Because $T_pD_t=\{f=0\}$, $d_ph$ is a scalar multiple of $d_pf=f$.  From this, equations
  (\ref{eq1}) and (\ref{eq2}) give an equality (up to non-zero scalar multiple) of
  $n\times n$ matrices,
  \begin{equation}
    \label{eq3}
    \left[\frac{\p^2 (f\cc\pp)}{\p u_i\p u_j}\right] ={^t\left[\frac{\p \pp_s}{\p
          u_i}\right]} \left[\frac{\p^2h}{\p x_s\p x_t}\cc\pp\right] \left[\frac{\p
        \pp_t}{\p u_j}\right]
  \end{equation}
  It follows that if $H\neq 0$ then the restriction of $f$ to $D_t$ has a non-degenerate
  critical point at $p$.

  \smallskip
  \noindent{(iii)(a)} If $f$ is $\s R_h$-finite then $L_f$ must be a line intersecting $D$
  only at $0$. If $\s R_D$ finiteness of $f$ fails, then $L_f\subset\{f=0\}$, and $f$ is
  constant along $L_f$. But at all points $p\in L_f$, $\ker d_pf\subset\ker d_ph$, so $h$
  also is constant along $L_f$.

  \smallskip
  \noindent{(iii)(b)} As $L_f$ is a line and $\OO_{V}/df(\Der(-\log h))=\OO_{L_f}$,
  $h_{|L_f}$ is necessarily the $n$'th power of a generator of $m_{L_f,0}$.  It follows
  that $T^1_{\s R_h}f$ is generated by the first $n$ non-negative powers of any linear
  form whose zero locus is transverse to the line $L_f$.

  \smallskip
  \noindent{(iii)(c)} Since $f$ is $\s R_D$ finite, $L_f$ is a line transverse to
  $\{f=0\}$. The critical points of $f_{|D_t}$ are those points $p\in D_t$ where
  $T_pD_t=\{f=0\}$; thus $L_f\,\tr D_t$ at each critical point.  In $\OO_{D_t}$, the
  ideals $df(\Der(-\log h))$ and $J_{f_{|D_t}}$ coincide. Thus the intersection number of
  $L_f$ with $D_t$ at $p$, which we already know is equal to 1, is also equal to the
  Milnor number of $f_{|D_t}$ at $p$. The fact that there are $n$ critical points,
  counting multiplicity, is just the fundamental theorem of algebra, applied to the
  single-variable polynomial $(h-t)_{|L_f}$. The fact that these $n$ points form an orbit
  under the diagonal action of the group $\GG_n$ of $n$-th roots of unity is a consequence
  simply of the fact that $h$ is $\GG_n$-invariant and $L_f$ is preserved by the action.
\end{proof}
% @@@@@@@@@@@@@@@@@@@@@@@@@@@@@@
%For homogeneous functions $f$ of higher degree, the implications of \ref{fincon}(iii)
%hold only with the additional assumption that $f$ has an isolated singularity at $0\in\CC^n$.

If $D$ is a linear free divisor, there may be no $\s R_h$-finite linear forms, or even no
$\s R_D$ finite linear forms, as the following examples shows.
\begin{example}
  \label{path1}
  Let $D$ be the free divisor in the space $V$ of $2\times 5$ complex matrices defined by
  the vanishing of the product of the $2\times 2$ minors $m_{12},m_{13},m_{23}, m_{34}$
  and $m_{35}$. Then $D$ is a linear free divisor (\cite[Example 5.7(2)]{gmns}), but
  $\rho^\vee$ has no open orbit in $V^\vee$: it is easily checked that $h^\vee=0.$ It
  follows by \ref{fincon} (i) that no linear function $f\in V^\vee$ is $\s R_D$-finite, and
  so by \ref{fincon} (iii) that none is $R_h$-finite.
\end{example}
In Example \ref{path1}, the group $\GD$ is not reductive.  Results of Sato and Kimura in
\cite[\S 4]{sk} show that if $\GD$ is reductive then $(\GD,\rho^\vee,V^\vee)$ is
prehomogeneous, so that almost all $f\in V^\vee$ are $\s R_D$-finite, and moreover imply
that all $f$ in the open orbit in $V^\vee$ are $\s R_h$-finite.  We briefly review their
results. As we will see, the complement of the open orbit in $V^\vee$ is a divisor whose equation,
in suitable coordinates $x$ on $V$, and dual coordinates $y$ on $V^\vee$, is of the form
is of the form $h^\vee=\overline{h(\bar y)}$. From now on we will
denote the function $y\mapsto \overline{h(\bar y)}$ by $h^*(y)$.
The coordinates in question are chosen as follows: as $\GD$
is reductive, it has a Zariski dense compact subgroup $K$. In suitable coordinates on
$V=\CC^n$ the representation $\rho$ places $K$ inside $U(n)$. Call such a coordinate
system {\it unitary}.  From this it follows that if $f$ is any rational semi-invariant on
$V$ with associated character $\chi$ then the function $f^*:V^\vee\to\CC$ defined by
$f^*(y)=\overline{f(\bar y)}$ is also a semi-invariant {\it for the representation of
$K$} with associated
character $\bar{\chi}$, which
is equal to $\chi^{-1}$ since $\chi(K)\subset S^1$ by compactness.
Note that $f^*$ cannot be the zero polynomial. As $K$ is
Zariski-dense in $\GD$, the rational equality
\begin{displaymath}
  f^*(\rho^\vee(g)y)=\frac{1}{\chi(g)}f^*(y)
\end{displaymath}
holds for
all $g\in \GD$.
\begin{proposition} \label{propUnitaryCoord}
  Let $D\subset\CC^n$ be a linear free divisor with equation $h$.  If
  $\GD$ is reductive then
  \begin{enumerate}
  \item The tuple $(\GD,\rho^\vee,V^\vee)$ is a prehomogeneous vector space.
  \item $D^\vee$, the complement of the open orbit in $V^\vee$, has equation $h^*$, with
    respect to dual unitary coordinates on $V^\vee$.
  \item $D^\vee$ is a linear free divisor.
  \end{enumerate}
\end{proposition}
\begin{proof}
  As $\CC$-basis of the Lie algebra $\gg_D$ of $\GD$ we can take a real basis of the Lie
  algebra of $K$. With respect to unitary coordinates, $\rho$ represents $K$ in $U(n)$, so
  $d\rho(\gg_D)\subset\gl_n(\CC)$ has $\CC$-basis $A_1,\ld, A_n$ such that $A_i\in
  \mathfrak u_n$, i.e.  $^tA_i=-\bar A_i$, for $i=1,\ld,n$.  It follows that the
  determinant of the matrix of coefficients of the matrix (\ref{v*f}) above is equal to
  $h^*$, and in particular is not zero. This proves (i) and (ii).

  That $D^\vee$ is free follows from Saito's criterion (\cite{saito}): the $n$ vector
  fields (\ref{v*f}) are logarithmic with respect to $D^\vee$, and $h^*$, the determinant
  of their matrix of coefficients, is not identically zero, and indeed is a reduced
  equation for $D^\vee$ because $h$ is reduced.
\end{proof}
We now prove the main result of this section. In order to make the argument clear, we
postpone some steps in the proof to lemmas \ref{hessian} and \ref{f*h} and to proposition
\ref{serendip},
which we prove immediately afterwards.
\begin{theorem}
  \label{finfin}
  If $\GD$ is reductive then $f\in V^\vee$ is $\s R_h$-finite if and only if it is $\s
  R_D$-finite. In particular, $f$ is $\s R_h$-finite if and only if $f\in  V^\vee\backslash D^\vee$.
\end{theorem}
% @@@@@@@@@@@@@@@@@@@@@@@@@@@@@@@@@@@@@@@@@@@@
% @@@@@@@@@@@@@@@@@@@@@@@@@@@@@@@@@@@@@@@@@@@@
\begin{proof}
  Let $p\in D_t$ (for $t\neq 0$) and
  suppose that $T_pD_t$ has equation $f$, i.e. that $\nabla h(p)$ is a
non-zero multiple of $f$. We claim that $f$ is
  $\s R_h$-finite.
 For by Lemma (\ref{f*h}) below, $H(p)\neq 0$, where $H$ is the
  Hessian determinant of $h$.  It follows by \ref{fincon} (ii)
  that the restriction of $f$ to $D_t$ has a non-degenerate critical point at $p$. The
  critical locus of $f_{|D_t}$ is precisely $L_f\cap D_t$; so $L_f$ must be a line (recall
  that it is a linear subspace of $V$), and must meet $D_t$ transversely at $p$. By the
  homogeneity of $D$, it follows that $L_f\cap D=\{0\}$, so $f$ is $\s R_h$-finite. Thus
  \begin{displaymath}
    f\ \s R_D\mbox{-finite}\quad\stackrel{\ref{fincon}}{\implies}
\quad f\in V^\vee\ssm D^\vee\quad
    \stackrel{\ref{serendip}}
{\implies}\quad f=\nabla h(p)\ \mbox{for some}\ p\notin D\quad
    {\implies}\quad f\ \s
    R_h\mbox{-finite}.
  \end{displaymath}
  We have already proved the opposite implication,
  in \ref{fincon}.
\end{proof}
% @@@@@@@@@@@@@@@@@@@@@@@@@@@@@@@@@@@@@@@@@@@
% @@@@@@@@@@@@@@@@@@@@@@@@@@@@@@@@@@@@@@@@@@@
\begin{lemma}
  \label{hessian}
  Let $D\subset\CC^n$ be a linear free divisor with homogeneous equation $h$, let $h^\vee$
  be the determinant of the matrix of coefficients of (\ref{v*f}), and let, as before, $H$ be the
  Hessian determinant of $h$.  Then
  \begin{displaymath}
    h^\vee\left(\frac{\p h}{\p x_1},\,\ld\,,\frac{\p h}{\p x_n}\right)=(n-1)Hh.
  \end{displaymath}
\end{lemma}
\begin{proof}
  Choose the basis $A_1=I_n,\ld A_n$ for $\gl_D$ so that the associated vector fields
$\xi_2,\ld, \xi_n$ are in $\Der(-\log h)$.
The matrix $I_n$ gives rise to the Euler vector field $E$. Write
  $A_i=[a^k_{ij}]$, with the upper index $k$ referring to columns and the lower index $j$
  referring to rows.  Let $\a_{ji}=\sum_ka^k_{ij}x_k$ denote the coefficient of $\p/\p
  x_j$ in $\xi_i$ for $i=2,\ld,n-1$.  Then
  \begin{displaymath}
    0=dh(\xi_i)=\sum_j\a_{ji}\frac{\p h}{\p x_j},
  \end{displaymath}
  so
  differentiating with respect to $x_k$,
  \begin{equation}
    \label{d2x}0=\sum_j\frac{\p\a_{ji}}{\p x_k}\frac{\p h}{\p x_j}+
    \sum_j\a_{ji}\frac{\p^2h}{\p x_k\p x_j}
    =\sum_ja^k_{ij}\frac{\p h}{\p x_j}+\sum_j\a_{ji}\frac{\p^2 h}{\p x_k\p x_j}.
  \end{equation}
  For the Euler field $\xi_1$ we have
  \begin{displaymath}
    nh= dh(E)=\sum_j\a_{j1}\frac{\p h}{\p x_j}
  \end{displaymath}
  so
  \begin{equation}
    \label{dxi}
    n\frac{\p h}{\p x_k}=\sum_ja^k_{1j}\frac{\p h}{\p x_j}+\sum_j\a_{j1}
    \frac{\p^2h}{\p x_k\p x_j}=\frac{\p h}{\p x_k}+\sum_j\a_{j1}
    \frac{\p^2h}{\p x_k\p x_j}.
  \end{equation}
  Putting the $n$ equations (\ref{d2x}) and (\ref{dxi}) together in
  matrix form we get
  \begin{displaymath}
    \left[
      \begin{array}
        {c}{^tE}\\{^t\xi_1}\\ \cdot\\{^t\xi_{n-1}}
      \end{array}
    \right]\left[\frac{\p^2h}{\p x_k\p x_j}\right]
  =-\left[
    \begin{array}{c}
      (n-1)\nabla h\cdot E\\
      \nabla h\cdot A_1\\
      \cdot\\
      \nabla h\cdot A_{n-1}
    \end{array}
  \right].
\end{displaymath}
Now take determinants of both sides. The determinant on the right
hand side is
\begin{displaymath}
  (n-1)h^\vee\left(\frac{\p h}{\p x_1},\,\ld\,, \frac{\p h}{\p
      x_n}\right).
\end{displaymath}
The determinants of the two matrices on the left are,
respectively, $h$ and $H$.
\end{proof}
% @@@@@@@@@@@@@@@@@@@@@@@@@@@@@@@@@@@@@@@@@@@@
% @@@@@@@@@@@@@@@@@@@@@@@@@@@@@@@@@@@@@@@@@@@@
\begin{lemma}\label{f*h}{\em(\cite{sk})} If $D$ is a reductive linear free divisor,
then for all $p\in \CC^n$
$$h(p)\neq 0\quad\implies \quad H(p)\neq 0.$$
\end{lemma}
\begin{proof}
  In \cite[Page 72]{sk}, Sato and Kimura show that if $g$ is a homogeneous rational
  semi-invariant of degree $r$ with associated character $\chi_g$ then there is a
  polynomial $b(m)$ of degree $r$ (the b-function of $g$) such that, with respect to
  unitary coordinates on $\CC^n$,
  \begin{equation}
    \label{bfcn}
    g^*\left(\frac{\p}{\p x_1},\ld,\frac{\p}{\p x_n}\right)\cdot g^m=b(m)g^{m-1}
  \end{equation}
  This is proved by showing that the left hand side is a semi-invariant with associated
  character $\chi_g^{m-1}$, and noting that the semi-invariant corresponding to a given
  character is unique up to scalar multiple, since the quotient of two semi-invariants
  with the same character is an absolute invariant, and therefore must be constant (since
  $\GD$ has a dense orbit).  From this it follows (\cite[page 72]{sk}) that
  \begin{equation}
    \label{f**}
    g^*\left(\frac{\p g}{\p x_1},\,\ld,\,\frac{\p g}{\p x_n}\right)=
    b_0{g^{r-1}},
  \end{equation}
  where $b_0$ is the (non-zero) leading coefficient of the polynomial $b(m)$, and
  hence that
  \begin{equation}
    \label{g*h}
    (n-1)H=b_0h^{n-2},
  \end{equation}
  by lemma \ref{hessian}.
\end{proof}
  \begin{proposition}
    \label{serendip}
    If $D$ is a linear free divisor with reductive group $\GD$ and homogeneous equation
    $h$ with respect to unitary coordinates, then
    \begin{enumerate}
    \item the gradient map $\nabla h$ maps the fibres $D_t, t\neq 0$ of $h$
      diffeomorphically to the fibres of $h^*$.
    \item the gradient map $\nabla h^*$ maps Milnor fibres of $h^*$ diffeomorphically to
      Milnor fibres of $h$.
    \end{enumerate}
  \end{proposition}
  \begin{proof}
    The formula (\ref{f**}) shows that $\nabla h$ maps fibres of $h$ into fibres of
    $h^*$. Each fibre of $h$ is a single orbit of the kernel of $\chi_h:\GD\to \CC^*$, and
    each fibre of $h^*$ is a single orbit of the kernel of $\chi_{h^*}$.  These two
    subgroups coincide because $\chi_{h^*}=(\chi_h)^{-1}$.  The map is equivariant:
    $\nabla h(\rho(g)x)=\rho^*(g)^{-1}\nabla h(x)$.  It follows that $\nabla h$ maps $D_t$
    surjectively onto a fibre of $h^*$. By lemma \ref{f*h}, this mapping is a local
    diffeomorphism. It is easy to check that it is 1-1.  Since $(h^*)^*=h$ and dual
    unitary coordinates are themselves unitary, the same argument, interchanging the roles
    of $h$ and $h^*$, gives (ii).
  \end{proof}
\begin{question}
If we drop the condition that $D$ be a {\it linear} free divisor, what condition
    could replace reductivity to guarantee that for (linear) functions $f\in\OO_{\CC^n}$,
    $\s R_D$-finiteness implies $\s R_h$ finiteness?
\end{question}%@@@@@@@@@@@@@@@@@@@@@@@@@@@@@@@@@@@@@@@@@@
  \begin{remark}\label{aserendip}
    The following will be used in the proof of lemma \ref{Mond:tame2}. Let $AT_xD_t:=x+T_xD_t$ denote the affine tangent space at $x$. Proposition \ref{serendip}
    implies that the affine part $D_t^\vee=\{AT_xD_t: x\in D_t\}$ of the projective dual
    of $D_t$ is a Milnor fibre of $h^*$. This is because $AT_xD_t$ is the set
    \begin{displaymath}
      (y_1,\ld,y_n)\in \CC^n:d_xh(y_1,\ld,y_n)=
      d_xh(x_1,\ld,x_n);
    \end{displaymath}
    by homogeneity of $h$ the right hand side is
    just $nt$, and thus in dual projective coordinates $AT_xD_t$ is the point $(-nt:\p
    h/\p x_1(x):\cd:\p h/\p x_n(x))$.  In affine coordinates on $U_0$, this is the point
    \begin{displaymath}
      \left(\frac{-1}{nt}\frac{ \p h}{\p x_1}(x),\cd,\frac{-1}{nt}
        \frac{\p h}{\p x_n}(x)\right).
    \end{displaymath}
    By (\ref{f**}), the function $h^*$ takes the value
    $b_0t^{n-1}/(nt)^n=b_0/n^nt$ at this point, independent of $x\in D_t$,
    and so $D_t^\vee \subset (h^*)^{-1}(b_0/n^nt)$. The opposite inclusion
    holds by openness of the map $\nabla h$, which, in turn,
    follows from lemma
    \ref{hessian}.
  \end{remark}
%%%%%%%%%%%%%%%%%%%%%%%%%%%%%%%%%%%%%%%%%%%%%%%%%%%%%%%%%%%%%%%%%%%%%%%%%%%%%%%%%%%%
\subsection{Tameness}
\label{subsec:tame}

In this subsection, we study a property of the polynomial functions $f_{|D_t}$ known as
tameness. It describes the topological behaviour of $f$ at infinity, and is needed in order
to use the general results from \cite{Sa2} and \cite{DS} on the Gau{\ss}-Manin system and
the construction of Frobenius structures.  In fact we discuss two versions,
cohomological tameness and $M$-tameness.  Whereas the first will be seen to hold for
all $\mathcal{R}_h$-finite linear functions on a linear free divisor $D$, we show
$M$-tameness only if $D$ is reductive. Cohomological tameness is all that is
needed in our later construction of Frobenius manifolds, but we feel that the more
evidently geometrical condition of $M$-tameness is of independent interest.
\begin{definition}[\cite{Sa2}]
  Let $X$ be an affine algebraic variety and $f:X\rightarrow \CC$ a regular function.
  Then $f$ is called cohomologically tame if there is a partial compactification
  $X\stackrel{j}{\hookrightarrow} Y$ with $Y$ quasiprojective, and a proper regular
  function $F:Y\rightarrow \CC$ extending $f$, such that for any $c\in\CC$, the complex
  $\varphi_{F-c}(\mathbb{R}j_*\QQ_X)$ is supported in a finite number of points, which are
  contained in $X$. Here $\varphi$ is the functor of vanishing cycles of Deligne, see,
  e.g., \cite{Di}.
\end{definition}
It follows in particular that a cohomologically tame function $f$ has at most isolated
critical points.
\begin{proposition}\label{propCohomTame}
  Let $D\subset V$ be linear free and $f\in \CC[V]_1$ be an $\mathcal{R}_h$-finite linear
  section.  Then the restriction of $f$ to $D_t:=h^{-t}(t)$, $t\neq 0$ is cohomologically tame.
\end{proposition}
\begin{proof}
  A similar statement is actually given without proof in \cite{NS} as an example of a
  so-called weakly tame function.  We consider the standard graph compactification of $f$:
  Let $\overline{\Gamma}(f)$ be the closure of the graph $\Gamma(f)\subset D_t\times \CC$ of
  $f$ in $\overline{D}_t\times \CC$ (where $\overline{D}_t$ is the projective closure of
  $D_t$ in $\PP^n$), we identify $f$ with the projection $\Gamma(f)\to\CC$, and extend $f$
  to the projection $F:\overline{\Gamma}(f)\to\CC$.  Refine the canonical Whitney
  stratification of $\overline{D}_t$ by dividing the open stratum, which consists of
  $D_t\cup (B_h)_{\mbox{\scriptsize reg}}$, into the two strata $D_t$ and
  $(B_h)_{\mbox{\scriptsize reg}}$. Here $B_h=\{(0,x_1,\ldots,x_n)\in\PP^n\,|\,h(x_1,\ldots,x_n)=0\}$.
  Evidently this new stratification $\s S$ is still
  Whitney regular.  From $\s S$ we obtain a Whitney stratification $\s S'$ of
  $\overline{\Gamma}(f)$, since $\overline{\Gamma}(f)$ is just the transversal
  intersection of a hyperplane with $\overline{D}_t\times\CC$.  The isosingular locus of
  $\overline{D}_t$ through any point $(0:x_1:\ld:x_n)\in B_h$ contains the projectivised
  isosingular locus of $D$ through $(x_1,\ld,x_n)$, and so by the $\s R_h$-finiteness of
  $f$, $\{f=0\}$ is transverse to the strata of $\s S$.  This translates into the fact
  that the restriction of $F$ (i.e., the second projection) to the strata of the
  stratification $\s S'$ (except the stratum over $D_t$) is regular. It then follows from
  \cite[proposition 4.2.8]{Di} that the cohomology sheaves of
  $\varphi_{F-c}(\mathbb{R}j_*\QQ_{D_t})$ are supported in $D_t$ in a finite number of
  points, namely the critical points of $f_{|D_t}$. Therefore $f$ is cohomologically tame.
\end{proof}

\begin{definition}
  [\cite{NS}]
  \label{Mond:tame}
  Let $X\subset \CC^n$ be an affine algebraic variety and $f:X\to\CC$ a regular
  function. Set
  \begin{displaymath}
    M_{f}:=\{x\in X:f^{-1}(f(x))\ \not\hskip -.02in\tr \ S_{\Vert x \Vert}\},
  \end{displaymath}
  where $S_{\Vert x\Vert}$ is the sphere in $\CC^n$ centered at $0$ with radius
  $\Vert{}x\Vert$. We say that $f:X\rightarrow\CC$ is {\em $M$-tame} if there is no
  sequence $(\xk)$ in $M_f$ such that\\

\begin{enumerate}
  \addtocounter{enumi}{0}
  \item
The sequence $\Vert{}\xk\Vert{}$ tends to  infinity  as $k\to\infty$,
\item
The sequence $f(\xk)$ tends to a limit $\ell\in\CC$ as $k\to\infty$.
\end{enumerate}
\end{definition}
Suppose $\xk$ is a sequence in $M_f$ satisfying (i) and (ii).  After
passing to a subsequence, we
may suppose also that as $k\to\infty$,
\begin{enumerate}
  \addtocounter{enumi}{2}
\item $(\xk)\to x^{(0)}\in H_\infty$, where $H_\infty$ is the hyperplane at infinity in
  $\PP^n$,
\item $T^{(k)}\to T^{(0)}\in G_{d-1}(\PP^n)$ where $T^{(k)}$ denotes the {\it
    affine} tangent space $AT_{x^{(k)}}f^{-1}(f(x^{(k)})$, $d=\dim\,X$
and $G_{d-1}(\PP^n)$ is the Grassmannian of $(d-1)$-planes
in $\PP^n$.
\end{enumerate}
Let $f$ and $h$ be homogeneous polynomials on $\CC^n$ and $X=D_t=h^{-1}(t)$
for some $t\neq 0$. As before, let
%\begin{align*}
$$  B_f=\{(0,x_1,\ld,x_n)\in\PP^n:f(x_1,\ld,x_n)=0\}\quad
  B_h=\{(0,x_1,\ld,x_n)\in\PP^n:h(x_1,\ld,x_n)=0\}.
%\end{align*}
$$
Note that $B_f$ and $B_h$ are contained in the
projective closure of every affine fibre of $f$ and $h$ respectively.  We continue to
denote the restriction of $f$ to $D_t$ by $f$.  Let $\xk$ be a sequence satisfying
\ref{Mond:tame}(i)--(iv).
\begin{lemma}
  $\xo\in B_f\cap B_h$.
\end{lemma}
\begin{proof}
  Evidently $\xo\in \overline{D}_t\cap H_{\infty}=B_h$.
Let $U_1= \{(x_0,x_1,\ld,x_n)\in\PP^n:x_1\neq 0\}$.
After permuting the coordinates $x_1,\ld\,,x_n$ and
passing to a subsequence we may assume
  that $|\xk_1|\geq |\xk_j|$ for $j\geq 1$.  It follows that
  $x^{(0)}\in U_1$. In local coordinates $y_0=x_0/x_1,y_2=x_2/x_1,\ld,y_n=x_n/x_1$ on
  $U_1$, $B_f$ is defined by the two equations $y_0=0, f(1,y_2,\ld,y_n)=0$.  Since
  $f(\xk_1,\ld,\xk_n)\to\ell$ and $\xk_1\to\infty$, we have
  $f(1,\xk_2/\xk_1,\ld,\xk_n/\xk_1)\to 0$.  It follows that $f(\xo_1,\ld,\xo_n)=0,$ and
  $\xo\in B_f$.
\end{proof}
\begin{lemma}
  \label{Mond:tame3} If $f$ is a linear function then
  $T^{(0)}=B_f$.
%the limiting affine tangent space $T^{(0)}$ is equal to $B_f$.
\end{lemma}
\begin{proof}
  For all $k$ we have $\Tk\subset
  AT_{\xk}S_{\Vert{}\xk\Vert{}}$. Let $x^\bot$ denote the Hermitian orthogonal complement
  of the vector $x$ in $\CC^n$.  Then $\Tk$ is contained in $(\xk+{\xk}^\bot)\cap
  AT_{\xk}D_t$, since this is the maximal complex subspace of $AT_{\xk}D_t\cap
  AT_{\xk}S_{\Vert{}\xk\Vert{}}$.  With respect to dual homogeneous coordinates on
  $(\PP^n)^\vee$, $\xk+{\xk}^\bot= (-\Vert{}\xk\Vert{}^2:\xk_1:\ \cd\ :\xk_n)$. So
  \begin{displaymath}
    \lim_{k\rightarrow\infty}\xk+{\xk}^\bot=\lim_{k\rightarrow\infty}
    (1:\xk_1/\Vert{}\xk\Vert{}^2:\ \cd\ :\xk_n/\Vert{}\xk\Vert{}^2)=(1:0:\ \cd\
    :0)=H_{\infty}.
  \end{displaymath}
  It follows that $T^{(0)}\subset H_\infty$.
  To see that $T^{(0)}\subset B_f$, note that $T^{(k)}\subset f^{-1}(f(\xk))$. Since
  $f(\xk)\to\ell$, $f^{-1}(f(\xk))\to f^{-1}(\ell)$ and so in the limit $T^{(0)}\subseteq
  \overline{f^{-1}(\ell)}$. Since $T^{(0)}\subset H_\infty$, we conclude that
  $T^{(0)}\subset \overline {f^{-1}(\ell)}\cap H_\infty=B_f$. As $\dim\,B_f=\dim\,
  T^{(0)}$, the two spaces must be equal.
\end{proof}
By passing to a subsequence, we may suppose that
$AT_{\xk}D_t$ tends to a limit $L$ as $k\to\infty$.
\begin{lemma}\label{Mond:tame2}
  If $D=h^{-1}(0)$ is a reductive linear free divisor then
%  $D_t=h^{-1}(t)$ for $t\neq 0$, $(\xk)$ is a sequence in $D_t$
%  as in \ref{Mond:tame}(i)-(iv),
%  tending to $x^{(0)}\in H_\infty$, and $L=\lim_{k\rightarrow\infty}AT_{\xk}D_t$,
 $L\neq H_\infty$.
\end{lemma}
\begin{proof}
  It is only necessary to show that $H_{\infty}$ does not lie in the projective closure of
  the dual $D_t^\vee$ of $D_t$. By Remark \ref{aserendip}, $D_t^{\vee}=(h^*)^{-1}(c)$ for
  some $c\neq 0$. Its projective closure is thus
   $\{(y_0:y_1:\ \cd\ :y_n)\in (\PP^n)^\vee:h^*(y_1,\ld,y_n)=cy_0^n\}$,
  which does not contain $H_\infty=(1:0:\ \cd\ :0)$.
\end{proof}
Let $\{X_{\a}\}_{\a\in A}$ be a
Whitney stratification of $\overline{D}_t$, with regular stratum $D_t$, and suppose
$\xo\in X_\a$.
By Whitney regularity, $L\supset AT_{\xo}X_{\a}$.  Clearly $T^{(0)}\subset L$.  As $L\neq
H_\infty$ then since $T^{(0)}\subset H_{\infty}$, for dimensional reasons we must have
$T^{(0)}= L\cap H_{\infty}$.  It follows that $T^{(0)}\supset AT_{\xo}X_\a$, and thus, by
Lemma \ref{Mond:tame3},
\begin{displaymath}
  B_f\supset AT_{\xo}X_\a.
\end{displaymath}
We have proved
\begin{proposition}
  \label{Mond:tame4} If $D=\{h=0\}$ is a reductive linear free divisor, $D_t=h^{-1}(t)$
  for $t\neq 0$,
  and $f:D_t\to \CC$ is not $M$-tame, then $B_f$ is not transverse to the Whitney
stratification
$\{X_{\a}\}_{\a\in A}$ of $\overline{D}_t$.\eop
\end{proposition}
Now we can prove the result concerning M-tameness of (reductive) linear free divisors.
\begin{theorem}
  \label{theoMTame}
  If $D$ is a reductive linear free divisor with
  homogeneous equation $h$, and if the linear function $f$ is $\s R_h$-finite, then the
  restriction of $f$ to $D_t$, $t\neq 0$ is $M$-tame.
\end{theorem}
\begin{proof}
  $\s R_h$-finiteness of $f$ implies that for all $x\in D\cap \{f=0\}\ssm\{0\}$,
  \begin{equation}
    \label{Mond:tame5}
    T_x\{f=0\}+\Der(-\log h)(x)=T_x\CC^n.
  \end{equation}
  The strata of the canonical Whitney stratification $\s S$
  (\cite{tei} and \cite[Corollary 1.3.3]{tei2}) of $D$
  are unions of isosingular loci. So for any $x\in X_{\a}\in\s S$, $T_xX_\a\supset
  \Der(-\log D)(x)$.
  It follows from
  (\ref{Mond:tame5}) that $\{f=0\}\,\tr\,\s S$.  Because $D$ is
  homogeneous, the strata of $\s S$ are homogeneous too, and so we may form the projective
  quotient stratification $\PP\s S$ of $B_h$.  Transversality of $\{f=0\}$ to $D$ outside
  $0$ implies that $B_f$ is transverse to $\PP\s S$.  The conclusion follows by
  Proposition \ref{Mond:tame4}.
\end{proof}
\begin{remark}
  Reductivity is needed in Lemma \ref{Mond:tame2} to conclude that $L\neq
  H_{\infty}$. Indeed, consider the example given by Broughton in \cite[Example 3.2]{brou} of a
  non-tame function on $\CC^2$, defined as $g(x_1,x_2)= x_1(x_1x_2-1)$. Homogenising this equation, we obtain
  $h(x_1,x_2,x_3)=x_1(x_1x_2-x_3^2)$, which is exactly the defining equation
  of the non-reductive linear free divisor \eqref{eqNonReductiveDim3}.
The sequence $\xk=(1/k,k^2,\sqrt{2k})$
  lies in $D_{-1}$ and tends to $x^{(0)}=(0:0:1:0)$ in $\PP^3$, and
  $AT_{\xk}D_{-1}$ has dual projective coordinates $(3:0:1/n^2:-2n^{-1/2})$ and thus tends
  to $H_\infty$ as $n\to\infty$.

  Notice that M-tameness might also hold for $\mathcal{R}_h$-finite linear functions for non-reductive
  linear free divisors, but, as just explained, the above proof does not apply.
\end{remark}

%%%%%%%%%%%%%%%%%%%%%%%%%%%%%%%%%%%%%%%%%%%%%%%%%%%%%%%%%%%%%%%%%%
%\clearpage
\section{Gau{\ss}-Manin systems and Brieskorn lattices}
\label{sec:GMBrieskorn}

In this section we introduce the family of Gau{\ss}-Manin systems and Brieskorn lattices
attached to an $\mathcal{R}_h$-finite linear section of the fibration defined by the
equation $h$. All along this section, we suppose that $h$ defines a linear free divisor.

Under this hypothese, we show the freeness of the Brieskorn lattice, and prove that a particular
basis can be found yielding a solution of the so-called Birkhoff problem.
The proof of the freeness relies on two facts, first, we need that the deformation
algebra $T^1_{\s R_h/\CC}f$ is generated by the powers of $f$ (this would follow
only from the hypotheses of lemma \ref{mond:lst1}) and second on a division
theorem, whose essential ingredient is lemma \ref{trick} below, which in turn
uses the relative logarithmic de Rham complex associated to a linear free divisor which was studied in subsection \ref{subsec:RelLogdeRham}.
The particular form of the connection that we obtain on the Brieskorn lattice allows us
to prove that a solution to the Birkhoff problem always exists. This solution defines an
extension to infinity (i.e., a family of trivial algebraic bundles on $\PP^1$) of the Brieskorn
lattice. However, these solutions miss a crucial property needed in the next section: The
extension is not compatible with the canonical $V$-filtration at infinity, in other words,
it is not a $V^+$-solution in the sense of \cite[Appendix B]{DS}. We provide a very
explicit algorithm to compute these $V^+$-solutions.  In particular, this gives the
spectral numbers at infinity of the functions $f_{|D_t}$.

Using the tameness of the functions $f_{|D_t}$ it is shown in \cite{Sa2} that the Gau\ss-Manin systems
are equipped with a non-degenerate pairing with a specific pole order property on the Brieskorn lattices. A solution
to the Birkhoff problem compatible with this pairing is called $S$-solution in
\cite[Appendix B]{DS}. One needs such a solution in order to construct
Frobenius structures, see the next section. We prove that our solution
is a $(V^+,S)$-solution  under an additional hypothesis,
which is nevertheless satisfied for many examples.

Let us start by defining the two basic objects we are interested in this section. We
recall that we work in the algebraic category.
\begin{definition}
  \label{GMBrieskorn}
  Let $D$ be a linear free divisor with defining equation $h\in\CC[V]_n$ and
  $f\in\CC[V]_1$ linear and $\mathcal{R}_h$-finite. Let
  \begin{displaymath}
    \mathbf{G} := \frac{\Omega^{n-1}(\log\,h)[\tau,\tau^{-1}]}{(d-\tau df\wedge)(\Omega^{n-2}(\log\,h)[\tau,\tau^{-1}])}
  \end{displaymath}
  be the family of {\em algebraic} Gau{\ss}-Manin systems of $(f,h)$ and
  \begin{displaymath}
    G := \textup{Image of }\Omega^{n-1}(\log h)[\tau^{-1}]\mbox{ in }\mathbf{G} \quad =
    \frac{\Omega^{n-1}(\log\,h)[\tau^{-1}]}{(\tau^{-1}d-df\wedge)(\Omega^{n-2}(\log\,h)[\tau^{-1}])}
  \end{displaymath}
  be the family of {\em algebraic} Brieskorn lattices of $(f,h)$.
\end{definition}
\begin{lemma}
  \label{BrieskornFree}
  $\mathbf{G}$ is a free $\CC[t,\tau, \tau^{-1}]$-module of rank $n$, and $G$ is free over
  $\CC[t,\tau^{-1}]$ and is a lattice inside $\mathbf{G}$, i.e.,
  $\mathbf{G}=G\otimes_{\CC[t,\tau^{-1}]}\CC[t,\tau, \tau^{-1}]$. A $\CC[t,\tau,\tau^{-1}]$-basis of $\mathbf{G}$
  (resp. a $\CC[t,\tau^{-1}]$-basis of $G$) is given by $(f^i \alpha)_{i\in\{0,\ldots,n-1\}}$, where
  $\alpha:=n\cdot\vol/dh=\iota_E(\vol/h)$.
\end{lemma}
\begin{proof}
  As it is clear that $\mathbf{G}=G\otimes\CC[t,\tau, \tau^{-1}]$, we only have to show
  that the family $(f^i \alpha)_{i\in\{0,\ldots,n-1\}}$ freely generates $G$. This is done along the line of \cite[proposition 8]{ignacio3}.
  Remember from the discussion in subsection \ref{subsec:RelLogdeRham} that $\Omega^{n-1}(\log
  h)$ is $\CC[V]$-free of rank one, generated by the form
  $\alpha$. If we denote, as before, by $\xi_1,\ldots,\xi_n$ a linear
  basis of $\Der(-\log\,h)$, then we have that
  \begin{displaymath}
    G/\tau^{-1}G \cong \frac{\Omega^{n-1}(\log h)}{df\wedge\Omega^{n-2}(\log h)} \cong
    \left(h_*\mathcal{T}^1_{\s R_h/\CC}f\right)\alpha = \left(\frac{\CC[V]}{\xi_1(f),\ldots,\xi_{n-1}(f)}\right)\alpha
  \end{displaymath}
  which is a graded free $\CC[t]$-module of rank $n=\deg(h)$ by proposition \ref{numcons}
  and proposition \ref{fincon}.  Let $1,f,f^2,\ldots, f^{n-1}$ be the homogeneous
  $\CC[t]$-basis of $h_*\mathcal{T}^1_{\s R_h/\CC}f$ constructed in proposition
  \ref{fincon}, and $\omega=g\alpha$ be a representative for a section $[\omega]$ of $G$,
  where $g\in\CC[V]_l$ is a homogeneous polynomial of degree $l$. Then $g$ can be written
  as $g(x)=\widetilde{g}(h)\cdot f^i+\eta(f)$ where $\widetilde{g}\in\CC[t]_{\lfloor
    l/n\rfloor}$, $i=l\textup{ mod } n$ and $\eta\in\Der(-\log h)$. Using the basis
  $\xi_1,\ldots,\xi_{n-1}$, we find homogeneous functions
  $k_j\in\CC[V]_{l-1}$, $j=1,\ldots,n-1$ such that
  \begin{displaymath}
    \omega=\widetilde{g}(h) f^i \alpha + \sum_{j=1}^{n-1} k_j\xi_j(f)\alpha
  \end{displaymath}
  It follows from the next lemma that in $\mathbf{G}$ we have
  \begin{displaymath}
    [\omega]=\widetilde{g}(h) f^i\alpha+ \tau^{-1}\sum_{j=1}^{n-1}
    \left(\xi_j(k_j)+\operatorname{trace}(\xi_j)\cdot k_j\right)\alpha
  \end{displaymath}
  As $\deg(\xi_j(k_j)+\operatorname{trace}(\xi_j)\cdot k_j)=\deg(g)-1$, we see by
  iterating the argument (i.e., applying it to all the classes
  $\left[(\xi_j(k_j)+\operatorname{trace}(\xi_j)\cdot k_j)\alpha\right]\in G$) that
  $(f^i\alpha)_{i=0,\ldots,n-1}$ gives a system of generators for $G$ over
  $\CC[t,\tau^{-1}]$.

	To show that they freely generate, let us consider a
  relation
  \begin{displaymath}
    \sum_{j=0}^{n-1} a_j(t,\tau^{-1}) f^j\alpha = (d-\tau df\wedge)
    \sum_{i=l}^L \tau^i \omega_i,~ \omega_i\in\Omega^{n-2}(\log h)
  \end{displaymath}
  where $l\leq L\leq 0$. Rewriting the left-hand side as a polynomial in $\tau^{-1}$ with
  coefficients in $\Omega^{n-1}(\log h)$, the above equation becomes
  \begin{equation}
    \label{Ig:eq:11}
    \sum_{i=l}^{L+1} \tau^{i} \eta_i = (d-\tau df\wedge) \sum_{i=l}^L \tau^i \omega_i
  \end{equation}
  where we have written $\eta_i = \sum_{j=0}^{n-1} b_{ij}(t)(f^j\alpha)$.
	It follows that $\eta_{L+1}=df\wedge \omega_L\in df\wedge\Omega^{n-2}(\log\,h)$.
	Since $(f^j\alpha)_{j=0,\ldots,n-1}$ form a $\CC[t]$-basis of the quotient $\Omega^{n-1}(\log
  h)/df\wedge\Omega^{n-2}(\log h)$, it follows that $b_{{L+1},j}=0$ for $j=0,\dots,n-1$.
	Hence $\eta_{L+1}=0$, and we see by descending induction on $L$ that $\eta_i=0$ for
	any $i\in\{l,\ldots,L+1\}$. This shows $a_j=0$ for all $j=0,\ldots,n-1$, so that
	the relation is trivial, showing the $\CC[t,\tau^{-1}]$-freeness of $G$.
\end{proof}
\begin{lemma}
  \label{trick}
  For any $\xi$ in $\Der(-\log h)_0$ and $g\in\CC[V]$, the following relation holds in
  $\mathbf{G}$
  \begin{displaymath}
    \tau g \xi(f)\alpha = \left(\xi(g)+ g \cdot\operatorname{trace} (\xi)\right)\alpha
  \end{displaymath}
\end{lemma}
\begin{proof}
  We have that
  \begin{align*}
    \tau g \xi(f)\alpha &= \tau g i_\xi(df)\alpha =\tau g (i_\xi(df\wedge\alpha)+df\wedge
    i_\xi\alpha)= \tau g df\wedge i_\xi\alpha \\
    &= d(gi_\xi \alpha) = dg \wedge i_\xi \alpha + gd i_\xi \alpha =i_\xi(dg \wedge \alpha) + dg \wedge i_\xi \alpha  + gd i_\xi \alpha \\
    &=i_\xi(dg)\alpha
    + gd i_\xi \alpha=\xi(g)\alpha + gd i_\xi \alpha \\
    &=\left(\xi(g)+ g \cdot \operatorname{trace}(\xi)\right)\alpha.
  \end{align*}
  In this computation, we have twice used the fact that for any function $r\in\CC[V]$, the
  class $i_\xi(dr\wedge\alpha)$ is zero in $\Omega^{n-1}(\log h)$. This holds because for
  $\xi\in \Der(-\log h)$ and for $r\in\CC[V]$ the operations $i_\xi$ and $dr\wedge$ are
  well defined on $\Omega^\bullet(\log h)$ and moreover, $\Omega^n(\log h)=0$, so that
  already $dr\wedge \alpha=0\in\Omega^\bullet(\log h)$.
\end{proof}
We denote by $T:=\textup{Spec}\,\CC[t]$ the base of the family defined by $h$.  Then
$\mathbf{G}$ corresponds to a rational vector bundle of rank $n$ over $\PP^1\times T$,
with poles along $\{0,\infty\}\times T$. Here we consider the two standard charts of
$\PP^1$ where $\tau$ is a coordinate centered at infinity. The module $G$ defines an extension over
$\{0\}\times T$, i.e., an algebraic bundle over $\CC\times T$ of the same rank as
$\mathbf{G}$.

We define a (relative) connection operator on $\mathbf{G}$ by
\begin{displaymath}
  \nabla_{\partial_\tau}\left(\sum_{i=i_0}^{i_1} \omega_i\tau^i\right):=\sum_{i=i_0-1}^{i_1}
  \left((i+1)\omega_{i+1}-f\cdot\omega_i\right)\tau^i
\end{displaymath}
where $\omega_{i_1+1}:=0, \omega_{i_0-1}:=0$. Then it is easy to check that this gives a
well defined operator on the quotient $\mathbf{G}$, and that it satisfies the Leibniz
rule, so that we obtain a relative connection
\begin{displaymath}
  \nabla:\mathbf{G}\longrightarrow\mathbf{G}\otimes\Omega^1_{\CC \times T /T}(*\{0\}\times T).
\end{displaymath}
As multiplication with $f$ leaves invariant the module $\Omega^{n-1}(\log h)$, we see that
the operator $-\nabla_{\partial_\tau}$ sends $G$ to itself, in other words, $G$ is stable
under
$-\nabla_{\partial_\tau}=\tau^{-2}\nabla_{\partial_{\tau^{-1}}}=\theta^2\nabla_{\partial_\theta}$,
where we write $\theta:=\tau^{-1}$. This shows that the relative connection $\nabla$ has a
pole of order at most two on $G$ along $\{0\}\times T$ (i.e., along $\tau=\infty$).

Consider, for any $t\in T$, the restrictions $\mathbf{G}_t:=\mathbf{G}/\mathfrak{m}_t
\mathbf{G}$ and $G_t:=G/\mathfrak{m}_t G$.  Then $\mathbf{G}_t$ is a free
$\CC[\tau,\tau^{-1}]$-module and $G_t$ is a $\CC[\tau^{-1}]$-lattice in it.  It follows
from the definition that if $t\neq 0$, this is exactly the (localized partial Fourier-Laplace
transformation of the) Gau{\ss}-Manin system (resp. the Brieskorn lattice) of the function
$f:D_t\rightarrow \CC$, as studied in \cite{Sa2}. We will make use of the results of
loc.cit. applied to $f_{|D_t}$ in the sequel. Let us remark that the freeness of the
individual Brieskorn lattices $G_t$ (and consequently also of the Gau{\ss}-Manin systems
$\mathbf{G}_t$) follows from the fact that $f_{|D_t}$ is cohomological tame (\cite[theorem
10.1]{Sa2}). In our situation we have the stronger statement of lemma \ref{BrieskornFree},
which gives the $\CC[\tau^{-1},t]$-freeness of the whole module $G$.

Our next aim is to consider the so-called Birkhoff problem, that is, to find a basis
$\underline{\omega}^{(1)}$ of $G$ such that the connection take the particularly simple
form
\begin{displaymath}
  \partial_\tau(\underline{\omega}^{(1)}) =
  \underline{\omega}^{(1)}\cdot\left(\Omega_0+\tau^{-1} A_\infty\right),
\end{displaymath}
(from now on, we write $\partial_\tau$ instead of $\nabla_{\partial_\tau}$ for short)
where we require additionally that $A_\infty$ is diagonal. We start with the basis
$\underline{\omega}^{(0)}$, defined by
\begin{equation}\label{eqOmega0}
  \omega^{(0)}_i:=(-f)^{i-1}\cdot \alpha
  \quad\forall i\in\{1,\ldots,n\}.
\end{equation}
Then we have $\partial_\tau(\omega^{(0)}_i)=\omega^{(0)}_{i+1}$ for all
$i\in\{1,\ldots,n-1\}$ and
\begin{displaymath}
  \partial_\tau(\omega^{(0)}_n)=(-f)^n\alpha.
\end{displaymath}
As $\deg((-f)^n)=n$, $(-f)^n$ is a non-zero multiple of $h$ in the Jacobian algebra
$\CC[V]/(df(Der(-\log h)))$, so that we have an expression $(-f)^n=c_0\cdot h +
\sum_{j=1}^{n-1}d^{(1)}_j\xi_j(f)$, where $c_0\in\CC^*$, $d^{(1)}_j\in\CC[V]_{n-1}$. This
gives by using lemma \ref{trick} again that
\begin{displaymath}
  \partial_\tau(\omega^{(0)}_n)=(-f)^n\alpha =
  \left(c_0 t+\sum_{j=1}^{n-1}d^{(1)}_j \xi_j(f)\right)\alpha=
  \left(c_0 t+\tau^{-1}\sum_{j=1}^{n-1} \left(\xi_j(d^{(1)}_j)+\operatorname{trace}(\xi_j)\right)\right)\alpha.
\end{displaymath}
As $\deg\left(\xi_j(d^{(1)}_j)+\operatorname{trace}(\xi_j)\right)=n-1$, there exist
$c_1\in\CC$ and $d^{(2)}_r \in\CC[V]_{n-2}$ such that
\begin{align*}
    \left(\sum_{j=1}^{n-1}
      \left(\xi_j(d^{(1)}_j)+\operatorname{trace}(\xi_j)\right)\right)\alpha
    &=  \left(c_1 (-f)^{n-1}+\sum_{r=1}^{n-1}d^{(2)}_r\xi_r(f)\right)\alpha \\
    &= \left(c_1 (-f)^{n-1}+ \tau^{-1}\sum_{r=1}^{n-1}\left(\xi_r(d^{(2)}_r)
        +\operatorname{trace}(\xi_r)d^{(2)}_r\right)\right)\alpha,
\end{align*}
and $\deg\left(\xi_r(d^{(2)}_r)+\operatorname{trace}(\xi_r)d^{(2)}_r\right)=n-2$. We see
by iteration that the connection operator $\partial_\tau$ takes the following form with
respect to $\underline{\omega}^{(0)}$:
\begin{equation}
  \label{Ig:eq:1}
  \partial_\tau(\underline{\omega}^{(0)}) =
  \underline{\omega}^{(0)}
  \cdot
  \begin{pmatrix}
    0&0&\ldots &0&c_0t+c_n\tau^{-n}\\
    1   &0&\ldots&0&c_{n-1}\tau^{-n+1}\\
    \hdotsfor{5}\\
    0 &0 & \ldots &0& c_2\tau^{-2}\\
    0&0& \ldots& 1&c_1\tau^{-1}
  \end{pmatrix}=:\underline{\omega}^{(0)}\cdot \Omega =:\underline{\omega}^{(0)}\cdot
  \left(\sum_{k=0}^{n}\Omega_{k} \tau^{-k}\right).
\end{equation}
Notice that if $D$ is special then $c_n=0$, i.e., $\Omega_n=0$.

The matrix
$\Omega_0$ has a very particular form, due to the fact that the jacobian algebra
$h_*\mathcal{T}^1_{\mathcal{R}_h/\CC} f$ is generated by the powers of $f$.
Notice also that the restriction $(\Omega_0)_{|t=0}$ is nilpotent, with a single
Jordan block with eigenvalue zero. This reflects the fact that
$(G_0,\nabla)$ is \emph{regular singular} at $\tau=\infty$, which is not the
case for any $t\neq 0$. Remember that although $D$ is singular itself,
so that it is not quite true that there is only one critical value of $f$ on $D$,
we have that $f$ is regular on $D\backslash\{0\}$ in the stratified sense (see
the proof of proposition \ref{propCohomTame}).

The particular form of the matrix $\Omega_0$
is the key ingredient to solve the Birkhoff problem, which can actually be done by a
triangular change of basis.
\begin{lemma}
  \label{lemBirkhoff}
  There exists a base change
  \begin{equation}
    \label{eqBirkhoff}
    \omega^{(1)}_j:=\omega_j^{(0)}+\sum_{i=1}^{j-1}b_i^j\tau^{-i}\omega_{j-i}^{(0)},
  \end{equation}
  such that the matrix of the connection with respect to $\underline{\omega}^{(1)}$ is
  given by
  \begin{displaymath}
    \Omega_0 + \tau^{-1}A_{\infty},
  \end{displaymath}
  where $A^{(1)}_\infty$ is diagonal. Moreover, if $D$ is special, then $b_i^{j}$ can be chosen such that
  $b_i^{i+1}=0$ for $i=1,\dots,n-1$.
\end{lemma}
\begin{proof}
  Let us regard $b_i^j$ as unknown constants to be determined and then let
  \begin{displaymath}
    B:=(b^j_{j-k}\tau^{k-j})_{kj}=:\sum_{i=0}^{n-1}B_i\tau^{-i}=\Id+\sum_{i=1}^{n-1} B_i\tau^{-i}.
  \end{displaymath}
  Here $b^i_j=0$ for $j<0$.  Notice that $B_i$ is a matrix whose only non-zero entries are
  in the position $(j,j+i)$ for $j=1,\dots,n-i$. %In particular, its first row is zero.

  The matrix of the action of $\partial_\tau$ changes according to the formula:
  \begin{equation}
    \label{Ig:eq:5}
    X := B^{-1}\cdot \frac{dB}{d\tau} + B^{-1}\Omega B =:\sum_{i=0}^{n} X_i\tau^{-i}.
  \end{equation}
  Multiplying by $B$ both sides of the above equation we find
  \begin{equation}
    \label{Ig:eq:6}
    BX=\sum_{i=0}^{n} \left(\sum_{j=0}^i B_jX_{i-j} \right)\tau^{-i}=\sum_{i=1}^{n}\left
    (-(i-1)B_{i-1} + \Omega_0 B_i+\Omega_i\right)\tau^{-i} + \Omega_0.
  \end{equation}
  where $B_{-1}:=0$. Let $N=(n_{ij})$ be the matrix with $n_{ij}=1$ if $j=i-1$ or $0$
  otherwise. Hence $\Omega_0=N+C_0$ where $C_0$ is the matrix whose only non-zero entry is
  $c_0t$ in the right top corner. It follows that $X_0=\Omega_0$ and that
  \begin{equation}
    \label{Ig:eq:7}
    X_i= -\sum_{j=1}^{i-1} B_jX_{i-j} - \left[ B_i,N \right] - (i-1)B_{i-1} +  \Omega_{i}.
  \end{equation}
  We are looking for a solution to the system $X_1=A^{(1)}_\infty$, $X_i=0$, $i=2,\dots,n$,
  where $A^{(1)}_\infty$ is diagonal with entries yet to be determined. In view of the above,
  this system is equivalent to:
  \begin{align}
    \label{Ig:eq:8}
    \begin{split}
      X_1&=-[B_1,N]+\Omega_1 = A^{(1)}_\infty,\\
      [B_{i+1},N]&=-B_{i}X_1-iB_i+\Omega_{i+1},~i=1,\dots,n-1.
    \end{split}
  \end{align}
  We are going to show that this system of polynomial equations in the variables $b^j_i$
  can always be reduced to a triangular system in $b^j_1$, so that there exists a
  solution. In particular, this determines the entries of the diagonal matrix
  $-[B_1,N]+\Omega_1$, i.e., the matrix $A^{(1)}_\infty$ we are looking for.

  A direct calculation shows that if we substitute the first equation of \eqref{Ig:eq:8}
  into the right hand side of the second one, we obtain
  $[B_{i+1},N]=B_{i}([B_1,N]-\Omega_1+i\Id)+\Omega_{i+1}=:P^i$, where the only non-zero
  coefficients of the matrix $P^i$ are $P^i_{j,i+j}$, namely:
  \begin{align}
    \label{Ig:eq:9}
    \begin{split}
      P^i_{j,i+j}&=b^{i+j}_i(b_1^{i+j+1}-b_1^{i+j}+i),j=1,\dots,n-i-1,\\
      P^i_{n-i,n}&=b_i^n(-b_1^n-c_1+i)+c_{i+1}.
    \end{split}
  \end{align}
  A matrix $B_{i+1}$ satisfying $[B_{i+1},N]=P^i$ exists if and only if
  $Q^{i}:=\sum_{j=1}^{n-i} P_{j,i+j}^{i}=0$, and if this is case, the solution is given by
  setting:
  \begin{equation}
    \label{Ig:eq:10}
    b_{i+1}^{i+k+1}=-\sum_{j=k+1}^{n-i} P^i_{j,j+i},~k=1,\dots,n-i-1.
  \end{equation}
  For $i=1$ and $j=1,\dots,n-1$, we have from \eqref{Ig:eq:9}
  \begin{displaymath}
    P_{j,j+1}^{1} =
    -(b_1^{j+1})^2+\text{lower degree terms in $b_1^{j+1}$ with coefficients in $b_1^k,k>j+1$}
  \end{displaymath}
  so that substituting \eqref{Ig:eq:10} in \eqref{Ig:eq:9} for $i=2$
  \begin{displaymath}
    P^2_{j,2+j}=-(b_1^{j+2})^3
    + \text{lower degree terms in $b_1^{j+2}$ with coefficients in $b_k^1,~k>j+2$}
  \end{displaymath}
  By induction we see that after substitution we have
  \begin{displaymath}
    P^i_{j,i+j}=-(b_1^{i+j})^{i+1}
    + \text{lower degree terms in $b_1^{i+j}$ with coefficients in $b^k_1,~k>i+j$}%}
  \end{displaymath}
  from which it follows that
  \begin{displaymath}
    Q^{i}=-(b_1^{i+1})^{i+1}
    + \text{lower degree terms in $b_1^{i+1}$ with coefficients in $b_1^k,~k>i+1$}
  \end{displaymath}
  The system $Q^{i}=0,i=1,\dots,n-1$ is triangular (e.g.,
  $Q^{n-1}\in\CC[b_1^n]$) and thus has a solution.

  In the case where $D$ is special, the vanishing of $\Omega_n$ can be used to
  set  $b_i^{i+1}=0$ from the start. The above proof then works verbatim.
\end{proof}
Notice that we can assume by a change of coordinates on $T$ that the non-zero constant
$c_0$ is actually normalized to $1$. We will make this assumption from now on.

%{\tiny
In the next section, we are interested in constructing Frobenius structures associated
to the tame functions $f_{|D_t}$ and to study their limit behaviour when $t$ goes to zero.
For that purpose, it is desirable to complete the relative connection $\nabla$ from above
to an absolute one, which will acquire an additional pole at $t=0$.
Although such a definition exists in general, we will give it in the
reductive case only. The reason for this is that in order to obtain an
explicit expression for this connection, we will need the special form of the relative connection in the basis $\underline{\omega}^{(1)}$ as
well as theorem \ref{nbi}, which is valid in the reductive case only.
It is, however, true that formula \eqref{horizontalconnection} defines an integrable connection on $\mathbf{G}$
in all cases, more precisely, it defines the (partial Fourier-Laplace transformation of the) Gau\ss-Manin connection
for the complete intersection given by the two functions $(f,h)$. We will not discuss this in detail here.

The completion of the relative connection $\nabla$ on $\mathbf{G}$ refered to above is given by the formula
\begin{equation}
  \label{horizontalconnection}
  \nabla_{\partial_t}(\omega)  :=  \frac{1}{n\cdot t} \left(L_E(\omega)- \tau L_E(f) \cdot \omega \right),
\end{equation}
for any $[\omega]\in\Omega^{n-1}(\log h)$ and extending $\tau$-linearly. One checks that
$$
(t \nabla_{\partial_t})\left((\tau^{-1}d- df\wedge)(\Omega^{n-2}(\log\,h)[\tau^{-1}])\right)\subset (\tau^{-1}d- df\wedge)(\Omega^{n-2}(\log\,h)[\tau^{-1}]),
$$
so that we obtain operator
\begin{equation}
\label{eqIntConnection}
  \nabla: G \longrightarrow G \otimes \tau\Omega^1_{\CC\times T}(\log \mathcal{D}),
\end{equation}
where $\mathcal{D}$ is the divisor $(\{0\} \times T) \cup (\CC \times \{0\})
\subset \CC \times T$.

%One readily checks that the restriction of $\mathbf{G}$
%to $\{\tau\neq\infty,t\neq 0\}$ is flat and we denote by $\mathbf{G}^\nabla$
%the corresponding local system and by $\mathbf{G}^\infty$ its space of
%multivalued flat sections.  Using the particular triangular form of
%\ref{eqBirkhoff} together with the results of Section~\ref{sec:Cohomology} we
%obtain the following result. The second statement and its proof were
%communicated to us by Claus Hertling.
\begin{proposition}\label{propVerticalConnectionReductive}
  Let $D$ be reductive.
Then:
  \begin{enumerate}
  \item The elements of the basis $\underline{\omega}^{(1)}$ constructed above can be
    represented by differential forms $\omega_i^{(1)}=[g_i\alpha]$ with $g_i$ homogeneous
    of degree $i=0,\dots,n-1$, i.e., by elements outside of $\tau^{-1}\Omega^{n-1}(\log
    h)[\tau^{-1}]$.
    \item
    The connection operator defined above is flat outside $\theta=0, t=0$.
    We denote by $\mathbf{G}^\nabla$ the corresponding local system and by $\mathbf{G}^\infty$ its space of
    multivalued flat sections.
	\item Consider the Gau\ss-Manin system, localized at $t=0$, i.e.
    $$
        \mathbf{G}[t^{-1}]:=\mathbf{G}\otimes_{\CC[\tau,\tau^{-1},t]}\CC[\tau,\tau^{-1},t,t^{-1}] \cong \frac{\Omega^{n-1}_{V/T}(*D)[\tau,\tau^{-1}]}{(d-\tau df\wedge)\Omega^{n-2}_{V/T}(*D)[\tau,\tau^{-1}]}.
    $$
    and similarly, the localized Brieskorn lattice
    $$
        G[t^{-1}]:=G\otimes_{\CC[\tau^{-1},t]}\CC[\tau^{-1},t,t^{-1}] \cong \frac{\Omega^{n-1}_{V/T}(*D)[\tau^{-1}]}{(\tau^{-1}d-df\wedge)\Omega^{n-2}_{V/T}(*D)[\tau^{-1}]}
        \subset \mathbf{G}[t^{-1}].
    $$
	Then $\underline{\omega}^{(1)}$ provides a solution to the Birkhoff problem
	for $(G[t^{-1}],\nabla)$ ``in a family'', i.e., an extension to a trivial algebraic bundle $\widehat{G}[t^{-1}]
    \subset \widetilde{i}_*G[t^{-1}]$ (here $\widetilde{i}:\CC\times (T\backslash\{0\}) \hookrightarrow \PP^1\times (T\backslash\{0\})$)
	on $\PP^1\times (T\backslash\{0\})$, on which the connection has a logarithmic pole along $\{\infty\}\times (T \backslash
        \{0\})$ and, as before, a pole of type one along $\{0\}\times
        (T\backslash \{0\})$ (remember that $\{0\}\times T=\{\theta=0\}$).
  \item Let $\gamma$ resp. $\gamma'$ be a small counterclockwise loop around the divisor
    $\{0\}\times T$ resp. $\CC\times\{0\}$ in $\CC\times T$. Let $M$ resp. $M'$ denote the
    mondromy endomorphisms on $\mathbf{G}^\infty$ corresponding to $\gamma$ resp.
    $\gamma'$. Then
    \begin{displaymath}
      M^{-1} = (M')^n.
    \end{displaymath}
  \item
   Let $u:\CC^2\rightarrow\CC\times T$, $(\theta,s) \mapsto (\theta,s^n)$. Consider
   the pullback $u^*(G,\nabla)$ and denote by $(\widetilde{G},\nabla)$ the restriction
   to $\CC\times\CC^*$ of the analytic bundle corresponding to $u^*(G,\nabla)$. Then
   $\widetilde{G}$ underlies a \emph{Sabbah-orbit} of TERP-structures, as defined in
   \cite[definition 4.1]{HS1}.
  \end{enumerate}
\end{proposition}
\begin{proof}
\begin{enumerate}
\item
It follows from theorem \ref{nbi} that for
  $g\in\CC[V]_i$ with $1< i < n$, the $n-1$-form $g\alpha$ is exact in the
  complex $\Omega^\bullet(\log
  h)$. Therefore in $\mathbf{G}$ we have $\tau^{-1} g\alpha =\tau^{-1} d\omega' = df\wedge
  \omega' =g'\alpha$ for some $\omega'\in\Omega^{n-2}(\log h)$ and $g'\in\CC[V]$. Note
  that necessarily $g'\in\CC[V]_{i+1}$. Moreover, in the above constructed base change
  matrix we had $B_{1l}=\delta_{1l}$ (as $D$ is reductive hence special), which implies that for all $i>0$, $\omega^{(1)}_i$
  is represented by an element in $f\CC[V]\alpha[\tau^{-1}]$, i.e, by a sum of terms of
  the form $\tau^{-k} g\alpha$ with $g\in\CC[V]_{\geq 1}$. This proves that we can
  successively erase all negative powers of $\tau$, i.e., represent all $\omega^{(1)}_i$,
  $i>0$ by pure forms (i.e., without $\tau^{-1}$'s), and $\omega^{(1)}_0 = \omega^{(0)}_0
  = \alpha$ is pure anyhow.

\item
From (i) and the definition of $\nabla_{\partial_t}$ in
  (\ref{horizontalconnection}) we obtain
  \begin{displaymath}
    \nabla(\underline{\omega}^{(1)}) = \underline{\omega}^{(1)}\cdot\left[(\Omega_0 + \tau^{-1}A^{(1)}_\infty)d\tau + (\diag(0,\ldots,n-1)+\tau\Omega_0+A^{(1)}_\infty)\frac{dt}{nt}\right].
  \end{displaymath}
  The flatness conditions of an arbitrary connection of the form
  $$
  \nabla(\underline{\omega}^{(1)}) = \underline{\omega}^{(1)}\cdot\left[(\tau A + B)\frac{d\tau}{\tau} + (\tau A'+B')\frac{dt}{t}\right]
  $$
  with $A, A' \in M(n\times n, \CC[t])$ and $B, B' \in M(n\times n, \CC)$ is given by the following system of equations:
  $$
  [A,A']=0\quad;\quad[B,B']=0\quad;\quad (t\partial_t) A - A' = [A,B']-[A',B]
  $$
  One checks that for $A=\Omega_0$, $A'=\frac1n\Omega_0$, $B=A^{(1)}_\infty$ and $B'=\frac1n(A^{(1)}_\infty+\diag(0,\ldots,n-1))$, these
  equations are satisfied.
  \item
  The extension defined by
  $\underline{\omega}^{(1)}$, i.e., $\widehat{G}[t^{-1}]:=\oplus_{i=1}^n \OO_{\PP^1\times
    T}[t^{-1}] \omega^{(1)}_i$ provides the solution in a family to
  the Birkhoff problem, i.e., we have $\nabla_{X} \widehat{G}[t^{-1}] \subset
  \widehat{G}[t^{-1}]$ for any $X\in\Der(-\log(\{\infty\}\times (T\backslash\{0\})))$.
\item

If we restrict $(\mathbf{G},\nabla)$ to the
  curve $C:=\left\{(\tau,t)\in(\CC^*)^2\,|\,\tau^nt=1\right\}$ we obtain
  \begin{displaymath}
    \nabla_{|C} = -\diag(0,\ldots,n-1) \frac{d\tau}{\tau}.
  \end{displaymath}
  As the diagonal this connection matrix are integers, the monodromy of $(\mathbf{G},\nabla)_{|C}$
  is trivial which implies the result (notice that the composition of $\gamma_1$ and
  $\gamma_2^n$ is homotopic to a loop around the origin in $C$).
   \item

    That the restriction to $\CC\times (T\backslash \{0\})$ of (the analytic bundle corresponding to)
    $G$ underlies a variation of pure polarized TERP-structures is a general fact, due to the
    tameness of the functions $f_{|D_t}$ (see \cite{Sa2} and \cite{Sa8}, \cite[theorem 11.1]{HS1}).
    Using the connection matrix from (ii), it is an easy calculation that $\nabla_{s\partial_s-\tau\partial_\tau}
    (\widetilde{\underline{\omega}}^{(1)})=0$,
    where $\widetilde{\underline{\omega}}^{(1)}:=u^*\underline{\omega}^{(1)} \cdot s^{-\diag(0,\ldots,n-1)}$
%    and $P$ is the upper triangular base change $P=(p_{ij})\in \Gl(n,\CC[\tau^{-1},t])$,
%    $$
%    p_{ij} = \left \{
%    \begin{array}{rcl}
%    \tau^{j-i} & \textup{if} & i\geq j\\
%    0 & \textup{else}
%    \end{array}
%    \right. ,
%    $$
    so that $(\widetilde{G},\nabla)$ satisfies condition 2.(a) in \cite[definition 4.1]{HS1}.
\end{enumerate}
\end{proof}
%}
For the purpose of the next section, we need to find a much more special solution to the
Birkhoff problem, which is called $V^+$-solution in \cite{DS}.  It takes into account the
Kashiwara-Malgrange filtration of $\mathbf{G}$ at infinity (i.e., at $\tau=0$). We briefly
recall the notations and explain how to construct the $V^+$-solution starting from our
basis $\underline{\omega}^{(1)}$.

Fix $t\in T$ and consider, as before, the restrictions $\mathbf{G}_t$ (resp. $G_t$)
of the family of Gau\ss-Manin-systems (resp. Brieskorn) lattices $\mathbf{G}$ resp. $G$.
As already pointed out, for $t\neq 0$, these are the Gau\ss-Manin-system resp. the Brieskorn
lattice of the tame of the function $f_{|D_t}$.  The meromorphic bundle $\mathbf{G}_t$ is known to be a holonomic left
$\CC[\tau]\langle\partial_\tau\rangle$-module, with singularities at $\tau=0$ and
$\tau=\infty$ only.  The one at infinity, i.e. $\tau=0$ is regular singular, but not
necessarily the one at zero (at $\tau=\infty$).  Similarly to the notation used above, we
have the local system $\mathbf{G}^\nabla_t$ and its space of multivalued global flat
sections $\mathbf{G}_t^\infty$.  Recall that for any $t\neq 0$, the monodromy of $\mathbf{G}^\nabla_t$ is
quasi-unipotent, so any logarithm of any of its eigenvalues is a rational number. As we will
see in section \ref{sec:examples}, the same is true in all examples for $t=0$, but
this is not proved for the moment. Let $\KK$ be either $\CC$ or $\QQ$, depending on whether $t=0$ or $t\neq 0$.
In the former case, we chose the lexicographic ordering on $\CC$ which extends the usual ordering of $\RR$.
Recall that there is a unique increasing exhaustive filtration $V_\bullet \mathbf{G}_t$
indexed by $\KK$, called the Kashiwara-Malgrange or canonical V-filtration on
$\mathbf{G}_t$ with the properties
\begin{enumerate}
\item It is a good filtration with respect to the $V$-filtration
  $V_\bullet\CC[\tau]\langle\partial_\tau\rangle$ of the Weyl-algebra, i.e. it satisfies
  $V_k \CC[\tau]\langle\partial_\tau\rangle V_l \mathbf{G}_t \subset V_{k+l} \mathbf{G}_t$
  and this is an equality for any $k\leq 0$, $l\leq -l_0$ and $k\geq 0$, $l\geq l_0$ for some
  sufficiently large positive integer $l_0$.
\item For any $\alpha\in\KK$, the operator $\tau\partial_\tau+\alpha$ is nilpotent on the
  quotient $\mathit{gr}^V_\alpha\mathbf{G}_t$
\end{enumerate}
We have an induced V-filtration on the Brieskorn lattice $G_t$, and we denote by
\begin{displaymath}
  \textup{Sp}(G_t,\nabla):=\sum_{\alpha\in\KK}\dim_{\CC}\left(\frac{V_\alpha \cap G_t}{V_{<\alpha} \cap G_t+\tau^{-1}G_t\cap V_\alpha}\right)\alpha\in\ZZ[\KK]
\end{displaymath}
the spectrum of $G_t$ at infinity (for $t\neq 0$ it is also called the spectrum at infinity associated to
$f_{|D_t}$).  We also write it as an ordered tuple of (possibly repeated) numbers
$\alpha_1\leq\ldots\leq\alpha_n$.  We recall the following notions from \cite[appendix
B]{DS}.
\begin{lemmadefinition}
  \label{Vsolution}
  \begin{enumerate}
  \item The following conditions are equivalent.
    \begin{enumerate}
    \item There is a solution to the Birkhoff problem, i.e, a basis $\underline{\omega}$ of $G_t$
      with $\partial_\tau(\underline{\omega})=
      \underline{\omega}(\Omega_0+\tau^{-1}A_\infty)$ (where $A_\infty$ is not
      necessarily semi-simple).
    \item There is a $\CC[\tau]$-lattice $G_t'$ of $\mathbf{G}_t$ which is stable under
      $\tau\partial_\tau$, and such that $G_t=(G_t\cap G'_t)\oplus\tau^{-1}G_t$.
    \item There is an extension to a free $\mathcal{O}_{\PP^1}$-module $\widehat{G}_t
      \subset \widetilde{i}_*G_t$ (where $\widetilde{i}:\CC \hookrightarrow \PP^1$) with the property
      that $(\tau\nabla_\tau)\widehat{G}_t\subset\widehat{G}_t$.
    \end{enumerate}
  \item A solution to the Birkhoff problem $G_t'$ is called a $V$-solution iff
    \begin{displaymath}
      G_t\cap V_\alpha\mathbf{G}_t = (G_t\cap G'_t \cap V_\alpha \mathbf{G}_t) \oplus (\tau^{-1}G_t\cap V_\alpha\mathbf{G}_t).
    \end{displaymath}
  \item It is is called a $V^+$-solution if moreover we have
    \begin{displaymath}
      (\tau\partial_\tau+\alpha)(G_t\cap G'_t \cap V_\alpha \mathbf{G}_t) \subset
      (G_t\cap G'_t \cap V_{<\alpha} \mathbf{G}_t) \oplus \tau(G_t\cap G'_t \cap V_{\alpha+1} \mathbf{G}_t).
    \end{displaymath}
    In this case, a basis as in (i) (a) can be chosen such that the matrix $A_\infty$ is
    diagonal, and the diagonal entries, multiplied by $-1$, are the spectral numbers of $(G_t,\nabla)$ at
    infinity.
  \item Suppose that we are moreover given a non-degenerate flat Hermitian pairing on $\mathbf{G}_t$ which has weight
    $n-1$ on $G_t$, more precisely (see \cite[section 1.f.]{DS} or \cite[section 4]{DS2}) a morphism $S:\mathbf{G}_t\otimes_{\CC[\tau,\tau^{-1}]} \overline{\mathbf{G}}_t \rightarrow
    \CC[\tau,\tau^{-1}]$ (where $\overline{\mathbf{G}}_t$ denotes the module $\mathbf{G}_t$
    on which $\tau$ acts as $-\tau$)
    with the following properties
    \begin{enumerate}
    \item
        $\tau\partial_\tau S(a,\overline{b})=S(\tau\partial_\tau a, \overline{b})+S(a,\tau\partial_\tau\overline{ b})$,
    \item
        $S:V_0\otimes\overline{V}_{<1}\rightarrow \CC[\tau]$,
    \item
        $S(G_t,\overline{G}_t)\subset\tau^{-n+1}\CC[\tau^{-1}]$, and the induced symmetric
        pairing $G_t/\tau^{-1}G_t\otimes G_t/\tau^{-1}G_t \rightarrow \tau^{-n+1}\CC$ is non-degenerate.
    \end{enumerate}
    In particular, the spectral numbers
    then obey the symmetry $\alpha_1+\alpha_{n+1-i}=n-1$.  A $V^+$-solution $G'_t$ is called
    $(V^+,S)$-solution if $S(G_t\cap G'_t, \overline{G}_t\cap \overline{G}'_t)\subset \CC \tau^{-n+1}$.
  \end{enumerate}
\end{lemmadefinition}
We will see in the sequel (theorem \ref{theoS-solution}) that under a technical hypothesis
(which is however satisfied in many examples)
we are able to construct directly a $(V^+,S)$-solution. Without this hypothesis,
we can for the moment only construct a $V^+$-solution. In order to obtain Frobenius
structures in all cases, we need the following general result, which we quote from \cite{Sa2} and \cite{DS}.
\begin{theorem}\label{canSol}
  Let $Y$ be a smooth affine complex algebraic variety and
  $f:Y\rightarrow \CC$ be a cohomologically tame function. Then the Gau\ss-Manin system
  of $f$ is equipped a pairing $S$ as above, and
  there is a canonical
  $(V^+,S)$-solution to the Birkhoff problem for the Brieskorn lattice of $f$, defined by
  a (canonical choice of an) opposite filtration to the Hodge filtration of the mixed
  Hodge structure associated to $f$.
\end{theorem}
The key tool to compute the spectrum and to obtain such a $V^+$-solution to the Birkhoff
problem is the following result.
\begin{proposition}
  \label{ExchangeTrick}
  Let $t\in T$ be arbitrary, $G_t\subset \mathbf{G}_t$ as before and consider any solution to the Birkhoff problem  for $(G_t,\nabla)$, given by a
	basis $\underline{\omega}$ of $G_t$ such that
  $\partial_\tau(\underline{\omega}) =\underline{\omega}(\Omega_0 +
  \tau^{-1}A_\infty)$ with $\Omega_0$ as above and such that
  $A_\infty=\diag(-\nu_1,\ldots, -\nu_n)$ is diagonal. Suppose moreover that
  $\nu_i-\nu_{i-1}\leq 1$ for all $i\in\{1,\ldots,i\}$ and additionally that
  $\nu_1-\nu_n\leq 1$ if $t\neq 0$.

  Then $\underline{\omega}$ is a $V^+$-solution to the Birkhoff problem and the
  numbers $(\nu_i)_{i=1,\ldots,n}$ give the spectrum $\textup{Sp}(G_t,\nabla)$ of $G_t$
  at infinity.
\end{proposition}
\begin{proof}
  The basic idea is similar to \cite{DS2} and \cite{ignacio3}, namely, that the spectrum of
  $A_\infty$ can be used to \emph{define} a filtration which turns out to coincide with
  the $V$-filtration using that the latter is unique with the above properties. More
  precisely, we define a $\KK$-grading on $\mathbf{G}_t$ by $\deg(\tau^k
  \omega_i):=\nu_i-k$ and consider the associated increasing filtration
  $\widetilde{V}_\bullet\mathbf{G}_t$ given by
  \begin{align*}
      \widetilde{V}_\alpha\mathbf{G}_t & :=  \left\{\sum_{i=1}^n c_i \tau^{k_i} \omega_i \in \mathbf{G}_t\,|\,
        \textup{max}_i(\nu_{i}-k_i)\leq \alpha\right\} \\
      \widetilde{V}_{<\alpha}\mathbf{G}_t & := \left\{\sum_{i=1}^n c_i \tau^{k_i} \omega_i \in \mathbf{G}_t\,|\,
        \textup{max}_i(\nu_{i}-k_i) < \alpha\right\}.
  \end{align*}
  By definition $\partial_\tau \widetilde{V}_\bullet \mathbf{G}_t\subset
  \widetilde{V}_{\bullet+1} \mathbf{G}_t$ and $\tau \widetilde{V}_\bullet \mathbf{G}_t
  \subset \widetilde{V}_{\bullet-1} \mathbf{G}_t$ and moreover, $\tau$ is obviously
  bijective on $\mathbf{G}$. Thus to verify that $V_\bullet \mathbf{G}_t =
  \widetilde{V}_\bullet \mathbf{G}_t$, we only have to show that
  $\tau\partial_\tau+\alpha$ is nilpotent on $gr_\alpha^{\widetilde{V}}\mathbf{G}_t$. This
  will prove both statements of the proposition: the conditions in definition
  \ref{Vsolution} for $\underline{\omega}$ to be a $V^+$-solution are trivially
  satisfied if we replace $V$ by $\widetilde{V}$.  The nilpotency of
  $\tau\partial_\tau+\alpha \in \mathit{End}_{\CC}(gr^{\widetilde{V}}_\alpha\mathbf{G}_t)$
  follows from the assumption $\nu_i-\nu_{i-1}\leq 1$:

  First define a block decomposition of the ordered tuple $(1,\ldots,n)$ by putting
  $(1,\ldots,n)=(I_1,\ldots,I_s)$, where $I_r=(i_r,i_r+1,\ldots,i_r+l_r=i_{r+1}-1)$ such
  that $\nu_{i_r+l+1}-\nu_{i_r+l}=1$ for all $l\in\{0,\ldots,l_r-1\}$ and
  $\nu_{i_r}-\nu_{i_r-1}<1$, $\nu_{i_{r+1}}-\nu_{i_r+l_r}<1$.  Then in $\mathbf{G}_t$ we
  have $(\tau\partial_\tau+(\nu_i-k_i))(\tau^{k_i}\omega_i)=\tau^{k_i+1}\omega_{i+1}$ for
  $i\in\{1,\ldots,n-1\}$ and
  $(\tau\partial_\tau+(\nu_n-k_n))(\tau^{k_n}\omega_n)=t\tau^{k_n+1}\omega_1$, so that
  $(\tau\partial_\tau+(\nu_i-k_i))^{i_{r+1}-i}(\tau^{k_i}\omega_i)=0$ in
  $gr^{\widetilde{V}}_{\nu_i-k_i}\mathbf{G}_t$ for all $i\in I_r$ (here we put
  $i_{s+1}:=n+1$, note also that if $t\neq 0$ we suppose that $\nu_1-\nu_n\leq 1$).
\end{proof}
As a by-product, a solution with the above properties also makes it possible to compute
the monodromy of $\mathbf{G}_t$. Consider the local system $\mathbf{G}_t^\nabla$ and the
space $\mathbf{G}_t^\infty$ of its multivalued flat sections. There is a natural
isomorphism $\oplus_{\alpha\in(0,1]} gr^V_\alpha\mathbf{G}_t \stackrel{\psi}{\rightarrow}
\mathbf{G}_t^\infty$.  The monodromy $M\in\mathit{Aut}(\mathbf{G}_t^\infty)$, which
corresponds to a counter-clockwise loop around $\tau=\infty$, decomposes as $M=M_s\cdot
M_u$ into semi-simple and unipotent part, and we write $N:=\log(M_u)$ for the nilpotent
part of $M$. The endomorphism $N$ corresponds under the isomorphism $\psi$, up to a constant factor, to the
operator $\oplus_{\alpha\in(0,1]} (\tau\partial_\tau+\alpha)\in \oplus_{\alpha\in(0,1]}
\mathit{End}_{\CC}(gr^V_\alpha\mathbf{G})$. This gives the following result, notice that
a similar statement and proof are given in \cite[end of section 3]{DS2}.
\begin{corollary}
  \label{monodromy}
  Consider the basis of $\mathbf{G}^\infty_t$ induced from a basis
  $\underline{\omega}$ as above, i.e.,
  \begin{displaymath}
    \mathbf{G}^\infty_t=
    \oplus_{i=1}^n\CC \psi^{-1}([\tau^{l_i}\omega_i]),
  \end{displaymath}
  where $l_i=\lfloor\nu_i\rfloor+1$.  Then $M_s \psi^{-1}[\tau^{l_i} \omega_i] =
  e^{-2\pi i \nu_i} \cdot \psi^{-1}[\tau^{l_i}\omega_i]$ and
  \begin{displaymath}
    N(\psi^{-1}[\tau^{l_i}\omega_i])=
    \begin{cases}
      2\pi i\psi^{-1}[\tau^{l_i}\omega_{i+1}]&\textup{if }\nu_{i+1}-\nu_i=1\\
      0 & \textup{else,}
    \end{cases}
  \end{displaymath}
  where $\omega_{n+1}=\omega_1$ if $t\neq 0$ and $\omega_{n+1}=0$ if $t=0$.  Thus the
  Jordan blocks of $N$ are exactly the blocks appearing above in the decomposition of the
  tuple $(1,\ldots,n)$.
\end{corollary}

We can now use proposition \ref{ExchangeTrick} to compute a $V^+$-solution and the
spectrum of $G_t$. We give an explicit algorithm, which we split into two parts for the
sake of clarity. Once again it should be emphasized that the special form of the matrix
$\Omega_0$ is the main ingredient for this algorithm.
\begin{algorithm}
  \label{algorithm1}
  Given $\underline{\omega}^{(1)}$ from lemma \ref{lemBirkhoff}, i.e., $\partial_\tau(\underline{\omega}^{(1)})
  =\underline{\omega}^{(1)}(\Omega_0+\tau^{-1}A^{(1)}_\infty)$ and
  $A^{(1)}_\infty=\diag(-\nu^{(1)}_1,\ldots,-\nu^{(1)}_n)$, whenever there is $i\in\{2,\ldots,n\}$
  with $\nu^{(1)}_i-\nu^{(1)}_{i-1}>1$, put
  \begin{equation}\label{eqAlg1}
  \begin{array}{rcl}
     \widetilde{\omega}^{(1)}_i & :=  & \omega^{(1)}_i+\tau^{-1}(\nu^{(1)}_i-\nu^{(1)}_{i-1}-1)\omega^{(1)}_{i-1}\\
      \widetilde{\omega}^{(1)}_j & :=  & \omega^{(1)}_j\quad\forall j\neq i
  \end{array}
  \end{equation}
  so that $\partial_\tau(\widetilde{\underline{\omega}}^{(1)})
  =\widetilde{\underline{\omega}}^{(1)}(\Omega_0+\tau^{-1}\widetilde{A}^{(1)}_\infty)$ and
  $\widetilde{A}^{(1)}_\infty=\diag(-\widetilde{\nu}^{(1)}_1,\ldots,-\widetilde{\nu}^{(1)}_n)$,
  where $\widetilde{\nu}^{(1)}_i=\nu^{(1)}_{i-1}+1$,
  $\widetilde{\nu}^{(1)}_{i-1}=\nu^{(1)}_i-1$ and $\widetilde{\nu}^{(1)}_j=\nu^{(1)}_j$
  for any $j\notin\{i,i-1\}$. Restart algorithm $1$ with input
  $\widetilde{\underline{\omega}}^{(1)}$.
\end{algorithm}
Now we have
\begin{lemma}\label{Alg1Term}
  Given any basis $\underline{\omega}^{(1)}$ of $G_t$ as above, algorithm 1 terminates.
  Its output $\underline{\omega}^{(2)}$ is a $V^+$-solution for $G_t$ if $t=0$.
\end{lemma}
\begin{proof}
  The first statement is a simple analysis on the action of the algorithm on the array
  $(\nu^{(1)}_1,\ldots,\nu^{(1)}_n)$, namely, if $(\nu^{(1)}_1,\ldots,\nu^{(1)}_k)$ is ordered (i.e.,
  $\nu^{(1)}_i-\nu^{(1)}_{i-1}\leq 1$ for all $i\in\{2,\ldots k\}$), then after a finite number of
  steps the array $(\widetilde{\nu}^{(1)}_1,\ldots,\widetilde{\nu}^{(1)}_{k+1})$ will be ordered.
  This shows that the algorithm will eventually terminate. Its output is then a
  $V^+$-solution for $G_t$ if $t=0$ by proposition \ref{ExchangeTrick}.
\end{proof}
If we want to compute the spectrum and a $V^+$-solution of $G_t$ for $t\neq 0$, we also
have to make sure that $\nu_1-\nu_n\leq 1$.  This is done by the following procedure.
\begin{algorithm}
  Run algorithm 1 on the input $\underline{\omega}^{(1)}$ with output
  $\underline{\omega}^{(2)}$ where
  $A^{(2)}_\infty=(-\nu^{(2)}_1,\ldots,-\nu^{(2)}_n)$.  As long as
  $\nu^{(2)}_1-\nu^{(2)}_n>1$, put
  \begin{equation}\label{eqAlg2}
  \begin{array}{rcl}
      \widetilde{\omega}^{(2)}_1 & := & t\omega^{(2)}_1+\tau^{-1}(\nu^{(2)}_1-\nu^{(2)}_n-1)\omega^{(2)}_n \\
      \widetilde{\omega}^{(2)}_i & := & t\omega^{(2)}_i\quad\forall i\neq 1.
  \end{array}
  \end{equation}
  so that $\partial_\tau(\underline{\widetilde{\omega}}^{(2)})
  =\underline{\widetilde{\omega}}^{(2)}(\Omega_0+\tau^{-1}\widetilde{A}^{(2)}_\infty)$ with
  $\widetilde{A}^{(2)}_\infty=\diag(-\widetilde{\nu}^{(2)}_1,\ldots,-\widetilde{\nu}^{(2)}_n)$, where $\widetilde{\nu}^{(2)}_1=\nu^{(2)}_n+1$,
  $\widetilde{\nu}^{(2)}_n=\nu^{(2)}_1-1$ and $\widetilde{\nu}^{(2)}_i=\nu^{(2)}_i$ for any
  $i\notin\{1,n\}$. Run algorithm 2 again on input $\underline{\widetilde{\omega}}^{(2)}$.
\end{algorithm}

\begin{lemma}\label{Alg2Term}
  Let $t\neq 0$, given any solution $\underline{\omega}^{(1)}$ to the Birkhoff problem for
  $G_t$, such that $\partial_\tau(\underline{\omega}^{(1)})
  =\underline{\omega}^{(1)}(\Omega_0+\tau^{-1}A^{(1)}_\infty)$ with $\Omega_0$ as above and $A^{(1)}_\infty$
  diagonal, then algorithm 2 with input $\underline{\omega}^{(1)}$ terminates and yields a
  basis $\underline{\omega}^{(3)}$ with $\partial_\tau(\underline{\omega}^{(3)})
  =\underline{\omega}^{(3)}(\Omega_0+\tau^{-1}A^{(3)}_\infty)$, where
  $A^{(3)}_\infty=(-\nu^{(3)}_1, \ldots, -\nu^{(3)}_n)$ with $\nu^{(3)}_{i+1}-\nu^{(3)}_i\leq 1$ for
  $i\in\{1,\ldots n\}$ (here $\nu^{(3)}_{n+1}:=\nu^{(3)}_1$).
\end{lemma}
\begin{proof}
  We only have to prove that algorithm 2 terminates. This is easily be done by showing
  that in each step, the number $\widetilde{\nu}^{(2)}_1-\widetilde{\nu}^{(2)}_n$ does not
  increase, that it strictly decreases after a finite number of steps, and that the
  possible values for this number are contained in the set
  $\{a-b\,|\,a,b\in\{-\nu^{(1)}_1,\ldots,-\nu^{(1)}_n\}\}+\ZZ$ (which has no accumulation points), so
  that after a finite number of steps we necessarily have
  $\widetilde{\nu}^{(2)}_1-\widetilde{\nu}^{(2)}_n\leq 1$.
\end{proof}
Note that for any fixed $t\neq 0$, algorithm 2 produces a base change of $G_t$, but this
does not lift to a base change of $G$ itself, i.e., $G^{(3)}:=\oplus_{i=1}^n\CC[\tau^{-1},t]
\omega^{(3)}_i$ is a proper free submodule of $G$ which coincides with $G$ only after
localization off $t=0$. In other words, it is a $\CC[t]$-lattice of $G[t^{-1}]$ which is
in general different from $G$.

Summarizing the above calculations, we have shown the following.
\begin{corollary}
  \label{summary}
  \begin{enumerate}
  \item
  Let $D\subset V$ be a linear free divisor with defining equation $h\in\CC[V]_n$, seen as
  a morphism $h:V\rightarrow T$. Let $f\in \CC[V]_1$ be linear and $\mathcal{R}_h$-finite.
  Then for any $t\in T$, there is a $V^+$-solution of the Birkhoff problem for
  $(G_t,\nabla)$, defined by bases $\underline{\omega}^{(2)}$ if $t = 0$
  resp. $\underline{\omega}^{(3)}$ if $t\neq 0$
  as constructed
  above.  If $\nu^{(2)}_1-\nu^{(2)}_n\leq 1$ then $\underline{\omega}^{(3)}=\underline{\omega}^{(2)}$. Moreover, we have that $\underline{\omega}^{(2)}_i-(-f)^{i-1}\alpha$
  and $\omega^{(3)}_i-(-f)^{i-1}\alpha$ lie in $\tau^{-1}G_t$ for all $i\in\{1,\ldots,n\}$.
  \item
  Let $D$ be reductive. Then the integrable connection $\nabla$ on $\mathbf{G}[t^{-1}]$ defined by formula \eqref{horizontalconnection} takes the following form in the basis $\underline{\omega}^{(3)}$:
  $$
  \nabla(\underline{\omega}^{(3)}) = \underline{\omega}^{(3)} \cdot
  \left[
    (\Omega_0 + \tau^{-1}A^{(3)}_\infty)d\tau + (\widetilde{D}+\tau\Omega_0+A^{(3)}_\infty)\frac{dt}{nt}
  \right]
  $$
  where $\widetilde{D}:=\diag(0,\ldots,n-1)+k\cdot n \cdot\mathit{Id}$, here $k$ is the
  number of times the (meromorphic) base change \eqref{eqAlg2} in algorithm 2 is performed.

   Hence, in the reductive case, $\underline{\omega}^{(3)}$ gives a $V^+$-solution $\widehat{G}[t^{-1}]$ to the Birkhoff problem for $(G[t^{-1}],\nabla)$.
  \end{enumerate}
\end{corollary}
\begin{proof}
  Starting with the basis $\omega^{(0)}_i=(-f)^{i-1}\alpha$ of $G_t$, we construct
  $\underline{\omega}^{(2)}$ resp. $\underline{\omega}^{(3)}$ using lemma \ref{lemBirkhoff},
  proposition \ref{ExchangeTrick} and lemma \ref{Alg1Term} resp. lemma \ref{Alg2Term}. In both cases,
  the base change matrix $P\in \Gl(n, \CC[\tau^{-1}])$
  defined by $\underline{\omega}^{(2)} = \underline{\omega}^{(0)} \cdot P$
  resp. $\underline{\omega}^{(3)} = \underline{\omega}^{(0)} \cdot P$ has the
  property that $P-\Id\in \tau^{-1}\Gl(n, \CC[\tau^{-1}])$ which shows the second
  statement of the first part. As to the second part, one checks that
  the base change steps \eqref{eqAlg1} performed in algorithm 1 have the effect that
  $n\cdot t\partial_t(\widetilde{\underline{\omega}}^{(1)})
  =\widetilde{\underline{\omega}}^{(1)}(\tau\Omega_0 + \diag(0,\ldots,n-1) + \widetilde{A}^{(1)}_\infty)$,
  whereas step \eqref{eqAlg2} in algorithm 2 gives
  $n\cdot t\partial_t(\underline{\widetilde{\omega}}^{(2)})
  =\underline{\widetilde{\omega}}^{(2)}(\tau\Omega_0+\diag(0,\ldots,n-1)+n\cdot\mathit{Id}+\widetilde{A}^{(2)}_\infty)$.
\end{proof}

As already indicated above, we can show that the solution obtained behaves well
with respect to the pairing $S$, provided that a technical hypothesis holds true. More precisely, we have the
following statement.
\begin{theorem}
\label{theoS-solution}
Let $t\neq 0$. Suppose that the minimal spectral number of the tame function $f_{|D_t}$
is of multiplicity one, i.e., there is a unique $i\in\{1,\ldots,n\}$ such that
$\nu^{(3)}_i=\min_{j\in\{1,\ldots,n\}}(\nu^{(3)}_j)$. Then $\underline{\omega}^{(3)}$ is a
$(V^+,S)$-solution of the Birkhoff problem for $(G_t,\nabla)$, i.e.,
$S(G_t\cap G'_t, \overline{G}_t\cap \overline{G}'_t)\subset \CC \tau^{-n+1}$, where $G'_t:=\oplus_{i=1}^n\OO_{\PP^1\backslash\{0\}
\times\{t\}}\omega_i^{(3)}$.
\end{theorem}
\begin{proof}
The proof is essentially a refined version of the proof of the similar
statement \cite[lemma 4.1]{DS2}. Denote by $\alpha_1,\ldots,\alpha_n$
a non-decreasing sequence of rational numbers such that
we have
an equality of sets $\{\nu^{(3)}_1,\ldots,\nu^{(3)}_n\}=\{\alpha_1,\ldots,\alpha_n\}$.
Then, as was stated in lemma \ref{Vsolution} (iv), we have
$\alpha_i+\alpha_{n+1-i}=n-1$ for all $k\in\{1,\ldots,n\}$.

Let $i$ be the index of the smallest spectral number $\nu^{(3)}_i$. The
symmetry $\alpha_k+\alpha_{n+1-k}=n-1$ implies that there is a unique $j\in\{1,\ldots,n\}$ such that
$\nu^{(3)}_i+\nu^{(3)}_j=n-1$, or, equivalently, that $\nu^{(3)}_j=\max_{l\in\{1,\ldots,n\}}(\nu^{(3)}_l)$.
Then, as in the proof of loc.cit., we have that for all $l\in\{1,\ldots,n\}$
$$
S(\omega^{(3)}_i,\overline{\omega}^{(3)}_l) = \left\{
\begin{array}{lcl}
0 & \textup{if} & l \neq j\\
c \cdot\tau^{-n+1}, c \in \CC & \textup{if} & l=j
\end{array}
\right.
$$
This follows from the compatibility of $S$ with the $V$-filtration and the pole order property
of $S$ on the Brieskorn lattice $G_t$ (i.e., properties (iv) (b) and (c) in definition
\ref{Vsolution}).
Suppose without loss of generality that $i < j$, if $i=j$, i.e., if there is only
one spectral number, then the result is clear. Now the proof of the theorem follows from the next lemma.
\end{proof}
\begin{lemma}\label{lemSpectrum}
Let $i$ and $j$ as above. Then
\begin{enumerate}
\item
For any
$k\in\{i,\ldots,j\}$, we have
$$
S(\omega^{(3)}_k,\overline{\omega}^{(3)}_l) = \left\{
\begin{array}{lcl}
0 & \textup{for all} & l \neq i+j-k\\
S(\omega^{(3)}_i, \overline{\omega}^{(3)}_j) \mbox{\textbf{ and }}\nu^{(3)}_k+\nu^{(3)}_l=n-1& \textup{for} & l = i+j-k
\end{array}
\right.
$$
\item
For any  $k\in\{1,\ldots,n\}\backslash\{i,\ldots,j\}$, we have that
$$
S(\omega^{(3)}_k,\overline{\omega}^{(3)}_l) = \left\{
\begin{array}{lcl}
0 & \textup{for all} & l \neq i+j-k\\
c_{kl}\cdot S(\omega^{(3)}_i, \overline{\omega}^{(3)}_j) \mbox{\textbf{ and }}\nu^{(3)}_k+\nu^{(3)}_l=n-1& \textup{for} & l = i+j-k
\end{array}
\right.
$$
where $c_{kl}\in\CC$.
\end{enumerate}
\end{lemma}
\begin{proof}
\begin{enumerate}
\item
We will prove the statement by induction over $k$. It is obviously true for $k=i$ by the hypothesis above.
Hence we suppose that there is $r\in\{i,\ldots,j\}$ such that statement in (i) is true for
all $k$ with $i\leq k < r \leq j$. The following identity is a direct consequence of property (iv) (a) in definition
\eqref{Vsolution}.
$$
\begin{array}{c}
(\tau\partial_\tau + (n-1)) S(\omega^{(3)}_k,\overline{\omega}^{(3)}_l) = S(\tau\partial_\tau \omega^{(3)}_k, \overline{\omega}^{(3)}_l)+
S(\omega^{(3)}_k, \tau\partial_\tau \overline{\omega}^{(3)}_l)+(n-1)S(\omega^{(3)}_k,\overline{\omega}^{(3)}_l)\\ \\
=S(\tau\partial_\tau \omega^{(3)}_k, \overline{\omega}^{(3)}_l)+
S(\omega^{(3)}_k, \overline{\tau\partial_\tau \omega}^{(3)}_l)+(n-1)S(\omega^{(3)}_k,\overline{\omega}^{(3)}_l) = \\ \\
S(\tau\omega^{(3)}_{k+1}-\nu^{(3)}_k\omega^{(3)}_k,\overline{\omega}^{(3)}_l)+S(\omega^{(3)}_k,\overline{\tau\omega^{(3)}_{l+1}-\nu^{(3)}_l\omega^{(3)}_k})+(n-1)S(\omega^{(3)}_k,\overline{\omega}^{(3)}_l)
=\\ \\
(n-1-\nu^{(3)}_k-\nu^{(3)}_l)S(\omega^{(3)}_k,\overline{\omega}^{(3)}_l)+\tau\left(S(\omega^{(3)}_{k+1},\overline{\omega}^{(3)}_l)-S(\omega^{(3)}_k,\overline{\omega}^{(3)}_{l+1})\right).
\end{array}
$$
By induction hypothesis, we have that $(\tau\partial_\tau+(n-1))S(\omega^{(3)}_k,\overline{\omega}^{(3)}_l)=0$ for all $l\in\{1,\ldots,n\}$.

Now we distinguish several cases depending on the value of $l$: If $l\notin\{i+j-k,i+j-k-1\}$, then by the induction
hypothesis, both $S(\omega^{(3)}_k,\overline{\omega}^{(3)}_l)$ and $S(\omega^{(3)}_k,\overline{\omega}^{(3)}_{l+1})$ are zero. Hence it follows that
$S(\omega^{(3)}_{k+1},\overline{\omega}^{(3)}_l)=0$ in this case.

If $l=i+j-k$ then again by the induction hypothesis we know that $(n-1)-\nu^{(3)}_k-\nu^{(3)}_l=0$ and that moreover
$S(\omega^{(3)}_k,\overline{\omega}^{(3)}_{l+1})=0$. Thus we have $S(\omega^{(3)}_{k+1},\overline{\omega}^{(3)}_{i+j-k})=0$.

Finally, if $l=i+j-k-1$, then $S(\omega^{(3)}_k,\overline{\omega}^{(3)}_l)=0$, and so
$S(\omega^{(3)}_{k+1}, \overline{\omega}^{(3)}_l)=S(\omega^{(3)}_k, \overline{\omega}^{(3)}_{l+1})$
in other words: $S(\omega^{(3)}_{k+1}, \overline{\omega}^{(3)}_{i+j-(k+1)})=S(\omega^{(3)}_k,\overline{\omega}^{(3)}_{i+j-k})$.
In conclusion, we obtain that
$$
S(\omega^{(3)}_{k+1},\overline{\omega}^{(3)}_l) = \left\{
\begin{array}{lcl}
0 & \textup{if} & l \neq i+j-(k+1)\\
S(\omega^{(3)}_k, \overline{\omega}^{(3)}_{i+j-k}) & \textup{if} & l = i+j-(k+1).
\end{array}
\right.
$$
In order to make the induction work, it remains to show that $\nu^{(3)}_{k+1}+\nu^{(3)}_{i+j-(k+1)}=n-1$. It is obvious
that $\nu^{(3)}_{k+1}+\nu^{(3)}_{i+j-(k+1)} \geq n-1$ for otherwise we would have $S(\omega^{(3)}_{k+1},\overline{\omega}^{(3)}_{i+j-(k+1)})=0$.
(Remember that it follows from the flatness of $S$, i.e. from condition (iv) (a) in \ref{Vsolution}, that
$S:V_\alpha \otimes \overline{V}_{<1-\alpha+m}\rightarrow \tau^{-m}\CC[\tau]$ for any $\alpha\in\QQ, m\in\ZZ$, so that
$S(\omega^{(3)}_{k+1},\overline{\omega}^{(3)}_{i+j-(k+1)})\in\tau^{-n+2}\CC[\tau]$ if $\nu^{(3)}_{k+1}+\nu^{(3)}_{i+j-(k+1)} < n-1$, which is impossible since $S:G_t\otimes_{\CC[\tau^{-1}]}\overline{G}_t\rightarrow\tau^{-n+1}\CC[\tau^{-1}]$). Thus the only case to exclude is
$\nu^{(3)}_{k+1}+\nu^{(3)}_{i+j-(k+1)} > n-1$.

First notice that it follows from property (iv) (c) of definition \ref{Vsolution} that $S$ induces an isomorphism
$$
\overline{\tau^{n-1} G}_t \cong G_t^*:=\mathit{Hom}_{\CC[\tau^{-1}]}(G_t,\CC[\tau^{-1}]).
$$
On the other hand, we deduce from \cite[remark 3.6]{Sa2} that for any $\alpha\in\{\nu^{(3)}_1,\ldots,\nu^{(3)}_n\}$,
$$
gr^{V^*}_\alpha(G_t^*/\tau^{-1}G_t^*) \cong gr^{V}_{-\alpha}(G_t/\tau^{-1}G_t),
$$
where $V^*$ denotes the canonical V-filtration on the dual module $(G_t,\nabla)^*$.
In conclusion, $S$ induces a non-degenerate pairing
$$
S:gr^V_\alpha(G_t/\tau^{-1}G_t)\otimes gr^V_{n-1-\alpha}(G_t/\tau^{-1}G_t)\rightarrow \tau^{-n+1}\CC
$$
which yields a non-degenerate pairing on the sum $gr^V_\bullet(G_t/\tau^{-1}G_t):=\oplus_{\alpha\in \QQ}
gr^V_\alpha (G_t/\tau^{-1}G_t)$. However, we know that $\underline{\omega}^{(3)}$ induces a basis
of $gr^V_\bullet(G_t/\tau^{-1}G_t)$, compatible with the above decomposition. This, together with the fact
that $S(\omega^{(3)}_{k+1}, \overline{\omega}^{(3)}_l) \in \tau^{-n+1}\CC\delta_{i+j,k+1+l}$, yields that
$\nu^{(3)}_{k+1}+\nu^{(3)}_{i+j-(k+1)}=n-1$, as required.

\item
For this second statement, we consider the constant (in $\tau^{-1}$) base change given by
$\omega'^{(3)}_{k+1}:=\omega^{(3)}_{j+k}$ for all $k\in \{0,\ldots,n-j\}$ and
$\omega'^{(3)}_{k+1+n-j}:=t\omega^{(3)}_k$ for all $k\in\{1,\ldots,j-1\}$.
Then we have
$$
\partial_\tau(\underline{\omega}'^{(3)}) = \underline{\omega}'^{(3)} \cdot
    (\Omega_0 + \tau^{-1}(A^{(3)}_\infty)'),
$$
where $(A^{(3)}_\infty)'=\diag(-\nu^{(3)}_j,-\nu^{(3)}_{j+1},\ldots,-\nu^{(3)}_n,-\nu^{(3)}_1,\ldots,-\nu^{(3)}_{j-1})$.
Now the proof of (i) works verbatim for the basis $\underline{\omega}'^{(3)}$, with the index $i$ from above
replaced by $1$ and the index $j$ from above replaced by $n-j+i+2$. Notice that then the spectral number
corresponding to $1$ is the biggest one and the one corresponding to $n-j+i+2$ is the smallest one, but this does
not affect the proof. Depending on the value of the indices $k$ and $l$, we have that $c_{kl}(t)$ is either
$t^{-1}$, $1$ or $t$.

%{\tiny \bf In order to complete the prove, we have to show that the claim also holds ``the other way round'', i.e., for indices $k\in\{1,\ldots,n\}\backslash\{i,\ldots,j\}$. This can be
%done by a simple trick: We set $\widetilde{\omega}^{(3)}_{k+1}:=\omega^{(3)}_{j+k}$ for all $k\in \{0,\ldots,n-j\}$ and $\widetilde{\omega}^{(3)}_{k+1+n-j}:=\omega^{(3)}_k$ for all $k\in\{1,\ldots,j-1\}$.
%Then as above we have $S(\widetilde{\omega}^{(3)}_1,\overline{\widetilde{\omega}^{(3)}}_l)=0$ for all $l\neq i-j+n+1$ and
%$S(\widetilde{\omega}^{(3)}_1,\overline{\widetilde{\omega}^{(3)}}_{i-j+n+1})\in\CC\tau^{-n}$, so that the
%Then we show the same claim with the same induction method for the basis $\underline{\widetilde{\omega}^{(3)}}$.
%Notice that in the above proof, we did not use the fact that $\omega^{(3)}_i$ is the unique section corresponding
%to the smallest spectral number, and $\omega^{(3)}_j$ is the one corresponding to the largest one, but we only
%used that property to show that both of them cannot pair with any other section and that the value of $S$ on them
%is in $\CC\tau^{-n}$. This obviously still holds for the sections $\widetilde{\omega}^{(3)}_1$ and $\widetilde{\omega}^{(3)}_{i-j+n+1}$ as
%these are just the ones from before, interchanged.
%}
\end{enumerate}
\end{proof}

% @@@@@@@@@@@@@@@@@@@@@@@@@@@@@@@@@@@@@@@@@@
\section{Frobenius structures}
\label{sec:Frobenius}

\subsection{Frobenius structures for linear functions on Milnor fibres}
\label{subsec:Frobenius-tneq0}

In this subsection, we derive one of the main results of this paper: the existence of a
Frobenius structure on the unfolding space of the function $f_{|D_t}, t\neq 0$.
Depending on whether we restrict to the class of examples satisfying the hypotheses of theorem \ref{theoS-solution},
the Frobenius structure can be derived directly from the $(V^+,S)$-solution $\underline{\omega}^{(3)}$ of the
Birkhoff problem constructed in the last section, or otherwise is obtained by appealing
to theorem \ref{canSol}.
%Assuming conjecture \ref{conjSpectrum}, we also obtain a Frobenius
%structure associated to the restriction of $f$ to $D$. A rather easy argument shows
%that this limit structure is constant, i.e., its associated potential
%is a polynomial of degree at most three when expressed in flat coordinates. We give
%some heuristic arguments that this should be related to the left-right stability
%as discussed in subsection \ref{subsec:LeftRight}. Finally, we discuss,
%for the case of the normal crossing divisor,
%how the Frobenius structures for various $t\neq 0$ fit together to form a Frobenius
%manifold with logarithmic poles (as defined and studied in \cite{Reich1}) along the divisor $t=0$.
%The same result has been obtained at the same time in \cite{Dou3}, using similar techniques.
%(The case of the normal crossing divisor is the intersection of the class of functions
%studied in loc.cit. with the one studied here).
%However, it seems unclear for the moment what kind of degeneration behaviour at $t=0$ occur
%for arbitrary linear free divisors (notice that the same problem is observed in loc.cit.,
%where the logarithmic structure does not extend to the case of Laurent polynomials).

We refer to \cite{He3} or \cite{Sa4} for the definition of a Frobenius manifold. It is
well known that a Frobenius structure on a complex manifold $M$ is equivalent to the
following set of data (sometimes called first structure connection).
\begin{enumerate}
\item a holomorphic vector bundle $E$ on $\PP^1\times M$ such that
  $\textup{rank}(E)=\dim(M)$, which is fibrewise trivial,
  i.e. $\mathcal{E}=p^*p_*\mathcal{E}$, (where $p:\PP^1\times M\rightarrow M$ is the
  projection) equipped with an integrable connection with a logarithmic pole along
  $\{\infty\}\times M$ and a pole of type one along $\{0\}\times M$,
\item an integer $w$,
\item A non-degenerate, $(-1)^w$-symmetric pairing $S:\mathcal{E}\otimes j^*
  \mathcal{E}\rightarrow \mathcal{O}_{\PP^1\times M}(-w,w)$ (here $j(\tau,u)=(-\tau,u)$,
  with, as before, $\tau$ a coordinate on $\PP^1$ centered at infinity and $u$ a
  coordinate on $M$ and we write $\mathcal{O}_{\PP^1\times M}(a,b)$ for the sheaf of
  meromorphic functions on $\PP^1\times M$ with a pole of order $a$ along $\{0\}\times M$
  and order $b$ along $\{\infty\}\times M$) the restriction of which to $\CC^*\times M$ is
  flat,
\item A global section $\xi\in H^0(\PP^1\times M, \mathcal{E})$, whose restriction to $\{\infty\}\times M$ is flat with
  respect to the residue connection $\nabla^{res}:\mathcal{E}/\tau\mathcal{E}\rightarrow
  \mathcal{E}/\tau\mathcal{E}\otimes\Omega^1_M$ with the following two properties
  \begin{enumerate}
  \item The morphism
    \begin{align*}
        \Phi_\xi:\mathcal{T}_M & \longrightarrow  \mathcal{E}/\tau^{-1}\mathcal{E}\cong p_*\mathcal{E} \\
        X & \longmapsto  -[\tau^{-1}\nabla_X](\xi)
    \end{align*}
    is an isomorphism of vector bundles (a section $\xi$ with this property is called
    primitive),
  \item $\xi$ is an eigenvector of the residue endomorphism $[\tau\nabla_\tau]\in
    \mathcal{E}\!nd_{\mathcal{O}_M}(p_*\mathcal{E})\cong
    \mathcal{E}\!nd_{\mathcal{O}_M}(\mathcal{E}/\tau \mathcal{E})$ (a section with this
    property is called homogeneous).
  \end{enumerate}
\end{enumerate}
In many application one is only interested in constructing a Frobenius structure on a germ
at a given point, in that case $M$ is a sufficiently small representative of such a germ.
% In many applications it is sufficient to construct a germ of a Frobenius structure.
% Then it follows from results of Malgrange (see \cite{Mal4, Mal5} and also \cite{Sa3})
% that if the above properties of a section $\xi$ are satisfied at a point $x\in M$ (i.e.,
% at the fibre $\mathcal{E}/\mathfrak{m}_x\mathcal{E}$), then a sufficiently small
% representative of the germ $(M,x)$ can be endowed with a Frobenius structure.

We now come back to our situation of a $\mathcal{R}_h$-finite linear section $f$ on the
Milnor fibration $h:V\rightarrow T$. In this subsection, we are interested to construct Frobenius structures on the
(germ of a) semi-universal unfolding of the function $f_{|D_t}$, $t\neq 0$.
It is well known that in contrast to the local case, such an unfolding does
not have obvious universality properties.
One defines, according to \cite[2.a.]{DS},
a deformation
\begin{displaymath}
  F=f+\sum_{i=1}^n u_i g_i : B_t \times M \rightarrow D
\end{displaymath}
of the restriction $f_{|B_t}$ to some
intersection $D_t\cap B_\epsilon$ such that the critical locus
$C$ of $F$ is finite over $M$ via the projection
$q:B_t\times M\twoheadrightarrow M$
to be a semi-universal unfolding if the
Kodaira-Spencer map $\mathcal{T}_M \rightarrow q_*\mathcal{O}_C$,
$X \mapsto [X(F)]$ is an isomorphism.

%In particular, we might restrict the
%function $f$ to some open set $B_t:=B_\varepsilon \cap f^{-1}(D) \cap D_t$, where
%$B_\varepsilon$ is an open ball in $V$ such that all critical points of $f$ are contained
%in $B$ and $D$ is an open disc in $\CC$.  For any unfolding $F: B_t \times M \rightarrow
%D$, of $f_{|B_t}$, we denote by $C$ the critical locus of $F$. Then $B_\epsilon$ might be
%chosen in such a way that $C$ is finite over $M$ (this is not true if we consider $F$ as a
%function on $D_t\times M$). Following \cite{DS}, we call $F$ the semi-universal unfolding
%of $f_{|D_t}$ if the Kodaira-Spencer map $\mathcal{T}_M \rightarrow q_*\mathcal{O}_C$
%($q:B_t\times M\twoheadrightarrow M$ being the projection), $X \mapsto [X(F)]$ is an
%isomorphism.

From proposition \ref{numcons} we know that any basis $g_1,\ldots,g_n$ of
$T^1_{\s R_h} f$ gives a representative
\begin{displaymath}
  F=f+\sum_{i=1}^n u_i g_i : B_t \times M \rightarrow D
\end{displaymath}
of this unfolding, where $M$ is a sufficiently small neighborhood of the origin in
$\CC^n$, with coordinates $u_1,\ldots,u_n$.

%We consider the Gau\ss-Manin system and the
%Brieskorn lattice of this family, defined analogously to definition \ref{GMBrieskorn} by
%\begin{align*}
%  \begin{split}
%    \mathbf{G}_{t,u}&:=\displaystyle\frac{q_*\Omega^{n-1}_{B_t \times M/M}[\tau,\tau^{-1}]}{(d_x-\tau dF\wedge)q_*\Omega^{n-2}_{B_t \times M/M}[\tau,\tau^{-1}]}\quad\textup{and}\\
%    G_{t,u}&:=\textup{ Image of }q_*\Omega^{n-1}_{B_t \times M/M}[\tau^{-1}]\textup{ in
%    }\mathbf{G}_{t,u},
%  \end{split}
%\end{align*}
%respectively. They come equipped with the connection operator $\nabla_{\partial_\tau}$,
%defined as before, which we complete to an integrable connection by putting for any $X\in
%q^{-1}\mathcal{T}_M$:
%\begin{equation}
%  \label{eqGMconnection}
%  \nabla_X(\omega) :=
%  L_X(\omega) - \tau X(F)\cdot \omega
%\end{equation}
%and extending $\tau$-linearly.

In order to exhibit Frobenius structures via the approach sketched in the beginning of
this section, one has to find a $(V^+,S)$-solution to the Birkhoff problem for
$G_t$. If the minimal spectral number of $(G_t,\nabla)$ has multiplicity
one, then, according to corollary \ref{summary} and theorem \ref{theoS-solution},
the basis $\underline{\omega}^{(3)}$ yields such a solution, which we denote by
$\widehat{G}_t$ (which is, if $D$ is reductive, the restriction of $\widehat{G}[t^{-1}]$ from
corollary \ref{summary} (ii) to $\PP^1\times\{t\}$). Otherwise, we consider the
canonical solution from theorem \ref{canSol}, which is denoted by
$\widehat{G}^{can}_t$. The bundle called $\mathcal{E}$ in the beginning of this subsection
is then obtained by \emph{unfolding} the solution $\widehat{G}_t$ resp.
$\widehat{G}^{can}_t$. We will not describe $\mathcal{E}$ explicitly, but use
a standard result due to Dubrovin which gives directly the corresponding Frobenius structure
provided that one can construct a homogenous and primitive form for $\widehat{G}_t$ resp.
$\widehat{G}^{can}_t$, i.e., a section called $\xi$ above at the point $t$.

%In both cases, we obtain an
%extension of $G_{t,0}$ to a trivial bundle over $\PP^1$ with a logarithmic pole at
%infinity.  According to results of Malgrange (see \cite{Mal4, Mal5} and also \cite{Sa3}),
%it extends to a solution (denote by $\widehat{G}^{can}_{t,u}\rightarrow \PP^1 \times M$ resp. $\widehat{G}^{can}_{t,u}
%\rightarrow \PP^1 \times M$) on a sufficiently small neighborhood of $0$ in $M$.  This is
%the bundle $\mathcal{E}$ from above. The pairing $S$ is constructed in \cite{Sa6}, it is expected
%that it coincides, up to a constant factor, to the dual
%of the intersection form on Lefschetz thimbles, at least for functions
%which are also M-tame (which is true for reductive linear free divisors, as shown
%in the second part of section \ref{sec:tame}).

We can now state and prove the main result of this section.
\begin{theorem}\label{theoFrobeniusSemiSimple}
  Let $f\in\CC[V]_1$ be an $\mathcal{R}_h$-finite linear function. Write
  $M_t$ for the parameter space of a semi-universal unfolding
  $F:B_t\times M_t \rightarrow D$ of $f_{|B_t}$, $t\neq 0$ as described above.
  Let $\alpha_{\min}=\alpha_1$ be the minimal spectral number of $(G_t,\nabla)$.
  \begin{enumerate}
    \item
    Suppose that $\alpha_{\min}$ has multiplicity one, i.e., $\alpha_2>\alpha_1$. Then any of the sections $\omega^{(3)}_i\in H^0(\PP^1, \widehat{G}_t)$
    is primitive and homogeneous.  Any choice of such a section yields a Frobenius structure
    $(M_t,\circ,g,e,E)$ which we denote by $M_t^{(i)}$.
    \item
    Let $i\in\{1,\ldots,n\}$ such that $\nu^{(3)}_i=\alpha_{min}$. Then $\omega^{(3)}_i \in H^0(\PP^1,\widehat{G}^{can}_t)$
    (remember that $\widehat{G}^{can}_t$ is the canonical $(V^+,S)$-solution to the Birkhoff problem for $G_t$ described in
    theorem \ref{canSol}), and this
    section is primitive and homogeneous (with respect to $\widehat{G}^{can}_t$) and hence yields a Frobenius structure
    $(M_t,\circ,g,e,E)$, denoted by $M_t^{(i),can}$.
%    Without assumption on the multiplicity of $\alpha_{\min}$, we have that any
%    of the sections $\omega^{(3)}_i$ such that $\nu^{(3)}_i \in [\alpha_{min}, \alpha_{min}+1)$ is primitive and
%    homogenous for the canonical $(V^+,S)$-solution to the Birkhoff problem $\widehat{G}^{can}_t$.
%    Hence any such section yields a Frobenius structure
  \end{enumerate}
\end{theorem}
\textbf{Remark:} It is obvious that under the hypotheses of (i),
any non-zero constant multiple of the sections $\omega^{(3)}_i$ is also primitive and
homogeneous. In particular, this is true for the sections $t^{-k}\omega^{(3)}_i$. We will later
need to work with these rescaled sections, rather than with $\omega^{(3)}_i$ (see proposition \ref{propAnalytContFrob}
and theorem \ref{theoLogFrob}).
\begin{proof}
  In both cases, we use the universal
  semi-simple Frobenius structure defined by a finite set of given initial data as
  constructed by Dubrovin (\cite{Du1}, see also \cite[th\'eor\`eme VII.4.2]{Sa4}). The initial
  set of data we need to construct is
  \begin{enumerate}
  \item
    an $n$-dimensional complex vector space $W$,
  \item
    a symmetric, bilinear, non-degenerate pairing $g:W\otimes_\CC W \rightarrow \CC$,
  \item
    two endomorphisms $B_0,B_\infty\in\End_\CC(W)$ such that $B_0$ is semi-simple with distinct eigenvalues
    and $g$-selfadjoint
    and such that $B_\infty+B^*_\infty=(n-1)\Id$, where
    $B^*_\infty$ is the $g$-adjoint of $B_\infty$.
  \item
    an eigenvector $\xi\in W$ for $B_\infty$, which is a cyclic generator of $W$ with respect to $B_0$.
  \end{enumerate}
  In both cases of the theorem, the vector space $W$ will be identified with $G_t/\tau^{-1}G_t$.
  Dubrovin's theorem yields a germ of a universal Frobenius structure on a certain $n$-dimensional manifold such that its
  first structure connection restricts to the data $(W,B_0,B_\infty, g,\xi)$ over the origin.
  The universality property then induces a Frobenius structure on the germ $(M_t,0)$, as the tangent space of the latter at the origin
  is canonically identified with $T^1_{\s R_h/\CC} f/\mathfrak{m}_t\cdot T^1_{\s R_h/\CC} f\cong
  (T^1_{\s R_h/\CC} f/\mathfrak{m}_t\cdot T^1_{\s R_h/\CC} f)\cdot\alpha \cong G_t/\tau^{-1}G_t$.

  Let us show how to construct the initial data needed in case (i) and (ii) of the theorem:
  \begin{enumerate}
  \item
    We put $W:=H^0(\PP^1,\widehat{G}_t)$, $g:=\tau^{n-1} S$ (notice that this is possible due to
    theorem \ref{theoS-solution}), $B_0:=[\nabla_\tau]\in\End_{\CC}(G_t/\tau^{-1} G_t)\cong \End_{\CC}(W)$ and
    $B_\infty:=[\tau\nabla_\tau]\in \End_{\CC}(\widehat{G}_t/\tau \widehat{G}_t)\cong \End_{\CC}(W)$. In order to verify
    the conditions from above on these initial data, consider the basis $\underline{\omega}^{(3)}$ of $W$. Then
    $B_0$ is given by the matrix $\Omega_0$, which is obviously semi-simple with distinct eigenvalues (these are the critical values of $f_{|D_t}$).
    It is self-adjoint due to the flatness of $S$. The endomorphism $B_\infty$ corresponds to the
    matrix $A^{(3)}_\infty$, so that the symmetry of the spectrum as well as the proof of lemma
    \ref{lemSpectrum} show that $B_\infty+B^*_\infty=(n-1)\Id$.
    Finally, it follows from corollary \ref{summary} that for all $i\in\{1,\ldots,n\}$, the class of $\omega^{(3)}_i$ in 					 $G_t/\tau^{-1} G_t$ is equal to the class of $(-f)^{i-1}\alpha$. By definition, $B_0=[\nabla_\tau]$ is the multiplication by $-f$
    on $W \cong G_t/\tau^{-1}G_t$, hence, any of the classes of the sections $\omega^{(3)}_i$ is a cyclic
    generator of $W$ with respect to $[\nabla_\tau]$. It is homogenous, i.e., an eigenvector of $B_\infty$ by construction. This proves
    the theorem in case (i).
  \item
    First notice that it follows from \cite[appendix B.b.]{DS} that the space $H^0(\PP^1,\widehat{G}_t)\cap V_{\alpha_{min}}$ is independent
    of the choice of the $V^+$-solution $\widehat{G}_t$ of the Birkhoff problem for $(G_t,\nabla)$. In particular, we
    have $\omega^{(3)}_i\in H^0(\PP^1,\widehat{G}^{can}_t)$ if $\nu^{(3)}_i=\alpha_{min}$. Now put
    $W:=H^0(\PP^1,\widehat{G}^{can}_t)$ and again $g:=\tau^{n-1} S$,  $B_0:=[\nabla_\tau]\in\End_{\CC}(G_t/\tau^{-1} G_t)\cong \End_{\CC}(W)$ and
    $B_\infty:=[\tau\nabla_\tau]\in \End_{\CC}(\widehat{G}^{can}_t/\tau \widehat{G}^{can}_t)\cong \End_{\CC}(W)$.
    The eigenvalues of the endomorphism $B_0$ are always the critical values of $f_{|D_t}$ so as in (i) it follows that
    $B_0$ is semi-simple with distinct eigenvalues. It is $g$-self-adjoint by the same argument as in (i).
    The endomorphism $B_\infty$ is also semi-simple, as $\widehat{G}^{can}_t$ is a $V^+$-solution. The section $\omega^{(3)}_i$ is an eigenvalue
    of $B_\infty$, i.e. homogeneous. The property $B_\infty+B^*_\infty=(n-1)\Id$ follows as in (i)
    by the fact that $\widehat{G}_t$ is also a $(V^+,S)$-solution (more precisely, by choosing a basis $\underline{w}$ of $W$ such that
    $B_\infty$ is again given by the matrix $A^{(3)}_\infty$ and such that $g(w_i,w_j)= 1$ if $\nu^{(3)}_i+\nu^{(3)}_j = n-1$ and $g(w_i,w_j)=0$ otherwise).
    Finally, the fact that $\omega^{(3)}_i$ is primitive also follows by the argument given in (i).
  \end{enumerate}
\end{proof}
The previous theorem yields for fixed $i$ Frobenius structures $M_t^{(i)}$
for any $t\neq 0$. One might ask whether they are related
in some way. It turns out that for a specific choice of the index $i$ they are (at least in the reductive case),
namely, one of them can be seen as
analytic continuation of the other. The proof relies on the fact
that it is possible to construct a Frobenius structure from the bundle
$G$ simultaneously for all values of $t$ at least on a small disc
outside of $t=0$.
This is done using a generalization of Dubrovins
theorem, due to Hertling and Manin \cite[theorem 4.5]{HM}. In
loc.cit., Frobenius manifolds are constructed from so-called
``trTLEP-structures''. The following result shows how they arise
in our situation.
\begin{lemma}\label{lemApplicationHM}
Suppose that $D$ is reductive. Fix $t\in T\backslash\{0\}$ and suppose that the minimal spectral number
$\alpha_{min}$ of $(G_t,\nabla)$ has multiplicity one, so that
theorem \ref{theoS-solution} applies. Let $\Delta_t$ be a sufficiently small disc centered at
$t$. Denote by $\widehat{H}^{(t)}$ the restriction to $\Delta_t$ of the analytic bundle corresponding to
$\widehat{G}[t^{-1}]$. Then $\widehat{H}^{(t)}$ underlies a trTLEP-structure
on $\Delta_t$, and any of the sections $t^{-k}\omega^{(3)}_i$ satisfy the
conditions (IC), (GC) and (EC) of \cite[theorem 4.5]{HM}. Hence, the
construction in loc.cit. yields a universal Frobenius structure on
a germ $(\widetilde{M}^{(i)},t):=(\Delta_t\times\CC^{n-1},(t,0))$.
\end{lemma}
\begin{proof}
That $\widehat{H}^{(t)}$ underlies a trTLEP-structure is a consequence
of corollary \ref{summary} (i) and theorem \ref{theoS-solution}. We have
already seen that the sections $t^{-k}\omega^{(3)}_i$ are homogenous and
primitive, i.e., satisfy conditions (EC) and (GC) of loc.cit. It follows from the connection form
computed in corollary \ref{summary} (ii), that they also satisfy
condition (IC). Thus the theorem of Hertling and Manin gives a universal Frobenius structure on
$\widetilde{M}^{(i)}$ such that its first structure connection restricts
to $\widehat{H}^{(t)}$ on $\Delta_t$.
\end{proof}
In order to apply this lemma we need to find
a homogenous and primitive section of $\widehat{H}^{(t)}$ which
is also $\nabla^{res}_t$-flat. This is done in the following lemma.
\begin{lemma}\label{lemNablaTFlat}
Let $D$ be reductive. Consider the $V^+$-solution to the Birkhoff-problem for $(G_0,\nabla)$ resp.
$(G_t,\nabla)$ given by $\underline{\omega}^{(2)}$ resp. $\underline{\omega}^{(3)}$. Then
there is an index $j\in\{1,\ldots,n\}$ such that $\deg(\omega^{(2)}_j)=\nu^{(2)}_i$ and an index
$i\in\{1,\ldots,n\}$ such that $\deg(\omega^{(3)}_i)=\nu^{(3)}_i+k\cdot n$. In particular,
$\nu^{(2)}_j,\nu^{(3)}_i\in\NN$.
Moreover, $\nabla^{res}_t(t^{-k}\omega^{(3)}_i)=0$,
where $\nabla^{res}_t:\widehat{G}/\tau\widehat{G} \rightarrow \widehat{G}/\tau\widehat{G}$ is the residue connection.
\end{lemma}
\begin{proof}
By construction we have $\omega^{(1)}_1=\omega^{(0)}_1=\alpha$, so in particular $\deg(\omega^{(1)}_1)=0$. We also have
$\nu_1^{(1)}=0$. Now it suffices to remark that in algorithm 1 (formula \eqref{eqAlg1}), whenever we have an index
$l\in\{1,\ldots,n\}$ with $\deg(\omega^{(1)}_l)=\nu^{(1)}_l$, then either $\deg(\widetilde{\omega}^{(1)}_l)=\widetilde{\nu}^{(1)}_l$
(this happens if the index $i$ in formula \eqref{eqAlg1} is different
from $l$ and $l+1$)
or $\deg(\widetilde{\omega}^{(1)}_{l-1})=\widetilde{\nu}^{(1)}_{l-1}$ (if $i=l$)
or $\deg(\widetilde{\omega}^{(1)}_{l+1})=\widetilde{\nu}^{(1)}_{l+1}$ (if $i=l+1$). It follows that
we always conserve some index $j$ with $\deg(\widetilde{\omega}^{(1)}_j)=\widetilde{\nu}^{(1)}_j$.
A similar argument works for algorithm 2,
which gives the second statement of the first part.
The residue connection is given by the matrix $\frac{1}{nt}\left(\widetilde{D} + A^{(3)}_\infty\right)$ in the basis
$\underline{\omega}^{(3)}$ of $\widehat{G}/\tau\widehat{G}$ (see corollary \ref{summary} (ii)).
This yields the $\nabla^{res}$-flatness of $t^{-k}\omega^{(3)}_i$.
\end{proof}
Finally, the comparison result can be stated as follows.
\begin{proposition}\label{propAnalytContFrob}
Let $i$ be the index from the previous lemma such that
$\nabla^{res}(t^{-k}\omega^{(3)}_i)=0$. Then for any $t'\in \Delta_t$, the germs of Frobenius structures
$(\widetilde{M}^{(i)},t')$ (from lemma \ref{lemApplicationHM}) and
$(M^{(i)},t')$ (from theorem \ref{theoFrobeniusSemiSimple}) are isomorphic.
\end{proposition}
\begin{proof}
We argue as in \cite[proposition
5.5.2]{Dou3}: The trTLEP-structure $\widehat{H}^{(t)}$ is a
deformation (in the sense of \cite[definition 2.3]{HM}) of
the fibre $\widehat{G}/t'\widehat{G}$, hence contained in the
universal deformation of the latter. Thus the (germs at $t'$ of the)
universal deformations of $\widehat{H}^{(t)}$ and
$\widehat{G}/t'\widehat{G}$
are isomorphic. This gives the result
as the homogenous and primitive section $t^{-k}\omega^{(3)}_i$ of
$\widehat{H}^{(t)}$ that we choose in order to apply lemma \ref{lemApplicationHM}
is $\nabla^{res}$-flat.
\end{proof}

\subsection{Frobenius structures at $t=0$}
\label{subsec:Frobenius-teq0}

In the last subsection, we constructed Frobenius structures on the unfolding spaces $M_t$
for any $t\neq 0$.
It is a natural question to know whether there is a way to attach a Frobenius structure to the
restriction of $f$ on $D$. In order to carry this out, one is faced with the difficulty that
the pairing $S$ from theorem \ref{canSol} is not, a priori, defined on $\mathbf{G}_0$. Hence
a more precise control over
this pairing on $\mathbf{G}[t^{-1}]$ is needed in order to make a statement
at $t=0$. The following conjecture provides exactly this additional information.
\begin{conjecture}\label{conjSpectrum}
The
pairing $S$ from theorem \ref{canSol} is defined on $\mathbf{G}[t^{-1}]$ and
meromorphic at $t=0$, i.e., induces a pairing $S:\mathbf{G}[t^{-1}]\otimes\overline{\mathbf{G}}[t^{-1}]\rightarrow \CC[\tau,\tau^{-1},t,t^{-1}]$.
Moreover, consider the natural grading of $\mathbf{G}$ resp. on $\mathbf{G}[t^{-1}]$ induced from the grading
of $\Omega^{n-1}(\log\,h)$ by putting $\deg(\tau)=-1$ and $\deg(t)=n$.
Then
\begin{enumerate}
\item
$S$ is homogenous, i.e., it sends
$(\mathbf{G}[t^{-1}])_k\otimes (\overline{\mathbf{G}[t^{-1}]})_l$ into
$\CC[\tau,\tau^{-1},t,t^{-1}]_{k+l}$.
%(remember that the ring $\CC[\tau,\tau^{-1},t,t^{-1}]$ is graded by
%$\deg(\tau)=-1$, $\deg(t)=n$).
\item
$S$ sends $G\otimes\overline{G}$ into $\tau^{-n+1}\CC[\tau^{-1},t]$.
%and is non-degenerate on $G_0=G/tG$, i.e., induces a pairing like in definition \ref{Vsolution} (iv) on $G_0$.
\end{enumerate}
\end{conjecture}
Some evidence supporting the first part of this conjecture comes from the computation of the examples in section \ref{sec:examples}.
Namely, it appears that in all cases, there is an extra symmetry satisfied by the spectral numbers, i.e., we
have $\nu^{(3)}_k+\nu^{(3)}_{n+1-k}=n-1$, and not only $\alpha_k+\alpha_{n+1-k}=n-1$ for all $k\in\{1,\ldots,n\}$
(Remember that $\alpha_1,\ldots,\alpha_n$ was the ordered sequence of spectral numbers).
Moreover, the eigenvalues of the residue of $t\partial_t$ on $(G/tG)_{|\tau\neq 0}$ are constant in $\tau$
and symmetric around zero, which indicates that $S$ extends without poles and as a non-degenerate pairing to $G$.
In particular, one obtains a pairing on $G_0$, which would explain the symmetry
$\nu^{(2)}_k+\nu^{(2)}_{n+1-k}=n-1$ observed in the examples
(notice that even the symmetry of the spectral numbers at $t=0$, written as an
ordered sequence, is not
a priori clear).
Notice also that in the case where $D$ is a normal crossing divisor (i.e., the first example
studied in section \ref{sec:examples}), the conjecture is true. This follows from the explicit form
of the pairing $S$ in this case, which can be found in \cite{Dou3}, based on \cite{DS2}.

The following corollary draws some consequences of the above conjecture.
\begin{corollary}\label{corConjSpectrum}
Suppose that conjecture \ref{conjSpectrum} holds true and that the minimal spectral
number $\alpha_{min}$ of $(G_t,\nabla)$, $t\neq 0$ has multiplicity one so that
theorem \ref{theoS-solution} holds. Then
\begin{enumerate}
\item
The pairing $S$ is expressed in the basis $\underline{\omega}^{(3)}$ as
$$
S(\omega^{(3)}_i,\overline{\omega}^{(3)}_j) = \left\{
\begin{array}{rcl}
c \cdot t^{2k} \cdot \tau^{-n+1}& \textup{if} & i+j=n+1 \\
0 & \textup{else}
\end{array}\right.
$$
for some constant $c\in \CC$, where, as before, $k\in\NN$ counts the
number of meromorphic base changes in algorithm 2.
Moreover, we have $\nu^{(3)}_i+\nu^{(3)}_{n+1-i}=n-1$ for
all $i\in\{1,\ldots,n\}$.
\item
The pairing $S$ is expressed in the basis $\underline{\omega}^{(2)}$ as
$$
S(\omega^{(2)}_i,\overline{\omega}^{(2)}_j) = \left\{
\begin{array}{rcl}
c \cdot \tau^{-n+1}& \textup{if} & i+j=n+1 \\
0 & \textup{else}
\end{array}\right.
$$
for the same constant $c\in \CC$ as in (i).
%\item
%$\underline{\omega}^{(2)}$ yields an $(V^+,S)$-solution in all cases, and moreover,
%the pairing is independent of $t$ in this basis, in other words, it
%extends in a non-degenerate way as a pairing $S:\widetilde{G} \otimes_{\CC[\tau^{-1},t]}\overline{\widetilde{G}} \rightarrow \tau^{-n+1}\CC[\tau^{-1},t]$,
%where $\widetilde{G}:=\oplus_{i=1}^n\CC[\tau^{-1},t]\omega^{(2)}_i$ is the proper free $\CC[\tau^{-1},t]$-submodule of $G$ considered after
%the proof of lemma \ref{Alg2Term}.
%\item
%???? The basis $\underline{\omega}^{(1)}$ is an $(V^+,S)$-solution for the restriction $(G_0,\nabla, \widetilde{S}_{|G_0})$. ????
\item
$S$ extends to a non-degenerate paring on $G$, i.e., it induces a
pairing
$S:G_0\otimes_{\CC[\tau^{-1}]}\overline{G}_0\rightarrow\tau^{-n+1}\CC[\tau^{-1}]$
with all the properties of definition \ref{Vsolution} (iv). Moreover,
$\underline{\omega}^{(2)}$ defines a $(V^+,S)$-solution for the
Birkhoff problem for $(G_0,\nabla)$ with respect to $S$.
\end{enumerate}
\end{corollary}
\begin{proof}
\begin{enumerate}
\item
Following the construction of the bases $\underline{\omega}^{(1)}$, $\underline{\omega}^{(2)}$ and
$\underline{\omega}^{(3)}$, starting from
the basis $\underline{\omega}^{(0)}$ (via lemma \ref{lemBirkhoff} and algorithms 1 and 2),
it is easily seen that $\deg(\omega^{(1)}_i)=\deg(\omega^{(2)}_i)=i-1$ and that
$\deg(\omega^{(3)}_i)=k\cdot n+i-1$.
The (conjectured) homogeneity of $S$ yields that
$\deg(S(\omega^{(2)}_i,\overline{\omega}^{(2)}_j))=i+j-2$ and
$\deg(S(\omega^{(3)}_i,\overline{\omega}^{(3)}_j))=2kn+i+j-2$.

The proof of lemma \ref{lemSpectrum} shows that
$\tau^{n-1}S(\omega^{(3)}_i,\overline{\omega}^{(3)}_j)$ is either zero or constant in
$\tau$, hence, by part (ii) of conjecture \ref{conjSpectrum}, $S(\omega^{(3)}_i,\overline{\omega}^{(3)}_j)=c(t)\cdot\tau^{-n+1}$, with $c(t)\in\CC[t]$, which is actually homogenous by part (i) of conjecture \ref{conjSpectrum}. Now since $i+j-2<2(n-1)$, $\deg(c(t)\cdot\tau^{-n+1})=2kn+(i+j-2)$ is only
possible if $i+j=n+1$, and then $c(t)=c\cdot t^{2k}$, in
particular, the numbers $c_{kl}$ in lemma \ref{lemSpectrum} (ii) are
always equal to one, and we have $\nu^{(3)}_i+\nu^{(3)}_j=n-1$.

% From $S(\omega^{(1)}_i,\omega^{(1)}_j)\in\tau^{-n+1}\CC[\tau^{-1},t]$, we see that
% $S(\omega^{(1)}_i,\omega^{(1)}_j)=0$ if $i+j < n+1$ (i.e., $\deg(\omega^{(1)}_i)+\deg(\omega^{(1)}_i)<n-1$), since otherwise
% by homogeneity of $S$ we would have $S(\omega^{(1)}_i,\omega^{(1)}_j)\in\CC[\tau^{-1},t]_{<n-1}$.
% On the other hand, if $i+j > n+1$
% then necessarily we have $i+j\in\{n+2,\ldots,2n\}$ (i.e., $\deg(\omega^{(1)}_i)+\deg(\omega^{(1)}_j)\in\{n,\ldots,2n-2\}$), and,
% again by homogeneity of $S$, we must have $S(\omega^{(1)}_i,\omega^{(1)}_j)\in\tau^{-n}\CC[t,\tau^{-1}]$ since for any
% $f\in\tau^{-n+1}\CC[t]$, we have $\deg(f)\notin\{n,2n-2\}$. But $S(\omega^{(1)}_i,\omega^{(1)}_j)\in\tau^{-n}\CC[t,\tau^{-1}]$
% contradicts the assumption that $\tau^{n-1}S$ induces a non-degenerate pairing on $G[t^{-1}]/\tau^{-1}G[t^{-1}]$.

\item
Using (i), one has to analyse the behaviour of $S$ under the base
changes inverse to
\ref{eqAlg1} (algorithm 1)
and \ref{eqAlg2} (algorithm 2). Suppose that $\underline{\omega}$ is
a basis of $\mathbf{G}[t^{-1}]$ with $\deg(\omega_i)=l\cdot
n+i-1$, $l\in\{0,\ldots,k\}$ and such that
$S(\omega_i,\overline{\omega}_j)=c \cdot t^{2l} \cdot
\tau^{-n+1} \cdot \delta_{i+j,n+1}$,
then if we define for any $i\in\{1,\ldots,n\}$ a new basis
$\underline{\omega}'$ by
\begin{equation}
\begin{array}{rcl}
  \omega'_i & :=  & \omega_i-\tau^{-1}\cdot\nu\cdot\omega_{i-1},\\
  \omega'_j & :=  & \omega_j\quad\forall j\neq i,
\end{array}
\end{equation}
where $\nu\in\CC$ is any constant, we see that we still have
$S(\omega'_i,\overline{\omega}'_j)=c \cdot t^{2l} \cdot \tau^{-n+1} \cdot
\delta_{i+j,n+1}$.
Notice that if $j=i+1$ and $i+j=n+1$, then in order to show $S(\omega'_i,\overline{\omega}'_i)=0$, one uses that if $i+(i-1)=n+1$, then
$S(\omega_i,\overline{\omega}_{i-1})=(-1)^{n-1}\overline{S(\omega_{i-1},\overline{\omega}_i)}
=S(\omega_{i-1},\overline{\omega}_i)$ since
$S(\omega_{i-1},\overline{\omega}_i)$ is
homogenous in $\tau^{-1}$ of degree $-n+1$.

Similarly, if we put, for any constant $\nu\in\CC$,
\begin{equation}
\begin{array}{rcl}
  \omega''_1 & := & t^{-2}\omega_1-t^{-1}\tau^{-1}\cdot\nu\cdot\omega_n, \\
  \omega''_i & := & t^{-1}\omega_i\quad\forall i\neq 1,
\end{array}
\end{equation}
then we have $S(\omega''_i,\overline{\omega}''_j)=c \cdot t^{2(l-1)} \cdot \tau^{-n+1} \cdot
\delta_{i+j,n+1}$.
\item
This follows from (ii) and the fact that $\underline{\omega}^{(2)}$ is a
$V^+$-solution for $(G_0,\nabla)$.
\end{enumerate}
\end{proof}

As a consequence, we show that under the hypothesis of conjecture \ref{conjSpectrum}, we obtain indeed
a Frobenius structure at $t=0$.
\begin{theorem}\label{theoLimitFrob}
Suppose that conjecture \ref{conjSpectrum} holds true and that the minimal spectral number $\alpha_{min}$
of $(G_t,\nabla)$ for $t\neq 0$ has multiplicity one, so that theorem \ref{theoS-solution} applies.
Then the (germ at the origin of the) $\mathcal{R}_h$-deformation space of $f$, which we call
$M_0$, carries a Frobenius structure, which is constant, i.e., given by a potential of degree at most three
(or, expressed otherwise, such that the structure constants $c^k_{ij}$ defined by $\partial_{t_i} \circ \partial_{t_j}
=\sum_k c^k_{ij} \partial_{t_k}$ are constant in the flat coordinates $t_1,\ldots,t_n$).
%It is the limit of
%the Frobenius structures $M^{(1)}_t$ for $t\rightarrow 0$ if $\underline{\omega}^{(3)}=\underline{\omega}^{(2)}$.
\end{theorem}
\begin{proof}
Remember that $(M_0,0)$ is a smooth germ of dimension $n$,
with tangent space given by $T^1_{\mathcal{R}_h}f \cong
G_0/\tau^{-1}G_0$ (notice that the deformation functor in question
is evidently unobstructed). As usual, a $\mathcal{R}_h$-semi-universal
unfolding of $f$ is given as
$$
F=f+\sum_{i=1}^n u_i g_i :V \times M_0 \longrightarrow \CC,
$$
where $u_1,\ldots,u_n$ are coordinates on $M_0$ and $g_1,\ldots,g_n$ is a basis of $T^1_{\mathcal{R}_h}f$.

In order to endow $M_0$ with a Frobenius structure, we will use a similar strategy as in subsection
\ref{subsec:Frobenius-tneq0}, namely, we construct
 a germ of an $n$-dimensional Frobenius manifold
which induces a Frobenius structure on $M_0$ by a universality
property. The case we need here has been treated by Malgrange
(see \cite[(4.1)]{Mal4}). The theorem of Hertling and Manin (\cite[theorem 4.5]{HM}) can be considered as a common
generalisation of Malgrange's result and of the
constructing of Duborovin used lemma \ref{lemApplicationHM}.
We use the result in the form that can be found in \cite[remark 4.6]{HM}.
Thus we have to construct a \emph{Frobenius type structure} on a point, and a section satisfying the
conditions called (GC) and (EC) in loc.cit. This is nothing but a
tuple $(W,g,B_0,B_\infty,\xi)$ as in the proof of theorem
\ref{theoFrobeniusSemiSimple}, except that we do not require the
endomorphism $B_0$ to be semi-simple, but to be \emph{regular}, i.e.,
its characteristic and minimal polynomial must coincide. Consider the $(V^+,S)$-solution
defined by $\widehat{G}_0:= \oplus_{i=1}^n \OO_{\PP^1\times\{0\}}
\omega^{(2)}_i$, and put, as before, $W:=H^0(\PP^1,\widehat{G}_0)$,
$g:=\tau^{-n+1} S$, $B_0:=[\nabla_\tau]$ and $B_\infty:=[\tau\nabla_\tau]$.
Considering the matrices $(\Omega_0)_{|t=0}$ resp. $A^{(2)}_\infty$ of $B_0$ resp. $B_\infty$
with respect to the basis $\underline{\omega}^{(2)}$ of $W$, we see immediately
that $g(B_0-,-)=g(-,B_0-)$, $g(B_\infty-,-)=g(-,(n-1)\Id-B_\infty-)$ and that $B_0$ is
regular since $(\Omega_0)_{|t=0}$ is nilpotent with a single Jordan block. Notice that the
assumption that conjecture \ref{conjSpectrum} holds is used through corollary
\ref{corConjSpectrum} (ii), (iii). The section
$\xi:=\omega^{(2)}_1$ is obviously homogenous and primitive, i.e., satisfies (EC) and (GC).
Notice that it is, up to constant multiplication, the only primitive and homogenous section, contrary to the case
$t\neq 0$, where we could chose any of the sections $\omega^{(3)}_i$, $i\in\{1,\ldots,n\}$.
We have thus verified all conditions of the theorem of Hertling and Manin, and obtain, as indicated
above, a Frobenius structure on $M_0$.
%, which might be considered as a limit $\lim_{t\rightarrow 0}M_t^{(1)}$
%gif the two modules $G$ and $G^{(3)}$ coincide, i.e., if $\omega^{(2)}=\omega^{(3)}$ (equivalently, if $k=0$).

It remains to show that it is given by potential of degree at most three. The argument is exactly the same as
in \cite[lemma 6.4.1 and corollary 6.4.2.]{Dou3} so that we omit the details here.
%The main point
%is that if one uses the unfolding theorem \cite[theorem 4.5]{HM}, then it is easily seen that the matrices
%of the Higgs field in the natural basis given by the flat extension of $\underline{\omega}^{(2)}$ are
%polynomials in the matrix $(\Omega_0)_{|t=0}$, but their first columns are fixed, by the fact
%that $\omega^{(2)}_1$ is the primitive form, i.e., corresponds to the unit field on $M_0$. In particular,
%they are constant.
\end{proof}

%\begin{remark}
%The following gives a heuristic argument how the constancy of the limit Frobenius structure
%is related to the left-right stability of $f_{|D}$ as studied in subsection \ref{subsec:LeftRight}.
%
%
%\textbf{...TO DO...}
%\end{remark}

\subsection{Logarithmic Frobenius structures}

\label{subsec:LogFrobenius}

The pole order property of the connection $\nabla$ on $G$ (see formula \ref{eqIntConnection}) suggests
that the family of germs of Frobenius manifolds $M_t$ studied above can be put together in a single Frobenius
manifold with a \emph{logarithmic} degeneration behaviour at the divisor $t=0$. We show that this is actually the
case for the normal crossing divisor;
the same result has been obtained from a slightly different viewpoint in
\cite{Dou3}. In the general case, we observe a phenomenon which also occurs in loc.cit.: one obtains a
Frobenius manifold where the multiplication is defined on the logarithmic tangent bundle, but the metric
might be degenerate on it (see loc.cit., section 7.1.).

We recall the following definition from \cite{Reich1}, which we extend to the more general situation studied here.
\begin{definition}
\begin{enumerate}
\item
Let $M$ be a complex manifold and $\Sigma\subset M$ be a normal crossing divisor. Suppose that
$(M\backslash\Sigma,\circ,g,E,e)$ is a Frobenius manifold. One says that it has a logarithmic
pole along $\Sigma$ if $\circ\in\Omega^1(\log\,\Sigma)^{\otimes 2}\otimes\Der(-\log\,\Sigma)$,
$g\in\Omega^1(\log\,\Sigma)^{\otimes 2}$ and $g$ is non-degenerate as a pairing on $\Der(-\log\,\Sigma)$.
\item
If, in the previous definition, we relax the condition of $g$ being non-degenerate on $\Der(-\log\,\Sigma)$,
then we say that $(M,\Sigma)$ is a weak logarithmic Frobenius manifold.
\end{enumerate}
\end{definition}
In \cite{Reich1}, logarithmic Frobenius manifolds are constructed
using a generalisation of the main theorem of \cite{HM}. More
precisely, universal deformations of so-called
``log$\Sigma$-trTLEP-structures'' (see \cite[definition
1.8.]{Reich1}) are constructed. In our situation, the base of such an object is the
space $T$, and the divisor $\Sigma:=\{0\}\subset T$. In order to adapt
the construction to the more general situation that we discuss here,
we define a weak log$\Sigma$-trTLEP-structures to be such a vector
bundle on $\PP^1\times T$ with connection and pairing, where the latter is
supposed to be non-degenerate only on $\PP^1\times (T\backslash\Sigma)$.
The result can then be stated as follows.
\begin{theorem}\label{theoLogFrob}
Let $D$ be reductive, $i\in\{1,\ldots,n\}$ be the index from
lemma \ref{lemNablaTFlat} such that
$\deg(t^{-k}\omega_i)=-\nu^{(3)}_i$
and suppose that the minimal spectral number $\alpha_{min}$ of $(G_t,\nabla)$ has
multiplicity one (so that theorem \ref{canSol} applies). Then
the (analytic bundle corresponding to the) module
$$
\begin{array}{rcl}
\widehat{G}' & := & \bigoplus_{j=1}^n\OO_{\PP^1\times T}
\omega^{(4)}_j \quad\textup{where} \\ \\
\omega^{(4)}_j & := & t^{-k}\omega^{(3)}_j \quad\forall j\in\{i,\ldots,n\}\\ \\
\omega^{(4)}_j & := & t^{-k+1}\omega^{(3)}_j \quad\forall j\in\{1,\ldots,i-1\}
\end{array}
$$
underlies a weak $\log\Sigma$-trTLEP-structure, and a
$\log\Sigma$-trTLEP-structure if conjecture \ref{conjSpectrum} holds true and if $i=1$. The form
$t^{-k}\omega^{(3)}$ is homogenous and primitive and yields a weak
logarithmic Frobenius manifold. It yields a logarithmic Frobenius manifold
if conjecture \ref{conjSpectrum} holds true and if $i=1$, e.g., in the case of a linear section $f$ of the normal crossing divisor.
\end{theorem}
\begin{proof}
It is clear by definition that $(\widehat{G}',\nabla, S)$ is a
weak $\log\Sigma$-trTLEP-structure (of weight $n-1$).
It is easy to see that the connection takes the form
$$
\nabla(\underline{\omega}^{(4)}) =
\underline{\omega}^{(4)} \cdot
\left[
(\Omega_0\tau+A^{(4)}_\infty)\frac{d\tau}{\tau}
+
(\Omega_0\tau+\widetilde{A}^{(4)}_\infty)\frac{dt}{nt}
\right],
$$
where
$$
\begin{array}{rcl}
A^{(4)}_\infty & = &
\mbox{diag}(-\nu^{(3)}_i,\ldots,-\nu^{(3)}_n,-\nu^{(3)}_1,\ldots,-\nu^{(3)}_{i-1})
\\ \\
\widetilde{A}^{(4)}_\infty & = & A^{(4)}_\infty+
\mbox{diag}\left(\deg(\omega^{(4)}_i),\ldots,\deg(\omega^{(4)}_n),\deg(\omega^{(4)}_1),\ldots,\deg(\omega^{(4)}_{i-1})\right).
\end{array}
$$
In particular, $\omega^{(4)}_1$ is $\nabla^{res}$-flat,
$[\nabla_\tau]$-homogenous and a cyclic generator of $H^0(\PP^1\times\{0\},\widehat{G}'/t
\widehat{G}')$ with respect to $[\nabla_\tau]$ and
$[\tau^{-1}\nabla_{t\partial_t}]$ (even with respect to $[\nabla_\tau]$ alone).
Moreover, $[\tau^{-1}\nabla_{t\partial_t}(\omega^{(4)}_1)]$ is
non-zero in $H^0(\PP^1\times\{0\},\widehat{G}'/t
\widehat{G}')$, so that $\omega^{(4)}_1$ satisfies the conditions
(EC), (GC) and (IC) of \cite[theorem 1.12]{Reich1}, except that the
form $S$ might be degenerate on $\widehat{G}'_{|t=0}$ (correspondingly,
the metric $g$ on $K:=\widehat{G}'/\tau\widehat{G}'$ from
loc.cit. might be degenerate on $K_{|t=0}$). One checks
that the proof of theorem 1.12 of
loc.cit can be adapted to the more general situation and
yields a weak logarithmic Frobenius structure.

Now assume conjecture \ref{conjSpectrum} and suppose that $i=1$. Then
$\underline{\omega}^{(4)}=t^{-k}\underline{\omega}^{(3)}$,
and we get that $S$ is non-degenerate on $\widehat{G}'$ by corollary
\ref{corConjSpectrum}. In particular, $(\widehat{G}', \nabla,S)$
underlies a $\log\Sigma$-trTLEP-structure in this case. This yields
a logarithmic Frobenius structure by applying \cite[theorem
1.12]{Reich1}. That the pairing is non-degenerate and that $i=1$ holds
for the normal crossing divisor case follows, e.g., from the
computations in \cite{DS2} (which, as already pointed out above, have
been taken up in \cite{Dou3} to give the same result as here).
\end{proof}
Let us remark that one might consider the result for the normal crossing divisor as being ``well-known'' by the mirror principle:
as already stated in the introduction, the Frobenius structure for fixed $t$ is known to be isomorphic
to the quantum cohomology ring of the ordinary projective space. But in fact we have more: the parameter $t$
corresponds exactly to the parameter in the small quantum cohomology ring
(note that the convention for the name of the coordinate on the parameter
space differs from the usual one in quantum cohomology, our $t$ is usually called $q$ and
defined as $q=e^t$, where this $t$ corresponds to a basis vector in the second cohomology
of the underlying variety, e.g., $\PP^{n-1}$). Using this interpretation, the logarithmic
structure as defined above is the same as the one obtained in \cite[subsection 2.1.2.]{Reich1}.
In particular, it is easily seen that the deformation algebra
$T^1_{\mathcal{R}_h/\CC} f=\CC[V]/df(\Der(-\log\,h))
=\CC[x_1,\ldots,x_n]/(x_1-x_2,\ldots,x_1-x_n)\cong\CC[x_1]$ specializes to $H^*(\PP^{n-1},\CC)=\CC[x_1]/(x_1^n)$ over
$t=0$ (and more generally to $\CC[x_1]/(x_1^{n-1}-t)$ at $t\in T$, i.e., to the small quantum product
of $\PP^{n-1}$ at the point $t\in H^2(\PP^{n-1},\CC)/H^2(\PP^{n-1},\ZZ)$).

%\clearpage

\section{Examples}
\label{sec:examples}

We have computed the spectrum and monodromy for some of the discriminants in quiver
representation spaces described in \cite{bm}. In some cases, we have implemented the
methods explained in the previous sections in {\sc Singular} (\cite{Greuel:2005aa}). For
the infinite families given in table \ref{tab:1} below, we have solved the Birkhoff
problem by essentially building the seminvariants $h_i$, where $h=h_1\cdots h_k$ is the
equation of $D$, by successive multiplication by $(-f)$.

We will present two types of examples. On the one hand, we will explain
in detail some specific ones, namely, the normal crossing divisor,
the star quiver with three exterior vertices (denoted by $\star_3$ in
example \ref{D_4/E_6} (i)), and the non-reductive example discussed
after definition \ref{mond:specdef} for $k=2$.
We also give the spectral numbers for the linear free divisor associated to the $E_6$
quiver (see example \ref{D_4/E_6} (ii)), but we do not write down the corresponding
good basis, which is quite complicated (remember that already the
equation of this divisor \eqref{eqE6} was not completely given).

On the other hand, we are able to determine the spectrum for
$(G_t,\nabla)$ ($t\neq 0$) and $(G_0,\nabla)$ for the whole $D_n$- and
$\star_n$-series by a combinatorical procedure. The details are rather
involved; therefore we present the results, but refer to the
forthcoming paper \cite{CompQuiver} for full details and
proofs. It should be noticed that except in the case of the normal
crossing divisor and in very small dimensions for other examples,
it is hard to write down explicitly
elements for the good bases $\underline{\omega}^{(2)}$ and
$\underline{\omega}^{(3)}$  as already the equation for the divisor
becomes quickly quite involved.

Let us start with the three explicit examples mentioned above.
%Notice that
%is all examples discussed below the linear function is the sum of the coordinates,
%which is, in each case, easily seen to lie in the open orbit of the dual action.

\textbf{The case of the normal crossing divisor:} As noticed in the first section,
this is the discriminant in the representation space $\Rep(Q,\11)$ of any quiver with
a tree as underlying (oriented) graph. In particular, it is the
discriminant of the Dynkin $A_{n+1}$-quiver. Chosing coordinates
$x_1,\ldots,x_n$ on $V$, we have $h=x_1\cdot\ldots\cdot x_n$. The linear
function $f=x_1+\ldots+x_n$ is $\mathcal{R}_h$-finite, and a direct calculation
(i.e., without using lemma \ref{eqBirkhoff} and algorithm 1) shows that
$\underline{\omega}^{(1)}=\underline{\omega}^{(2)}=\underline{\omega}^{(3)}
=\left((-n)^{i-1}\prod_{j=1}^{i-1} x_j \cdot \alpha\right)_{i=1,\ldots,n}$.
%, where $\alpha=\left(d\,x_0/x_0\wedge\ldots\wedge d\,x_n/x_n\right)/\left(dh/h\right)$.
This is consistent with the basis found in \cite[proposition 3.2]{DS2}.
In particular, we have $A^{(2)}=A^{(3)}=-\diag(0,\ldots,n-1)$, so the spectral numbers
of $(G_t,\nabla)$ for $t\neq 0$ and $(G_0,\nabla)$ are $(0,\ldots,n-1)$.
We also see that
$(nt\partial_t)\underline{\omega}^{(2)}=\underline{\omega}^{(2)}\cdot\tau\Omega_0$,
which is a well known result from the calculation of the quantum cohomology of
$\PP^{n-1}$ (see the last remark in subsection \ref{subsec:LogFrobenius}).

\vspace*{0,2cm}

\textbf{The case $\star_3$ (see example \ref{D_4/E_6} (i))}:
Remember that we had chosen coordinates $a_{11},\ldots,a_{23}$ on the
space $V=M(2\times 3,\CC)$ and that
$h=(a_{11}a_{22}-a_{12}a_{21})
(a_{11}a_{23}-a_{13}a_{21})
(a_{12}a_{23}-a_{22}a_{13}).
$
Defined as a discriminant in a quiver representation space,
this linear free divisor is reductive, and it follows from proposition
\ref{propUnitaryCoord} that the dual divisor has the same equation in dual
coordinates. Then the linear form
$f=a_{11}+a_{21}+a_{22}+a_{23}$ is $\mathcal{R}_h$-finite, as it
does not lie in the dual divisor.

In the next step, we will actually not make use of the basis $\underline{\omega}^{(0)}=((-f)^i\cdot\alpha)_{i=0,\ldots,n-1}$,
but instead compute a basis $\underline{\omega}^{(1)}$ which gives
a solution to the Birkhoff problem directly. Namely, we
write
$$
\begin{array}{rcl}
\Delta_1 & := & a_{13} a_{22} - a_{12} a_{23} \\
\Delta_1 & := & a_{11} a_{23} - a_{21} a_{13} \\
\Delta_1 & := & a_{21} a_{12} - a_{11} a_{22}
\end{array}
$$
for the equations of the components of $D$, and define
linear forms
$$
\begin{array}{rcl}
l_1 & := & \frac{1}{2}a_{13}  \\
l_2 & := & \frac{1}{2}(a_{23}-a_{13})  \\
l_3 & := & \frac{1}{2}a_{22}
\end{array}
$$
Using these notations, we have that $\omega^{(1)}$ is given as follows.
\begin{equation}\label{eqBases1Star3}
\begin{array}{rclcrclcrcl}
\omega^{(1)}_1 & = & \alpha &;& \omega^{(1)}_2 & = & -12\cdot l_1 \cdot \alpha &;& \omega^{(1)}_3 & = & -12\cdot \Delta_1 \cdot \alpha \\ \\
\omega^{(1)}_4 & = & -12^2 \cdot \Delta_\cdot l_2 \cdot \alpha &;& \omega^{(1)}_5 & = & -12^2 \cdot \Delta_1 \cdot \Delta_2 \cdot \alpha &;& \omega^{(1)}_6 & = & -12^3 \cdot \Delta_1 \cdot \Delta_2 \cdot l_3 \cdot \alpha.
\end{array}
\end{equation}
and one calculates that $A^{(1)}_\infty = \diag(-0,-3,-2,-3,-4,-3)$. Algorithm 1 yields
$\omega^{(2)}_2=\omega^{(1)}_2+2\tau^{-1}\omega^{(1)}_1$ and $\omega^{(2)}_i=\omega^{(1)}_i$ for all $i\neq 2$,
and we obtain $A^{(2)}_\infty = \diag(-2,-1,-2,-3,-4,-3)$. As $\nu^{(2)}_1-\nu^{(2)}_6=-1 \leq 1$, we have
$\underline{\omega}^{(2)}=\underline{\omega}^{(3)}$, hence $G^{(3)}=G$ and $(2,1,2,3,4,3)$ is the spectrum
for $(G_t,\nabla)$, $t\neq 0$ as well as for $(G_0,\nabla)$. We see that the minimal spectral number is unique,
therefore, $\underline{\omega}^{(2)}$ yields a $(V^+,S)$-solution for any $t$. Moreover, we have
$(nt\partial_t)\underline{\omega}^{(2)}=\underline{\omega}^{(2)}\cdot\left[\tau\Omega_0+\diag(-2,0,0,0,0,2)\right]$,
so that in this case the $\nabla^{res}$-flat section $\omega^{(3)}_i$ from lemma
\ref{lemNablaTFlat} is $\omega^{(2)}_2$, which is an eigenvector of $A^{(2)}_\infty$ with respect to the minimal spectral number.

\vspace*{0,5cm}

\textbf{The case $E_6$ (see example \ref{D_4/E_6} (ii))}: In the given coordinates $a,b,\ldots,v$ of $V$,
we chose the linear form $f=(a,b,\ldots,v)\cdot{^t}(1,2,0,1,3,0,1,3,2,1,0,2,1,3,0,1,3,0,2,1,3,2)$,
which lies in the complement of the dual divisor (again, by reductivity, we have $h^\vee=h^*$). Then
the spectrum of both $(G_t,\nabla)$, $t\neq 0$ and $(G_0,\nabla)$ is
$$
\left(\frac{44}{5},\frac{25}{3},\frac{28}{3},\frac{31}{3},\frac{34}{3},\frac{47}{5}, \underbrace{6,\ldots,15}_{10 \textup{ elements}},\frac{58}{5},\frac{29}{3},\frac{32}{3},\frac{35}{3},\frac{38}{3},\frac{61}{5}\right)
$$
Again we have a unique minimal spectral number, hence, theorem \ref{theoS-solution} applies. The symmetry
$\nu^{(3)}_i+\nu^{(3)}_{n+1-i}=21=n-1$ holds. Moreover, we obtain the following eigenvalues
for the residue of $t\partial_t$ on $G_{|\CC^*\times T}$ at $t=0$:
$$
\left(-\frac{2}{5},\left(-\frac{1}{3}\right)^4,-\frac{1}{5}, 0^{10},\frac{1}{5},\left(\frac{1}{3}\right)^4,\frac{2}{5}\right)
$$
which are (again) symmetric around zero (hence supporting conjecture \ref{conjSpectrum} (ii)).

\vspace*{0,2cm}

\textbf{A non-reductive example in dimension 3 (see \eqref{eqNonReductiveDim3})}: The linear free divisor
in $\CC^3$ with equation $h=x(xz-y^2)$ is not special and therefore not reductive. The dual divisor is given,
in dual coordinates $X,Y,Z$ by $h^\vee=Z(XZ-Y^2)\neq h^*(X,Y,Z)$. As an $\mathcal{R}_h$-finite linear form, we choose $f=x+z\in V^\vee\backslash D^\vee$.
The basis $\underline{\omega}^{(1)}$ is given as
\begin{equation}
\omega^{(1)}_1 = \alpha \quad ; \quad \omega^{(1)}_2 = (-f) \cdot \alpha \quad ; \quad \omega^{(1)}_3 = \frac92f^2 \cdot \alpha
\end{equation}
and we have $A^{(1)}_\infty = \diag(0,-\frac74,-\frac54)$. Algorithm 1 yields
\begin{equation}
\omega^{(2)}_1 = \omega^{(1)}_1 \quad ; \quad \omega^{(2)}_2 = \omega^{(1)}_2 +\frac34\tau^{-1} \omega^{(1)}_1 \quad ; \quad \omega^{(2)}_3 = \omega^{(1)}_3
\end{equation}
and $A^{(1)}_\infty = \diag(-\frac34,-1,-\frac54)$. Again, as $\nu^{2}_1-\nu^{2}_3=-\frac12\leq 1$, we obtain
$\underline{\omega}^{(2)}=\underline{\omega}^{(3)}$, $G=G^{(3)}$, and $(\frac34,1,\frac54)$ is the spectrum
of both $(G_t,\nabla)$, $t\neq 0$ and $(G_0,\nabla)$.
We can also compute the spectral numbers for the case $k=3,4$ and $5$ (these are again the same
for $(G_t,\nabla)$, $t\neq 0$ and $(G_0,\nabla)$), namely:
\small\renewcommand{\baselinestretch}{1.4}\normalsize
$$
\begin{array}{c|c|c}
\textbf{size of matrices} & \dim(V) & \textbf{Spectrum of } (G_t,\nabla) \\ \hline
k=3 & n=6  &  (2,\frac52,2,3,\frac52,3) \\ \hline
k=4 & n=10 &  (\frac{15}4,\frac{13}3,\frac92,\frac{17}4,4,5,\frac{19}4,\frac92,\frac{14}3,\frac{21}4) \\ \hline
k=5 & n=15 & (6,\frac{53}{8},7,\frac{27}{4},7,\frac{55}{8},6,7,8,\frac{57}{8},7,\frac{29}{4},7,\frac{59}{8},8)
\end{array}
$$
\small\renewcommand{\baselinestretch}{1}\normalsize
%$$
%\begin{array}{c|c|c|c}
%\textbf{size of matrices} & \dim(V) & \textbf{Spectrum of } (G_t,\nabla) & \textbf{Residue eigenvalues of } [t\partial_t]\\ \hline \\
%k=3 & n=6  &  \scriptstyle (2,\frac52,2,3,\frac52,3) & \scriptstyle \frac1{6}(2,-\frac32,0,0,\frac32,2)\\ \\ \hline \\
%k=4 & n=10 & \scriptstyle (\frac{15}4,\frac{13}3,\frac92,\frac{17}4,4,5,\frac{19}4,\frac92,\frac{14}3,\frac{21}4) & \scriptstyle \frac{1}{10}(-\frac{15}4,-\frac{10}3,-\frac52,-\frac54,0,0,\frac54,\frac52,
%\frac{10}3,\frac{15}4)\\ \\ \hline \\
%k=5 & n=15 & \scriptstyle(6,\frac{53}{8},7,\frac{27}{4},7,\frac{55}{8},6,7,8,\frac{57}{8},7,\frac{29}{4},7,\frac{59}{8},8) &
%\scriptstyle\frac1{15}(-6,-\frac{45}8,-5,-\frac{15}4,-3,-\frac{15}8,0,0,0,\frac{15}8,3,\frac{15}4,5,\frac{45}8,6) \\
%\end{array}
%$$
The case $k=5$ (and also $k=3$) is an example where the minimal spectral number is not unique, hence, theorem
\ref{theoS-solution} does not apply. According to theorem \ref{theoFrobeniusSemiSimple} (ii), we can
take $\omega^{(3)}_1, \omega^{(3)}_2, \omega^{(3)}_4, \omega^{(3)}_6$ and $\omega^{(3)}_7$ as primitive and homogenous
sections for $\widehat{G}_t^{can}$.
However, we observe that the ``extra symmetry''
$\nu^{(3)}_i+\nu^{(3)}_{n+1-i}=n-1$ from corollary \ref{corConjSpectrum} still holds,
which supports conjecture \ref{conjSpectrum}. One might speculate that
although the eigenspace of the smallest spectral number is two-dimensional (generated
by $\omega^{(3)}_1$ and $\omega^{(3)}_7$), we still have $\tau^{n-1}S(\omega^{(3)}_1,\omega^{(3)}_j)\in \CC \delta_{j,15}$
(resp. $\tau^{n-1}S(\omega^{(3)}_7,\omega^{(3)}_j)\in \CC \delta_{j,9}$) which would
imply that the conclusions of theorem \ref{theoS-solution} still hold, in particular,
that also for $k=5$ the above basis elements define a $(V^+,S)$-solution
and hence are all primitive and homogenous for it.
Notice also that if one formally calculates
$\frac1n(\deg(\omega^{(3)}_i)-\nu^{(3)}_i)_{i=1,\ldots,n}$ for the
above non-reductive examples, then the resulting numbers still have
the property of being symmetric around zero. This seem to indicate
that the conclusions of proposition \ref{propVerticalConnectionReductive} also
hold in the non-reductive case, although we cannot apply
theorem \ref{nbi} in this situation.

%As before, we observe again that the eigenvalues of the residue of $t\nabla_t$ are symmetric around zero,
%which confirms conjecture \ref{conjSpectrum} (ii). Moreover, the flat $\nabla^{res}$-flat section from lemma \ref{lemNablaTFlat}
%an eigenvector of $A^{(3)}_\infty$ for the (not necessarily unique) minimal spectral number.

\vspace*{0.5cm}

Now we turn to the series $D_m$ resp. $\star_m$.
%In Table \ref{tab:1} below, $D$ is the discriminant in the representation space of the
%quivers and dimension vectors shown here.
%\begin{displaymath}
%  \begin{array}{|c|c|c|c|}
%    \hline
%    A_n&D_n&E_6&\star_n\quad\quad\\
%    \hline
%    \xymatrix@R=0.1in@C=0.15in{1\ar[r]&1\ar[r]&\cdots\ar[r]&1}
%    &
%    \xymatrix@R=0.2in@C=0.15in{1\ar[r]&2\ar[r]&\cdots\ar[r]&2\ar[r]&1\\
%     &1\ar[u]}
%    &
%    \xymatrix@R=0.2in@C=0.15in{
%      1\ar[r]&2\ar[r]&3\ar[r]&2\ar[r]&1\\
%      &&2\ar[u]}
%    &
%    \overbrace{
%      \xymatrix@R=0.2in@C=0.15in{1\ar[drr]&1\ar[dr]&\cdots&1\ar[dl]&1\ar[dll]\\ &&n-1}}^n
%    \\
%    \hline
%  \end{array}
%\end{displaymath}
The results are given in table \ref{tab:1} below.
We write $(p_1,q_1),\ldots,(p_k,q_k)$ to indicate that the output of algorithm 1 resp. algorithm 2
is a basis $\underline{\omega}^{(2)}$ resp. $\underline{\omega}^{(3)}$ which decomposes into $k$ blocks as
in the proof of proposition \ref{ExchangeTrick}, where in each block $(p_i,q_i)$ the
eigenvalues of the residue endomorphism $\tau\partial_\tau$ along $\tau=0$ are
$-p_i,-p_i-1,\ldots,-p_i-q_i+1$. In particular, this gives the monodromy of $(G_t,\nabla)$
according to corollary \ref{monodromy}. We write moreover the
eigenvalues of the residue endomorphism of $t\partial_t$ on
$(G_0/tG_0)_{|\tau\neq\infty}$ as a tuple with multiplicities like
$[r_1]^{l_1},\ldots,[r_k]^{l_k}$ .
We observe that in all cases the symmetries
$\nu^{(2)}_i+\nu^{(2)}_{n+1-i}=n-1$ and
$\nu^{(3)}_i+\nu^{(3)}_{n+1-i}=n-1$ hold, and that the residue
eigenvalues of $t\partial_t$
on $(G_0/tG_0)_{|\tau\neq\infty}$ are symmetric around zero.

\small\renewcommand{\baselinestretch}{1.4}\normalsize
\begin{table}[h]
\begin{center}
$$
\begin{array}{c|c|c|c}
%\hline
 & D_m & \star_{m:=2k+1} & \star_{m:=2k} \\ \hline
\scriptstyle \dim(D)=n-1 & \scriptstyle 4m-11 & \multicolumn{2}{c}{\scriptstyle m^2-m-1}\\ \hline
&\scriptscriptstyle\left(\frac{4m-10}{3},m-3\right),&
\multicolumn{2}{c}{\scriptscriptstyle
%\overbrace{\scriptscriptstyle((m-2)(m-1),1),((m-2)(m-2),2),((m-2)(m-3)+1,3),((m-2)(m-4)+1+2,4),
%\ldots,((m-1-k)(m-2)+k(k-1)/2,k+1)}^{k \textup{ blocks}} }\\
\left((m-1-l)(m-2)+l(l-1)/2,l+1\right)_{\scriptscriptstyle l=0,\ldots,m-3},}\\
\scriptstyle Sp(G_0,\nabla)&\scriptscriptstyle (m-3,2m-4),& \multicolumn{2}{c}{\scriptscriptstyle((m-1)(m-2)/2,2(m-1)),}\\
&\scriptscriptstyle (\frac{5m-11}{3},m-3) &
\multicolumn{2}{c}{\scriptscriptstyle
%\overbrace{\scriptscriptstyle((m-2)(m-1),1),((m-2)(m-2),2),((m-2)(m-3)+1,3),((m-2)(m-4)+3,4),\ldots}^{?? \textup{ blocks}} }\\ \hline
\left(\frac12(m-l-1)(m+l),(m-l-2)\right)_{\scriptscriptstyle l=0,\ldots,m-3}}\\\hline
&&&\scriptscriptstyle  (mk-k,k-1), \\
\scriptstyle Sp(G_t,\nabla)& \scriptscriptstyle\left(\frac{4m-10}{3},m-3\right), &
\scriptscriptstyle \overbrace{\scriptscriptstyle (2k^2,m-2),(2k^2-1,m-2),\ldots,\left(2k^2-k+1,m-2\right)}^k,
& \scriptscriptstyle  \overbrace{\scriptscriptstyle  (mk-m,m-2),(mk-m-1,m-2),\dots,\left(mk-3k+2,m-2\right)}^{k-1},\\
	\scriptstyle t\neq 0&\scriptscriptstyle (m-3,2m-4),&\scriptscriptstyle \left(2k^2-k,2m-2\right),&\scriptscriptstyle  \left(2k^2-3k+1,2m-2\right),\\
&\scriptscriptstyle \left(\frac{5m-11}{3},m-3\right) & \scriptscriptstyle \overbrace{\scriptscriptstyle (2k^2+k,m-2),(2k^2+k-1,m-2),\dots,\left(2k^2+1,m-2\right)}^k &
\scriptscriptstyle \overbrace{\scriptscriptstyle (mk-k,m-2),(mk-k-1,m-2),\dots,\left(mk-m+2,m-2\right)}^{k-1},\\
&&&\scriptscriptstyle (mk-m+1,k-1)\\\hline
\scriptstyle \textup{Res}[t\partial_t] &
\scriptscriptstyle
\left[-\frac13\right]^{m-3}%\left(-\frac{4m-10}{3}\right)^{m-3}
, &
\multicolumn{2}{c}{\scriptscriptstyle
  \left(\frac{1}{m^2-m}[l-(m-1-l)(m-2)]^{l+1}\right)_{\scriptscriptstyle l=0,\ldots,m-3},}\\
\scriptstyle \textup{on } &\scriptscriptstyle 0^{2m-4},&
\multicolumn{2}{c}{\scriptscriptstyle 0^{2(m-1)},}\\
\scriptstyle (G_0/tG_0)_{|\tau\neq\infty}&\scriptscriptstyle \left[\frac13\right]^{m-3}
% \left(\frac{4m-10}{3}\right)^{m-3}
&\multicolumn{2}{c}{\scriptscriptstyle
  \left(\frac{1}{m^2-m}[(m-1)(l+1)]^{m-l-2}\right)_{\scriptscriptstyle l=0,\ldots,m-3}}\\
\end{array}
$$
\caption{Spectra of $f$ on the Milnor resp. zero fibre of the
  fibrations for $D_m$ and $\star_m$-series.\strut}
\label{tab:1}
\end{center}
\end{table}
\small\renewcommand{\baselinestretch}{1}\normalsize

\begin{remark}
  \begin{enumerate}
  \item
    We see that the jumping phenomenon (i.e., the fact that the spectrum
    of $(G_t,\nabla)$, $t\neq 0$ and $(G_0,\nabla)$ are different) occurs in our examples only
    for the star quiver for $m\geq 5$. However, there are probably many more examples
    where this happens, if the divisor $D$ has sufficiently high degree.
  \item Each Dynkin diagram supports many different quivers, distinguished by their edge
    orientations. Nevertheless, each of these quivers has the same set of roots.  For
    quivers of type $A_n$ and $D_n$, the discriminants in the corresponding representation
    spaces are also the same, up to isomorphism. However, for the quivers of type $E_6$,
    there are three non-isomorphic linear free divisors associated to the highest root
    (the dimension vector shown). Their generic hyperplane sections all have the same
    spectrum and monodromy.
  \item For the case of the star quiver with $n=2k$, the last and first blocks actually
    form a single block. We have split them into two to respect the order given by the
    weight of the corresponding elements in the Gau{\ss}-Manin system.
  \item In all the reductive examples presented above, the $\nabla^{res}$-flat basis element $t^{-k}\omega_i^{(3)}$
    from lemma \ref{lemNablaTFlat} was an eigenvector of $A^{(3)}_\infty$ for the smallest spectral number.
    An example where the latter does not hold is provided by the \emph{bracelet}, the discriminant in the space of
    binary cubics (the last example in 4.4 of \cite{gmns}). The spectrum of the generic
    hyperplane section is $(\frac{2}{3}, 1, 2, \frac{7}{3})$,
    and hence the minimal spectral number is not an integer. It is however unique, so that
    theorem \ref{theoS-solution} applies. On the other hand, we have a
    $\nabla^{res}$-flat section, namely $t^{-1}\omega^{(3)}_2$, but which does not coincide with
    the section corresponding to the smallest spectral number (i.e., the section $\omega^{(3)}_1$).
  \end{enumerate}
\end{remark}

Let us finish the paper by a few remarks on open questions and problems related to the results
obtained.

In \cite{DS2}, where similar questions for certain Laurent polynomials are studied, it is shown
that the $(V^+,S)$-solution constructed coincides in fact with the canonical solution as described
in theorem \ref{canSol} (see proposition 5.2 of loc.cit.). A natural question is to ask whether the same holds true
in our situation.

A second problem is to understand the degeneration
behaviour of the various Frobenius structures $M_t$ as discussed in theorem
\ref{theoLogFrob}, in particular in those cases
where we only have a weak logarithmic Frobenius manifold (i.e., all examples except the normal crossing case).
As already pointed out, a rather similar phenomenon occurs in \cite{Dou3}.

The constancy of the Frobenius structure at $t=0$ from theorem \ref{theoLimitFrob} is easy to understand
in the case of the normal crossing divisor: It corresponds to the
semi-classical limit in the quantum cohomology of $\PP^{n-1}$, which
is the Frobenius algebra given by the usual cup product and the Poincare
duality on $H^*(\PP^{n-1},\CC)$. One might speculate that for other
linear free divisors, the fact that the Frobenius structure at $t=0$ is constant
is related to the left-right stability of $f_{|D}$.

Another very interesting point is the relation of the Frobenius structures
constructed to the so-called $tt^*$-geometry (also known as variation of TERP-
resp. integrable twistor structures, see, e.g., \cite{He4}). We know from proposition \ref{propVerticalConnectionReductive}
(v) that the families studied here are examples of \emph{Sabbah orbits}. The degeneration
behaviour of such variations of integrable twistor structures has been studied
in \cite{HS1} using methods from \cite{Mo2}. However, the extensions over the boundary point $0\in T$ used in loc.cit.
are in general different from the lattices $G$ resp. $G^{(3)}$ considered here, as
the eigenvalues of the residue $[t\partial_t]$ computed above does not always
lie in a half-open interval of length one (i.e., $G_{|\CC^*\times T}$ is not
always a Deligne extension of $G_{|\CC^*\times (T\backslash\{0\})}$). One might want
to better understand what kind of information is exactly contained in the extension $G$.
Again, a similar problem is studied to some extend for Laurent polynomials in \cite{Dou3}.

%Its definition also involves the the
%cohomology of the Milnor fibre with real coefficients.  It is a generalization of the
%notion of pure resp. mixed Hodge structures, correspondingly, one can define pure and pure
%polarized TERP-structures. In the local analytic case, the variation on the unfolding
%space is pure polarized on (some) connected components of the complement of a real
%analytic subvariety. However, for cohomologically tame polynomial functions
%$f:Y\rightarrow \CC$, Sabbah (\cite{Sa8}) showed that this structure is always \emph{pure
%  polarized}.  This applies in particular to the situation studied here. One can similarly
%ask for the degeneration behaviour as $t$ tends to zero. These kind of degenerations has
%been studied in \cite{HS1} under the name Sabbah-orbits of TERP-structures.  The situation
%is very closely related to recent work of Mochizuki (\cite{Mo2}) who constructs a limit
%\emph{mixed} twistor structure on the fibre $t=0$ starting from such a pure polarized
%variation on a punctured disc. One might ask about the right kind of limit object starting
%from both a variation of pure polarized TERP-structures together with a family of
%$(V^+,S)$-solutions to the Birkhoff problem.

Finally, as we already remarked, the connection $\partial_\tau$ is regular singular at
$\tau=\infty$ on $\mathbf{G}_0$ but irregular for $t\neq 0$. Irregular connections are
characterized by a subtle set of topological data, the so-called Stokes matrices. It might
be interesting to calculate these matrices for the examples we studied, extending the
calculations done in \cite{Guzz} for the normal crossing case.

\bibliographystyle{amsalpha}
\providecommand{\bysame}{\leavevmode\hbox to3em{\hrulefill}\thinspace}
\providecommand{\MR}{\relax\ifhmode\unskip\space\fi MR }
\providecommand{\MRhref}[2]{%
  \href{http://www.ams.org/mathscinet-getitem?mr=#1}{#2}
}
\providecommand{\href}[2]{#2}

\bigskip
\noindent{\sc Ignacio de Gregorio}\quad{\tt degregorio@gmail.com}\\
\noindent{\it Mathematics Institute, University of Warwick, Coventry CV4 7AL, England.}

\bigskip
\noindent{\sc David Mond}\quad{\tt D.M.Q.Mond@warwick.ac.uk}\\
\noindent{\it Mathematics Institute, University of Warwick, Coventry CV4 7AL, England.}

\bigskip
\noindent{\sc Christian Sevenheck}\quad{\tt Christian.Sevenheck@uni-mannheim.de}\\
\noindent{\it Lehrstuhl VI f\"ur Mathematik, Universit\"at Mannheim, A6 5, 68131 Mannheim, Germany.}

\end{document}